\documentclass[11pt,a4paper]{amsart}
\usepackage{comment}
\usepackage{oldgerm}
\usepackage{amsmath}
\usepackage{amssymb}
\usepackage{amsthm}
\usepackage[svgnames]{xcolor}
\usepackage{slashed}
\usepackage{bm}
\usepackage{graphicx}

\usepackage[all]{xy}
\usepackage[foot]{amsaddr}
\usepackage[colorlinks,linktocpage=true,citecolor=blue,linkcolor=blue,urlcolor=blue]{hyperref}

\makeatletter
\@addtoreset{equation}{section}
\makeatother

\theoremstyle{definition}
\newtheorem{defn}[equation]{Definition}
\theoremstyle{plain}
\newtheorem{thm}[equation]{Theorem}

\newtheorem{prop}[equation]{Proposition}
\newtheorem{fact}[equation]{Fact}
\newtheorem{fact?}[equation]{Fact?}
\newtheorem{cor}[equation]{Corollary}
\newtheorem{lem}[equation]{Lemma}

\theoremstyle{remark}
\newtheorem{rem}[equation]{Remark}
\newtheorem{ex}[equation]{Example}

\newcommand{\Z}{\mathbb{Z}}

\newcommand{\R}{\mathbb{R}}
\newcommand{\C}{\mathbb{C}}

\newcommand{\bH}{\mathbb{H}}
\newcommand{\pt}{\mathrm{pt}}

\newcommand{\del}{\partial}

\newcommand{\Grad}{\nabla\!\!\!\!\nabla}

\newcommand{\Self}{\mathrm{Self}}
\newcommand{\Skew}{\mathrm{Skew}}
\newcommand{\End}{\mathrm{End}}
\newcommand{\Ori}{\mathrm{Ori}}
\newcommand{\type}{\mathrm{type}}
\newcommand{\lself}{\mathrm{self}}
\newcommand{\lskew}{\mathrm{skew}}
\newcommand{\Ph}{\mathrm{Ph}}
\newcommand{\CS}{\mathrm{CS}}
\newcommand{\id}{\mathrm{id}}
\newcommand{\Ch}{\mathrm{Ch}}
\newcommand{\Tw}{\mathfrak{Twist}}
\newcommand{\hatTw}{\mathfrak{Twist}^2_{\widehat{KO}_+}}
\newcommand{\dR}{\mathrm{dR}}

\begin{document}

\title[Differential $KO$-theory via gradations and mass terms]{Differential $KO$-theory via gradations and mass terms}

\author[K. Gomi]{Kiyonori Gomi}
\address{Department of Mathematics, 
Tokyo Institute of Technology, 
2-12-1 Ookayama, Meguro-ku, Tokyo, 152-8551, Japan}
\email{kgomi@math.titech.ac.jp}

\author[M. Yamashita]{Mayuko Yamashita}
\address{Research Institute for Mathematical Sciences, Kyoto University, 
606-8502, Kyoto, Japan}
\email{mayuko@kurims.kyoto-u.ac.jp}

\subjclass[]{}
\maketitle

\begin{abstract}
   We construct models of the differential $KO$-theory and the twisted differential $KO$-theory, by refining Karoubi's $KO$-theory \cite{KaroubiKtheory} in terms of gradations on Clifford modules. 
   In order for this, we set up the {\it generalized Clifford superconnection formalism} which generalizes the Quillen's superconnection formalism \cite{QuillenSuperconnection}. 
   One of our models can be regarded as classifying ``fermionic mass terms'' in physics. 
\end{abstract}

\renewcommand{\thefootnote}{\fnsymbol{footnote}} 
\renewcommand{\thefootnote}{\arabic{footnote}}

\tableofcontents

\section{Introduction}
In this paper, we construct models of the differential $KO$-theory and the twisted differential $KO$-theory, by refining Karoubi's $KO$-theory \cite{KaroubiKtheory} in terms of gradations on Clifford modules. 
In order for this, we set up the {\it generalized Clifford superconnection formalism} for module bundles over bundles of simple central graded algebras, which generalizes the Quillen's superconnection formalism \cite{QuillenSuperconnection} and is of independent interest. 
This work is intended to lay a foundation for understanding the theory on massive fermions in physics in terms of differential $KO$-theory. 
One of the models we construct in this paper can be regarded as classifying ``fermionic mass terms'' in physics.  

A {\it differential cohomology theory} $\widehat{E}$, or a {\it differential extension} of a gereralized cohomology theory $E$, is defined on manifolds and refines the original theory $E$ with a differential geometric data. 
Just as there can be various models for a cohomology theory $E$, there can be  various ways to realize differential refiments $\widehat{E}$. 
The most classical example is the ordinary cohomology theory $H\Z$, and differential refinements include the smooth Deligne cohomology \cite{Brylinski2008} and Cheeger-Simon's differential character groups \cite{CheegerSimonsDiffChar}. 
The case of $K$-theory has also been studied intensively, partly because of its physical applications. 
Nowadays a number of models of the differential $K$-theory are known: We can appeal to the general recipe as provided by Hopkins and Singer \cite{HopkinsSinger2005} originally. Ortiz \cite{Ortiz2009} introduced an analogous model based on the space of Fredholm operators.  There is also a geometric model based on vector bundles with connections, which is suggested in \cite{HopkinsSinger2005} and is realized by Freed and Lott \cite{FL2010}, for instance. The model given by Bunke and Schick \cite{BunkeSchicksmoothK} makes use of ``geometric cycles'', and the one given by Benameur and Maghfoul \cite{BenameurMaghfoul2006} generalizes differential characters to $K$-theory.

The differential $KO$-theory has less been studied, although its significance is suggested in particular in physics (for example see \cite{FreedDiracCharge}, \cite{FreedHopkinsRRfields} \cite{FMSHeisenberg} and \cite{FMS07}). 
Also the possibility of developing a model by a real analogue of ``geometric cycles'' is pointed out in \cite[Section 4.9]{BunkeSchickDiffKsurvey}, but has not been carried out, partly because of the lack of the theory of superconnections in the real settings. 
Recently, Grady and Sati gave a model of the differential $KO$-theory \cite{GradySatiDiffKO} and its twisted version \cite{GradySatiDiffTwistedKO} in terms of sheaves of spectra. 
Our paper is devoted to a construction of another type of models of the differential $KO$-theory and $K$-theory, as well as their twisted versions, by a differential refinement of Karoubi's models \cite{KaroubiKtheory} of $KO$ and $K$-theories. 
The motivation of developing such models comes from applications to physics, in particular its relation with ``fermionic mass terms'' as we explain in Subsection \ref{subsec_intro_phys}. 

In \cite{KaroubiKtheory}, the $KO$-theory is realized in terms of {\it gradations} on Clifford modules. 
For a finite CW-complex $X$ and a pair of nonnegative integers $(p, q)$, an element in the Karoubi's $KO$-theory group, which we denote by $KO_+^{p, q}(X)$, is represented by a triple $(S, h_0, h_1)$ of an ungraded module $S$ over the Clifford algebra $Cl_{p, q}$ with an inner product\footnote{
In this paper modules are assumed to be {\it ungraded} unless otherwise stated. 
An {\it inner product} on a Clifford module is always assumed to be compatible with the $*$-algebra structure on the Clifford algebras (see Subection \ref{subsec_algebras}). 
\label{footnote_inner_product}
} and self-adjoint invertible elements $h_0,h_1 \in \mathrm{Map}(X, \End(S))$ such that $c h_i =- h_i c$ for all odd elements $c \in Cl^1_{p, q}$. 
Such $h_i$ is called a {\it gradation} on $S$. 
We have a natural isomorphism $KO_+^{p, q} \simeq KO^{p-q}$ on the category of finite CW-complexes. 
An advantage of this model of $KO$-theory, particularly compared to the Atiyah-Singer's model \cite{AtiyahSingerSkew} in terms of skew-adjoint Fredholm operators, is that we only need to deal with finite-dimensional modules. 
The analytic issues do not arise, and this is one reason why the model suits differential refinements. 
Another advantage is that we can twist the model in a straightforward manner to have the twisted $KO$-theory of Donovan and Karoubi \cite{DonovanKaroubi}. 
Namely, given a bundle of simple central graded $\R$-algebras (for example Clifford algebras) $\mathcal{A} \to X$, the twisted $KO_+$-group $KO_+^{\mathcal{A}}(X)$ is constructed out of triples $(\slashed S, h_0, h_1)$, where $\slashed S$ is now an $\mathcal{A}$-module bundle and $h_0, h_1 \in \Gamma(X; \slashed S)$ are gradations. 
In this way, we get a model of twisted $KO$-theory with twists classified by $H^0(X; \Z_8) \times H^1(X; \Z_2) \times H^2(X; \Z_2)$. 
Also, by just replacing coefficient from $\R$ to $\C$, we get the models of the $K$-theory and its twistetd version, where twists are classified by $H^0(X; \Z_2) \times H^1(X; \Z_2) \times \mathrm{Tors}(H^3(X; \Z))$. 

In order to construct the differential refinement of Karoubi's $KO$-theory and the twisted variant above, we develop a certain generalization of Quillen's superconnection formalism \cite{QuillenSuperconnection}. 
This formalism, which we call the {\it generalized Clifford superconnection formalism} and is developed in Section \ref{sec_superconn}, should be of independent interest.  
The Quillen's formalism has been important in the analytic developments of the Atiyah-Singer's index theory (for example see \cite{BGVheatkernel}). 
The Quillen's formalism is for $\Z_2$-graded vector bundles, whereas our generalized Clifford superconnection formalism is for module bundles $\slashed S$ over bundles $\mathcal{A}$ of simple central graded algebras (over $\R$ or $\C$). 
Specializing to the case where $\mathcal{A}$ is the trivial bundle with fiber $Cl_{p, q}$, we get the superconnection formalism in Clifford-linear settings, and the possibility of the interpretation of ``mass terms'' in terms of this generalization is suggested in \cite[Section 7]{Cordova:2019jnf}. 

Given a $Cl_{p, q}$-module $S$, a smooth map $\xi \in C^\infty(X, \End(S))$ on a manifold $X$ which satisfies $\{c, \xi\} = 0$ for all $c\in Cl_{p, q}$ defines a $Cl_{p, q}$-superconnection
\begin{align*}
    d + \xi \colon \Omega^*(X; S) \to \Omega^*(X; S). 
\end{align*}
Then we can consider characteristic forms for this superconnection, such as the {\it Pontryagin character forms} and the {\it Chern-Simons forms}. 
A smooth gradation $h \in C^\infty(X, \End(S))$ is an example of a map satisfying the anticommutation relation above, and in particular invertible. 
The characteristic forms for gradations appearing in the definition of the differential extension $\widehat{KO}_+$ are constructed from the superconnection $d + th$ over the manifold $(0, \infty) \times X$, where $t \in (0, \infty)$. 

Our model of the differential $KO$-theory, which is denoted by $\widehat{KO}_+$, is defined in Section \ref{sec_diff_KO}. 
Elements of $\widehat{KO}_+^{p, q}(X)$ are of the form $[S, h_0, h_1, \eta]$, where $S$ is a $Cl_{p, q}$-module with an inner product, $h_0$ and $h_1$ are smooth gradations on $S$ and $\eta \in \Omega^{4\Z + p-q-1}(X)/\mathrm{Im}(d)$. 
We have an equality 
\begin{align*}
[S, h_0, h_1, \mathrm{CS}_\lself(h_I)] = 0
\end{align*} 
if $h_I$ is a homotopy between $h_0$ and $h_1$, where $\CS_\lself(h_I)$ is the {\it Chern-Simons form} for such a homotopy. 
There are structure homomorphisms $R$, $I$ and $a$ consistutng the data of differential refinements. In particular the {\it curvature homomorphism} $R$ is defined by using the {\it Pontryagin character form} as 
\begin{align*}
    R \colon \widehat{KO}^{p, q}_+(X) &\longrightarrow \Omega_{\mathrm{clo}}^{4\Z + p-q}(X), \\
    [S, h_0, h_1, \eta] &\mapsto \Ph_\lself(h_1) - \Ph_\lself(h_0) - d\eta. 
\end{align*}
A good point is that these characteristic forms are computable explicitely. 
The twisted models, $\widehat{KO}^{\mathcal{A}}_+(X)$, are constructed similarly. 
In our model, the isomorphism class of the twisted differential $KO$-theory group only depends on the class of $\mathcal{A}$ in $H^0(X; \Z_8) \times H^1(X; \Z_2) \times H^2(X; \Z_2)$ as in the topological case. 

In Section \ref{sec_skew} we develop a variant $\widehat{KO}_-$ of the differential model $\widehat{KO}_+$. 
This model is given in terms of skew-adjoint sections $m \in C^\infty(X; \End(\slashed S))$ which are invertible and $mc = (-1)^{|c|}cm$ for all homogeneous elements $c \in C^\infty(X; \mathcal{A})$, where $|c| \in \Z_2$ denotes the $\Z_2$-grading. 
We call such $m$ a {\it mass term} on $S$, suggesting that it models a fermionic mass term as we explain in Subsubsection \ref{subsubsec_intro_phys_interpretation}. 
On the topological level, replacing gradations to mass terms in the definition of $KO^{p, q}_+$, we get a functor $KO_-^{p, q}$. 
We have a natural isomorphism $KO_-^{p, q} \simeq KO^{q-p-2}$, reflecting the isomorphism of Clifford algebras $Cl_{p, q+1} \simeq Cl_{q, p+1}$ which does not preserve the $\Z_2$-gradings. 
Refining this topological theory $KO_-$ in a similar way, we get a differential extension $\widehat{KO}_-$. 
In the untwisted case, the elements in $\widehat{KO}_-^{p, q}(X)$ are of the form $[S, m_0, m_1, \eta]$ with $m_i$ smooth mass terms and $\eta \in \Omega^{4\Z + q-p-3}(X)/\mathrm{Im}(d)$. 

In Section \ref{sec_hat_K}, we explain that the $\C$-linear analogues of the above constructions give models $\widehat{K}_+$ and $\widehat{K}_-$ of the differential $K$-theory. 

\subsection{The physical motivations}\label{subsec_intro_phys}
\subsubsection{The interpretation of $\widehat{KO}_-$ and $\widehat{K}_-$ as the group of ``fermionic mass terms''}\label{subsubsec_intro_phys_interpretation}
Here, we explain the physical motivations mentioned above. 
Our models of differential extensions of the $KO$-theory and the $K$-theory, especially $\widehat{KO}_-$ and $\widehat{K}_-$ in terms of skew-adjoint operators, can be regarded as classifying ``fermionic mass terms''. 
It is known that fermionic mass terms on the $n$-dimensional Minkovski spacetime are classified topologically by $KO^{n-3} \simeq KO_{-}^{1, n}$ (see for example \cite[Section 10.2]{Freed19}), and our differential model $\widehat{KO}_-^{1, n}$ refines this classification on the differential level. 

First we explain the mathematical theory of fermionic mass terms, following \cite[Section 10.2]{Freed19} and \cite[Section 9.2]{Freed:2016rqq}\footnote{
Remark that the sign convention on the Clifford algebras used there is differerent from ours (see \eqref{eq_sign_C(V)}). 
In our convention, $Cl_{p, q}$ has the negative generators $\alpha_1, \cdots, \alpha_p$ and the positive generators $\beta_1, \cdots, \beta_q$. 
}. 
Let $n$ be the dimension of the spacetime. 
We start from an ungraded $Cl_{1, n-1}^0$-module $S$, without any specified inner product. 
Let $\mathrm{Spin}_{1, n-1} \subset Cl_{1, n-1}^0$ be the Lorentz spin group. 
Then there exists a $\mathrm{Spin}_{1, n-1}$-invariant symmetric nonnegative bilinear pairing
\begin{align}\label{eq_Gamma_pairing}
    \Gamma \colon S \otimes S \to \R^{1, n-1}, 
\end{align}
uniquely {\it up to a contractible choice}. 
Here the nonnegativity means that $\Gamma(s, s)$ is timelike for all $s \in S$. 
Such $\Gamma$ induces a unique compatible
$\Z_2$-graded $Cl_{1, n-1}$-module structure on $S \oplus S^*$, where the grading operator is given by $\gamma_{S \oplus S^*} := \mathrm{id}_{S} \oplus (-\mathrm{id}_{S^*})$.


In \cite{Freed19}, {\it mass forms} $m_{\mathrm{form}}$ on $S$ are defined as nondegenerate skew-symmetric $\mathrm{Spin}_{1, n-1}$-invariant bilinear forms
\begin{align}\label{eq_mass_form}
    m_{\mathrm{form}} \colon S \otimes S \to \R. 
\end{align}
Here we remark that such $m_{\mathrm{form}}$ is called ``mass terms'' in \cite{Freed19}. 
We use the above terminology and notation in order to distinguish it  from our definition of mass terms in terms of skew-adjoint operators. 
Then it is shown in \cite[Lemma 9.55]{Freed:2016rqq} that the existence of such $m_{\mathrm{form}}$ is equivalent to the existence of a $\Z_2$-graded $Cl_{2, n-1}$-module structure on $S \oplus S^*$ which extends the $Cl_{1, n-1}$-module structure above. 

We now explain how this formulation fits into our picture. 
Since the differential $KO$-groups should remember the differential, not just topological, information on mass terms, 
we do not want the ambiguity such as ``contractible choice'' above. 
Our model $\widehat{KO}_-$ is given in terms of Clifford modules with {\it inner products} (which are compatible with the Clifford action, see Footnote \ref{footnote_inner_product}), and skew-adjoint invariant operators on them. 
Suppose we have $S$ and a pairing $\Gamma$ as in \eqref{eq_Gamma_pairing}. 
Since the action on $S \oplus S^*$ by the negative Clifford generator $\alpha_1 \in Cl_{1, n-1}$ anticommutes with $\gamma_{S\oplus S^*}$, 
the restriction defines a linear isomorphism
\begin{align*}
    \alpha_1|_{S} \colon S \to S^*, 
\end{align*}
and defines a symmetric bilinear form $(\cdot, \cdot)_{S} \colon S \otimes S \to \R$ by
\begin{align}\label{eq_innerprod_S}
    (s_1, s_2)_{S} := \langle \alpha_1|_{S} (s_1), s_2 \rangle, 
\end{align}
where the right hand side is the duality pairing $\langle \cdot , \cdot \rangle \colon S^* \otimes S \to \R$. 
Using the positivity of $\Gamma$, we see that the form \eqref{eq_innerprod_S}, extended to $S \oplus S^*$ in the canonical way, defines an inner product on $S \oplus S^*$ which is compatible with the $\Z_2$-graded $Cl_{1, n-1}$-module structure induced by $\Gamma$. Now, we state a lemma essentially contained in the proof of \cite[Lemma 9.55]{Freed:2016rqq}  in the form we need, where we let $Cl_{1, (n-1)+1}$ act on $S \oplus S^*$ by the action of $Cl_{1, n-1}$ and $\gamma_{S \oplus S^*}$ noting that a $\Z_2$-graded $Cl_{1, n-1}$-module structure is equivalent to an ungraded $Cl_{1, (n-1)+1}$-module structure. 

\begin{lem}\label{lem_mass_form}
Let $S$ and $\Gamma$ be as above, and use the induced $Cl_{1, (n-1) + 1}$-module structure and inner product on $S \oplus S^*$. 
Then we have a bijection between the set of mass forms $m_{\mathrm{form}}$ on $S$ and the set
\begin{align*}
    &\Skew_{1, n}^*(S \oplus S^*) \\
    &:= \bigg\{m \in \End(S \oplus S^*)\ \bigg| \
    \begin{array}{l} 
    \mbox{invertible}, \
    m = -m^*, \\
    cm =-mc \mbox{ for all } c \in Cl^1_{1, (n-1) + 1}
    \end{array}
    \bigg\}.
\end{align*}
\end{lem}

The bijection is simply given as follows. 
Assume we are given an element $m\in \Skew_{1, n}^*(S \oplus S^*)$. 
Since $m$ anticommutes with $\gamma_{S\oplus S^*}$, the restriction of $m$ to $S$ is a linear isomorphism
\begin{align*}
    m|_{S} \colon S\to S^*. 
\end{align*}
We define the associated mass form $m_{\mathrm{form}} \colon S \times S \to \R$ by
\begin{align}
    m_{\mathrm{form}}(s_1, s_2) := \langle m|_{S}(s_1), s_2 \rangle, 
\end{align}
where $\langle \cdot , \cdot \rangle$ is the duality pairing. 
Then we can check that this assignment $m \mapsto m_{\mathrm{form}}$ gives the desired bijection. 
For details see the proof of \cite[Lemma 9.55]{Freed:2016rqq}. 

Elements of our differential model $\widehat{KO}^{1, n}_-(X)$ is represented by a quadruple $(S, m_0, m_1, \eta)$, where $S$ is a $Cl_{1, n}$-module with inner product, $m_0, m_1 \in C^\infty(X, \Skew_{1, n}^*(S))$, and $\eta \in \Omega^{4\Z + n}(X)/\mathrm{Im}d$ an additional data of a differential form.  Suppose that we are given a quadruple
\begin{align*}
    (S, \Gamma, m_{\mathrm{form}, 0}, m_{\mathrm{form}, 1}),
\end{align*}
where $S$ and $\Gamma$ are as above and $m_{\mathrm{form}, i} = \{m_{\mathrm{form}, i}(x)\}_{x \in X}$, $i  =0, 1$, are smooth families of mass forms on $S$ parametrized by $X$. Then, by Lemma \ref{lem_mass_form},
we get an element
\begin{align}\label{eq_intro_phys_KO_hat_element}
    [S \oplus S^*, m_0, m_1, 0] \in \widehat{KO}_-^{1, n}(X). 
\end{align}
In this way, our groups $\widehat{KO}^{1, n}_-(X)$ can be regarded as classifying pairs of smooth families of (nondegenerate) fermionic mass terms on $n$-dimensional Minkovski spacetime.\footnote{Typically in the physics literature, we often have a fixed constant mass term $m_*$. 
In such a case, we set $m_0 := m_*$ and regard $m_1$ as a single variable. \label{footnote_base_mass}}
Our model $\widehat{KO}^{1, n}$ is a differential extension of the topological $KO$-theory $KO_-^{1, n} \simeq KO^{n-3}$. 
On the topological level, the element \eqref{eq_intro_phys_KO_hat_element} corresponds to the element $[S \oplus S^*, m_0, m_1] \in KO_-^{1, n}(X) \simeq KO^{n-3}(X)$, which recovers the well-known topological classifications of mass terms by the $KO$-theory. 

\subsubsection{Further perspectives--differential pushforwards and the Anderson duality}\label{subsubsec_intro_ABS}
 
Now we explain further perspectives. 
We expect that the further development of our differential $KO$ and $K$-theories, in particular the theory of {\it differential pushforwards}, would lead to an understanding of the long-range effective theories of massive fermions in terms of the differential refinement of the Anderson dual to the Atiyah-Bott-Shapiro maps. 

A belief in the community of physicists is that deformation classes of invertible field theories should be classified by generalized cohomology theories. This idea is proposed in a lecture of Kitaev as reviewed in \cite{GaiottoFreyd2019}, and is further developed in \cite{freed2014shortrange} and \cite{Freed:2016rqq} from a mathematical viewpoint.
Moreover, it has also been noticed that differential cohomology theories give refined classifications of such theories. 

In the case of the theory on massive fermions, 
assume we have data of $S$, $\Gamma$ as in the last subsubsection, and fix a mass term $m_*$ (footnote \ref{footnote_base_mass}). 
By the process of the Wick rotation, we produce the corresponding theory on Euclidean Spin manifolds as follows.  
We consider the complexification $Cl_{1, n-1} \otimes_\R \C = \C l_{n}$, which has the Riemannian Clifford algebra $Cl_{0, n}$ as a subalgebra. 
Then the Riemannian Spin group $\mathrm{Spin}_n $ acts on the complexification $S^\C$ of $S$, and the $\Z_2$-graded $Cl_{1, n-1}$-module structure on $S \oplus S^*$ explained in Subsubsection \ref{subsubsec_intro_phys_interpretation} induces the $\Z_2$-graded $Cl_{0, n}$-module structure on $S_\C \oplus S_\C^*$. 
If we have a closed $n$-dimensional $\mathrm{Spin}_n$-manifold $X$ with a Spin-connection $\nabla$ (regarded as a ``Wick-rotated spacetime''), we form the associated bundle $\slashed S_\C$ to $S$, and the Dirac operator is given by
\begin{align}
    \slashed D = c \circ \nabla \colon C^\infty(X; \slashed S_\C \oplus \slashed S_\C^*) \to C^\infty(X; \slashed S_\C \oplus \slashed S_\C^*), 
\end{align}
where $c$ is the Clifford multiplication. 
Since $\slashed D$ anticommutes with $\gamma_{\slashed S_\C \oplus \slashed S_\C^*}$, it restricts to an operator from $\slashed S_\C$ to $\slashed S_\C^*$. 
Given a mass term $m = \{m(x)\}_{x \in X}$ parametrized by $X$, we get the {\it massive Dirac operator}, 
\begin{align}\label{eq_massive_dirac}
    \slashed D + m \colon C^\infty(X; \slashed S_\C ) \to C^\infty(X;  \slashed S_\C^*), 
\end{align}
which gives a formally skew-symmetric operator. 
Then the associated {\it Lagrangian density} on $X$ is
\begin{align}
    L(\psi, m) =\frac{1}{2} \left\langle(\slashed D +  m)\psi, \psi \right\rangle|dvol_{X}|
\end{align}
for $\psi \in C^\infty(X; \slashed S_\C)$. 
Physicists believe that the following expression makes sense and call the {\it partition function} for the massive fermions, 

\begin{align}\label{eq_part_func}
    \mathcal{Z}(m) = \frac{\int\mathcal{D} \psi\exp \left( -   \int_{X} L(\psi, m)\right)}{\int\mathcal{D} \psi\exp \left(- \int_X L(\psi, m_*)\right)},
\end{align}

which is formally equal to the quotient of the Pfaffians of the massive Dirac operators, 
\begin{align}\label{eq_pfaffian}
    \frac{\mathrm{Pf}(\slashed D +  m)}{\mathrm{Pf}(\slashed D + m_*)}. 
\end{align}
The nondegeneracy of the mass term $m$ implies that this theory is {\it gapped}, and the long-range limit is an {\it invertible} theory. 
The corresponding element $[S \oplus S^*, m_*, m, 0] \in \widehat{KO}_-^{1, n}(X)$ in \eqref{eq_intro_phys_KO_hat_element} in our model is regarded as classifying this invertible theory. 

Moreover, we expect that the theory of differential pushforwards in our model $\widehat{KO}_-$ should give a mathematical interpretation of the complex phase of the partition function \eqref{eq_part_func} \eqref{eq_pfaffian} in the long-range limit. 
For the differential $K$-theory, in the vector-bundle model by Freed-Lott \cite{FL2010} and the geometric-cycle model by Bunke-Schick \cite{BunkeSchicksmoothK}, the differential pushforward along the map $p_X \colon X \to \pt$ for $(2k-1)$-dimensional closed Spin$^c$ manifolds $X$ with Spin connections $\nabla$, 
\begin{align*}
    (p_X, \nabla)_* \colon \widehat{K}^0(X) \to \widehat{K}^{-(2k-1)}(\pt) \simeq \R/\Z, 
\end{align*}
is given by the reduced eta invariants of twisted Dirac operators. 
Indeed, in a simplest case where $m$ is constant, in the limit $m \to \infty$ the quantity \eqref{eq_pfaffian} is known to be given in terms of the eta invariants of the (massless) Dirac operator \cite{Witten:2019bou}. 
Abstractly, we have a canonical way to define pushforwards in multiplicative differential refinements of $KO$-theories for Spin-oriented proper submersions with connection \cite[Appendix]{Yamashita2021}. 
The problem is how to describe the pushforward explicitly; for example such a description is obtained in the model by Grady-Sati \cite{GradySatiDiffKO}. 
We expect that the differential pushfowards in $\widehat{KO}_-$ for $n$-dimensional Spin manifolds $X$, 
\begin{align}\label{eq_push_hat_KO_-}
    (p_X, \nabla)_* \colon \widehat{KO}_-^{1, n}(X) \simeq \widehat{KO}^{n-3}(X) \to \widehat{K}^{-3}(\pt) \simeq \R/\Z, 
\end{align}
would be described, and the complex phase of \eqref{eq_pfaffian} can be understood in terms of the image of the element $[S \oplus S^*, m_*, m, 0] \in \widehat{KO}_-^{1, n}(X)$. 
For example such a picture is compatible with the principle that, in invertible theories, the variation of partition functions under smooth variation of the geometric structures on manifolds is given by an integration of some locally-constructed differential forms.  
Indeed, it is a general feature, called the {\it Bordism formula} \cite[Problem 4.235]{Bunke2013} of pushforwards in differential cohomology theory, that the image under the differential pushforwards varies by integrations of appropriate characteristic forms. 

Finally we comment on the relation with the Anderson duality. 
Freed and Hopkins \cite[Conjecture 8.37]{Freed:2016rqq} conjectured that the deformation classes of Wick-rotated, fully extended reflection positive $n$-dimensional invertible theories on $G$-manifolds are classified by $(I\Omega^G)^{n+1}(\pt)$, where $I\Omega^G$ is a generalized cohomology theory called the {\it Anderson dual} to $G$-bordism theory. 
Motivated by this conjecture, in \cite{YamashitaYonekura2021} Yonekura and one of the authors of the present paper gave a model $\widehat{I\Omega^G_\dR}$ for a differential extension of $I\Omega^G$ with a physical interpretation. 
An element in $(\widehat{I\Omega^G_\dR})^{n+1}(X)$ is given by a pair $(\omega,h)$, where $h$ is a map which assigns an $\R/\Z$-value to an $n$-dimensional differential $G$-manifold with a map to $X$ and plays the role of partition functions, 
and $\omega$ is an element in $\Omega^*_{\mathrm{clo}}(X) \otimes_\R H^*(MTG; \R)$ of total degree $(n+1)$ which describes the variations of $h$. 
Moreover, it is shown in \cite{Yamashita2021} that, in terms of differential pushforwards, we can construct a transformation of differential cohomology theories which refines the Anderson dual of multiplicative genera. 

In the case of the theory on massive fermions, the conjecture by Freed-Hopkins \cite[Conjecture 9.70]{Freed:2016rqq} states that, on the topological level, the map which assigns the class of the data $(S, \Gamma, m_*, m)$ the deformation class of the corresponding invertible theory as above should coincide with the composition
\begin{align}\label{eq_IABS_top}
KO^{n-3}(\pt) \xrightarrow{\gamma_{KO}} IKO^{n+1}(\pt) \xrightarrow{I\mathrm{ABS}} (I\Omega^{\mathrm{Spin}})^{n+1}(\pt), 
\end{align}
where $\gamma_{KO} \colon KO^* \simeq IKO^{*+4}$ is the Anderson duality of $KO$ which shifts the degree by $4$, and $I\mathrm{ABS}$ is the Anderson dual to the Atiyah-Bott-Shapiro map $\mathrm{ABS} \colon MT\mathrm{Spin} \to KO$. 

Actually, the general theory in \cite{Yamashita2021} applied to this case provides the differential refinement of \eqref{eq_IABS_top}. 
The construction in \cite{Yamashita2021} gives the transformation
\begin{align}\label{eq_IABS_diff}
    \Phi_{\mathrm{ABS}}(- \otimes \gamma_{KO}) \colon \widehat{KO}^{n-3}(X) \to \widehat{I\Omega^{\mathrm{Spin}}}^{n+1}(X), 
\end{align}
which refines \eqref{eq_IABS_top}. 
For an element in $\widehat{KO}^{n-3}(X)$, the map \eqref{eq_IABS_top} assigns the element whose partition function is given in terms of the pushforwards of the pullbacks of this element in $\widehat{KO}$ to Spin manifolds with connections. 
This is indeed the expectation that we explained above for differential pushforwards \eqref{eq_push_hat_KO_-} in $\widehat{KO}_-$. 
In this way, the authors believe that the understanding of the differential pushfowards in $\widehat{KO}_-$, combined with the general results in \cite{YamashitaYonekura2021} and \cite{Yamashita2021}, would give a verification and a differential refinement of the statement in \cite[Conjecture 9.70]{Freed:2016rqq}. 
This deserves a promising future work.

\section{Preliminaries}
\subsection{Notations}
\begin{itemize}
\item The category of finite CW-complexes is denoted by $\mathrm{CW}_f$ and that of finite CW-pairs is denoted by $\mathrm{CWPair}_f$. 
\item By {\it manifolds}, we mean manifolds possibly with corners. 
The category of manifolds which have the homotopy types of finite CW-complexes is denoted by $\mathrm{Mfd}_f$. 
A {\it pair of manifolds} is a pair $(X, Y)$ consisting of a manifold $X$ and a submanifold $Y \subset X$ which is a closed subset. 
The category of pairs of manifolds which have the homotopy types of a finite CW-pairs is denoted by $\mathrm{MfdPair}_f$.
\item For a manifold $X$ and a real vector space $V$, we denote by $\underline{V}$ the trivial bundle $\underline{V} := X \times V$ over $X$. 
\item 
For a manifold $X$, its submanifold $Y$ and a local coefficient system $V$ realized as a real vector bundle on $X$, we denote
\begin{align*}
    \Omega^n(X, Y; V) :=\{\omega \in \Omega^n(X; V) \ | \ \omega|_Y = 0\}.  
\end{align*}
We denote $\Omega^*(X) := \Omega^*(X; \underline{\R}) = \Omega^*(X; \R)$. 
$\Omega_{\mathrm{clo}}^*$ denotes the space of closed forms.
The de Rham cohomology class of a closed form $\omega \in \Omega_{\mathrm{clo}}^*(X, Y; V)$ is denoted by $\mathrm{Rham}(\omega) \in H^*(X, Y; V)$. 
    \item For a homogeneous differential form $\omega \in \Omega^*(X)$, we denote its degree by $|\omega|$. 
    \item We let $I := [0, 1]$. 
\item A $\Z_2$-graded algebra is typically denoted by $A = A^0 \oplus A^1$, where $A^0$ denotes the even part and $A^1$ denotes the odd part. 
The degree of a homogeneous element $a \in A$ is denoted by $|a| \in \Z_2$. 
The supercommutator is denoted by $\{\cdot, \cdot\}$. 
\item For $\Z_2$-graded algebras $A$ and $B$, we denote their graded tensor product by $A \widehat{\otimes}B$. 

\end{itemize}
\subsection{Simple central graded \texorpdfstring{$*$}{*}-algebras}\label{subsec_algebras}

In \cite{DonovanKaroubi}, the gradings and twists in $KO$-theory and $K$-theory are given by simple central graded algebras and bundles of them. 
In this subsection we briefly review necessary parts of \cite{DonovanKaroubi}, and work a little more to adjust the theory to respect the $*$-algebra structures. 
In this subsection let the coefficient field $\mathbb{K}$ be either of $\R$ or $\C$.

Simple central graded algebras over $\R$ or $\C$ are classified by their {\it type} $t \in \Z_8$ in the case $\mathbb{K}= \R$ and $t \in \Z_2$ in the case $\mathbb{K}= \C$, and {\it size} $n \in \Z_{>0}$ or $(k, l) \in \Z_{>0} \times \Z_{>0}$. 
The classification results of the isomorphism classes of those algebras are given in Table \ref{table:gscR} and Table \ref{table:gscC}. 
Here, $I_n$ denotes the identity matrix of size $n$, and when $u$ is an element in an algebra $A$ such that $u^2 = \pm 1$, we write $Z(u) := \{a \in A \ | \ au=ua\}$ and $Z^*(u) := \{a \in A \ | \ au=-ua\}$.

The Clifford algebra $C(V, Q)$ for a real vector space with a nondegenerate quadratic form $(V, Q)$ is an important example, which is of type $\mathrm{sign}(Q)$. 
It is the unital $\R$-algebra generated over $V$ with relation
\begin{align}\label{eq_sign_C(V)}
    v \cdot v = -Q(v) \cdot 1. 
\end{align}
In particular we denote $Cl_{p, q} := C(\R^{p, q})$. 
It is generated by the anticommuting generators $\{\alpha_i, \beta_j\}_{i, j}$, $1 \le i \le p$, $1 \le j \le q$ with $\alpha_i^2 = -1$ and $\beta_j^2 = 1$. 
The type is $(p-q)$. 
The complexification is the complex Clifford algebra $\C l_{p+q} = Cl_{p, q} \otimes \C$, which is an example of a simple central graded algebra over $\C$ of type $(p+q)$. 

For any such algebra $A$, there exists a distinguished element $u \in A$, uniquely determined up to sign, which we call a {\it volume elemnent}. 
When $A$ is of odd-type, it is characterized by the property that $u \in A^1 \cap Z(A)$ and $u^2 = \pm 1$ in the real case and $u^2=1$ in the complex case, where $Z(A)$ denotes the center of $A$. 
When $A$ is of even-type, it is characterized by the property that $A^0 = Z(u)$, $A^1 = Z^*(u)$ and $u^2 = \pm 1$ in the real case and $u^2=1$ in the complex case.
In particular $u \in A^0$ if $A$ is of even type. 
We call such $u$ a volume element because it coincides with the Clifford volume element $\pm \alpha_1 \cdots \alpha_p \beta_1 \cdots \beta_q$ in the case $A = Cl_{p, q}$. 
We set $\mathrm{Ori}(A) := \R u \subset A$ in both the real and complex case. 

\begin{table}[h]
\[
\hspace{-1cm}
\begin{array}{c|c|c|c|c}
[\mathrm{type}; \mathrm{size}]  & A & A^0 & A^1 & \mbox{volume element } u
\\
\hline\hline
{[1; n]}& M(n, \C) & M(n, \R) & u\cdot M(n, \R) & \pm iI_n
\\
\hline
{[2; n]}&M(n, \bH) &M(n, \C) = Z(u) &Z^*(u) & \pm iI_n
\\
\hline
{[3; n]} &M(n, \bH) \oplus M(n, \bH) & M(n, \bH) &u \cdot M(n, \bH) &  \pm (I_n \oplus (-I_n))
\\
\hline
{[4; k, l]} &M(k+l,\bH)&Z(u) &Z^*(u) & \pm(I_k \oplus (-I_l))
\\
\hline
{[5; n]}&M(2n, \C)&M(n, \bH)&u \cdot M(n, \bH)& \pm i I_{2n}
\\
\hline
{[6; n]}&M(2n, \R)&Z(u)&Z^*(u) & 
\pm \left(
    \begin{array}{rr}
    0 & I_n \\
    -I_n & 0
    \end{array}
    \right)
\\
\hline
{[7; n]} &M(n, \R) \oplus M(n, \R)&M(n, \R)&u\cdot M(n, \R) & \pm (I_n \oplus (-I_n))
\\
\hline
{[8; k, l]} &M(k+l, \R)&Z(u)&Z^*(u) &\pm (I_k \oplus (-I_l))
\end{array}
\]
\caption{Simple central graded algebras over $\R$ \label{table:gscR}}
\end{table}

\begin{table}[h]
\[
\begin{array}{c|c|c|c|c}
(\mathrm{type}; \mathrm{size})  & A & A^0 & A^1 & u
\\
\hline\hline
{(1; n)}& M(n, \C) \oplus M(n, \C) & M(n, \C) & u\cdot M(n, \C) &  \pm (I_n \oplus (-I_n))
\\
\hline
{(2; k, l)}&M(k+l, \C)&Z(u)&Z^*(u) &\pm( I_k \oplus (-I_l))
\end{array}
\]
\caption{Simple central graded algebras over $\C$ \label{table:gscC}}
\end{table}

We introduce the $*$-algebra structure on the algebras $A$ in the Tables \ref{table:gscR} and \ref{table:gscC} by the composition of the complex or quarternion conjugation and the transpose of matrices. 
In this paper, by a {\it simple central graded $*$-algebra} we mean a simple central graded algebra equipped with a $*$-algebra structure which is isomorphic to one of those standard ones. 
In the case $A = Cl_{p, q}$, we have $\alpha_i = -\alpha_i^*$ for $1 \le i \le p$ and $\beta_j = \beta_j^*$ for $1 \le j \le q$. 

We say that $A$ is {\it degenerate} if $A = A^0$ and {\it nondegenerate} otherwise. 
The degenerate ones are $[t; n, 0] = [t; 0, n]$ for $t = 4, 8$ in Table \ref{table:gscR} and $(2; n, 0) = (2; 0, n)$ in Table \ref{table:gscC}. 

For two simple central graded $*$-algebras $A$ and $B$, their graded tensor product $A \widehat{\otimes} B$ is equipped with a canonical structure of a simple central graded $*$-algebra. 
For $A$ over $\R$, we denote $\Sigma^{0, 1}A :=  A\widehat{\otimes}Cl_{0, 1}$ and $\Sigma^{1, 0} A := A\widehat{\otimes}Cl_{1, 0}$. 
These notations are due to the fact that they correspond to the suspension isomorphisms in the $KO$-theory. 
For $A$ over $\C$, we define $\Sigma A := A \widehat{\otimes} \C l_{1}$. 
In the $\R$-linear case, if $\pm u \in A$ are volume elements of $A$, $\pm u\widehat{\otimes} \beta \in \Sigma^{0, 1}A$ are those for $\Sigma^{0, 1}A$, where $\beta \in Cl_{0, 1}$ denotes the (positive) Clifford generator. 
In the $\C$-linear case the volume elements are related by the rule
\begin{align}\label{eq_complex_volume}
    (\mbox{volume elements of }A) &\to (\mbox{volume elements of }\Sigma A) \\
   \pm u &\mapsto \pm (\sqrt{-1})^{\type(A)} u \widehat{\otimes} \beta. \notag
\end{align}

For a simple central graded $*$-algebra $A$, Let $\mathrm{Aut}(A)$ denote the group of automorphisms on $A$. 
We denote by $\mathrm{Aut}_{\Z_2}(A)$, $\mathrm{Aut}_{*}(A)$ and $\mathrm{Aut}_{*, \Z_2}(A)$ the subgroups of $\mathrm{Aut}(A)$ consisting of automorphisms preserving the $\Z_2$-grading, the $*$-algebra structure, and both the $\Z_2$-grading and the $*$-algebra structure, respectively. 
In this paper, by a {\it bundles of $A$'s} we mean a bundle of $\Z_2$-graded $*$-algebras with fiber $A$, which has the structure group $\mathrm{Aut}_{*, \Z_2}(A)$. 
In \cite{DonovanKaroubi} $*$-algebra structures on the fiber are not considered. 
But we lose or gain nothing topologically by this restriction by Lemma \ref{lem_star_aut_A} below. 
For such a bundle $\mathcal{A}$, we write $\Sigma^{0, 1}\mathcal{A} :=  \mathcal{A} \widehat{\otimes} Cl_{0, 1}$, $\Sigma^{1, 0}\mathcal{A} :=  \mathcal{A} \widehat{\otimes} Cl_{1, 0}$ and $\Sigma \mathcal{A} = \mathcal{A} \widehat{\otimes} \C l_{1}$. 

The automorphisms groups above are easy to describe explicitly. 
We explain the $\R$-linear case. 
The $\C$-linear case is similar. 
In the case of $M(n, \R)$ and $M(n, \mathbb{H})$, all the automorphisms are inner, so the automorphism group is $\mathrm{PGL}(n, \R) = \mathrm{GL}(n, \R) / \R^*$ and $\mathrm{GL}(n, \mathbb{H}) / \R^*$, respectively. 
In the case of $M(n, \mathbb{C})$, we also have the complex conjugation, and the inner automorphism group $\mathrm{GL}(n, \C) / \R^*$
is the index-two subgroup of the automorphism group. 
In the case of the direct sum of two copies of matrix algebras, the automorphism group is the obvious index-two subgroup of that of the matrix algebra of the double size. 

Using the descriptions above, we see that the $*$-preserving automorphism groups $\mathrm{Aut}_*(A)$ are simply given by replacing $\mathrm{GL}$ with $\mathrm{O}$ and $\R^*$ with $\Z_2$ in the above. 
For example the $*$-preserving automorphism group of $M(n, \R)$ is $\mathrm{PO}(n, \R) = \mathrm{O}(n, \R) / \Z_2$. 
In particular the inclusion $\mathrm{Aut}_*(A) \hookrightarrow \mathrm{Aut}(A)$ is a deformation retract. 

The $\Z_2$-grading-preserving automorphism groups $\mathrm{Aut}_{\Z_2}(A)$ are also easily described. 
The case of $[8; n, n]$ is explicitly given in \cite[p.9]{DonovanKaroubi}. 
For $\mathrm{Aut}_{*, \Z_2}(A) = \mathrm{Aut}_{\Z_2}(A) \cap \mathrm{Aut}_{*}(A)$, by a direct check we also have the following. 

\begin{lem}\label{lem_star_aut_A}
Let $\mathbb{K}$ be either of $\R$ or $\C$. 
For any simple central graded $*$-algebra $A$ over $\mathbb{K}$, 
The inclusion $\mathrm{Aut}_{*, \Z_2}(A) \hookrightarrow \mathrm{Aut}_{\Z_2}(A)$ is a deformation retract. 
\end{lem}

By the above description of $\mathrm{Aut}_{*, \Z_2}(A)$'s, we see that their unit components consist of inner automorphisms. 
From this, their Lie algebras are also easily understood. 
\begin{lem}\label{lem_Lie_aut_A}
Let $\mathbb{K}$ be either of $\R$ or $\C$. 
For any simple central graded $*$-algebra $A$ over $\mathbb{K}$, 
the Lie algebra of $\mathrm{Aut}_{*, \Z_2}(A)$ is given by
\begin{align*}
    \mathrm{Lie}(\mathrm{Aut}_{*, \Z_2}(A)) = A^0_{\mathrm{skew}}
    := \{a \in A^0 \ | \ a = -a^* \}, 
\end{align*}
with its natural Lie bracket. 
\end{lem}

We now explain characteristic classes for bundles of simple central graded $*$-algebras. 
We first consider the $\R$-linear case. 
Let $X$ be a CW-complex. 
Introduce an abelian group structure on $H^1(X; \Z_2) \times H^2(X;\Z_2)$ by $(a, b) + (a', b') := (a+a', b+b'+ a\cup a')$ and denote the resulting group by $HO(X)$. 
Given a bundle of simple central graded $*$-algebras $\mathcal{A}$ on $X$, we have the {\it characteristic class} $w(\mathcal{A}) = (w_1(\mathcal{A}), w_2(\mathcal{A})) \in HO(X)$ as defined in \cite{DonovanKaroubi}. 
We have
\begin{align}\label{eq_tensor_w}
    w(\mathcal{A}\widehat{\otimes} \mathcal{B})= w(\mathcal{A}) + w(\mathcal{B}). 
\end{align}
We define
\begin{align*}
    \mathrm{type}(\mathcal{A}) \in H^0(X; \Z_8)
\end{align*}
by assigning the types of the fibers. 
A bundle of simple central graded $*$-algebras is called {\it negligible} if it has the form $\mathrm{End}(E^0 \oplus E^1)$ for a graded vector bundle $E^0 \oplus E^1$ with a positive definite inner product on each $E^i$. 
Bundles of simple central graded $\R$-algebras over $X$, modulo negligible ones, form the {\it graded Brauer group} $\mathrm{GBrO}(X)$ under the graded tensor products. 
The characteristic class descends to this group. 
If $X$ is a finite CW-complex,
by \cite[Theorem 6]{DonovanKaroubi} we have an isomorphism of abelian groups, 
\begin{align}\label{eq_isom_GBrO}
    (\mathrm{type}, w) \colon \mathrm{GBrO}(X) \simeq H^0(X; \Z_8) \times HO(X). 
\end{align}
As we explain in Subsubsection \ref{subsubsec_Karoubi}, $\mathrm{GBrO}(X)$ classifies the twists in the $KO$-theory in the formulation of Donovan-Karoubi, which we denote by $KO_+$. 

In the $\C$-linear case, we consider an abelian group structure on the set $H^1(X; \Z_2) \times \mathrm{Tors}(H^3(X;\Z))$ by $(a, b) + (a', b') := (a+a', b+b'+ \beta(a\cup a'))$, where $\beta \colon H^2(X; \Z_2) \to H^3(X; \Z)$ is the Bockstein homomorphism. 
We denote the resulting group by $HU(X)$.
The {\it characteristic class} in this case is an element $u(\mathcal{A}) \in HU(X)$. 
We have a rule corresponding to \eqref{eq_tensor_w}. 
Bundles of simple central graded $\C$-algebras over $X$, modulo negligible ones, form the {\it graded Brauer group} $\mathrm{GBrU}(X)$ under the graded tensor products. 
If $X$ is a finite CW-complex, by \cite[Theorem 11]{DonovanKaroubi} we have 
\begin{align}\label{eq_isom_GBrU}
    (\mathrm{type}, u) \colon \mathrm{GBrU}(X) \simeq H^0(X; \Z_2) \times HU(X). 
\end{align}
and this classifies the twists in the $K$-theory $K_+$. 

An important class of such bundles are those associated with vector bundles with nondegenerate quadratic forms. 
Namely, if we have a real vector bundle $V$ over a CW-complex $X$ equipped with a fiberwise nondegenerate quadratic form $Q$, we denote by $C(V, Q)$, or simply $C(V)$, the associated Clifford algebra bundle. 
By \cite[Lemma 7]{DonovanKaroubi}, its characteristic class satisfies $w(C(V)) = (w_1(V), w_2(V))$, where the $w_1(V)$ and $w_2(V)$ are the Stiefel-Whitney classes of $V$. The bundle $C(V)$ is negligible if and only if $V$ admits a $\mathrm{Spin}(p, q)$-structure. 

In this paper, modules are assumed to be {\it ungraded}, unless otherwise stated. 
For a simple central graded $*$-algebra $A$ over $\mathbb{K}$, an {\it $A$-module with an inner product} is a (ungraded) vector space $S$ over $\mathbb{K}$ equipped with a positive definite inner product and a homomorphism $A \to \mathrm{End}(S)$ of $*$-algebras, where the $*$-structure on $\mathrm{End}(S)$ is given by the adjoint with respect to the inner product. 
Given a bundle of simple central graded $*$-algebras $\mathcal{A}$ over $X$, 
an {\it $\mathcal{A}$-module bundle with an inner product} is a real (ungraded) vector bundle $\slashed S$ over $X$ equipped with a fiberwise positive definite inner product and a fiberwise continuous $*$-preserving action of $\mathcal{A}$ realized by a homomorphism $\mathcal{A} \to \mathrm{End}(\slashed S)$ of $*$-algebra bundles.  

\begin{defn}[{$A$-endomorphisms}]\label{def_End_A}
Let $A$ be a simple central graded $*$-algebra.  
Let $S$ be an $A$-module with an inner product.
We define
\begin{align*}
    \mathrm{End}^0_{A}(S) &:= \{\xi \in\mathrm{End}(S) \ | \ \xi c = c \xi \mbox{ for all } c \in A \}, \\
    \mathrm{End}^1_{A}(S) &:= \{\xi \in\mathrm{End}(S) \ | \ \xi c =(-1)^{|c|} c \xi \mbox{ for all } c \in A \}. 
\end{align*}
If $A$ is nondegenerate, we have $\mathrm{End}^0_{A}(S) \cap \mathrm{End}^1_{A}(S) = \{0\}$. 
In this case, we define a $\Z_2$-graded algebra $\mathrm{End}_{A}(S) $ by
\begin{align*}
    \mathrm{End}_{A}(S) := \mathrm{End}^0_{A}(S)\oplus \mathrm{End}^1_{A}(S).  
\end{align*}
The important point to be noticed is that, even though $S$ is ungraded, the algebra $\mathrm{End}_{A}(S)$ is $\Z_2$-graded. 

If $A$ is degenerate, we set
\begin{align}
    \mathrm{End}_{A}(S) :=\mathrm{End}^0_{A}(S) =  \mathrm{End}^1_{A}(S).
\end{align}
\end{defn}

\begin{defn}\label{def_skew_self_module_A}
Let $A$ be a simple central graded $*$-algebra. 
Let $S$ be an $A$-module with an inner product.
\begin{itemize}
\item[(a)]
We define
$$
\Self_{A}(S)
:= \{ \xi \in \mathrm{End}^1_{A}(S) \ | \ \xi^* = \xi
\}.
$$
We also define the following subspaces 
\begin{align*}
\Self_{A}^*(S) &= \{ \xi \in \Self_{A}(S) |\ \xi \mbox{ is invertible} \}, \\
\Self_{A}^\dagger(S) &= \{ \xi \in \Self_{A}(S) |\ \xi^2 = 1 \}.
\end{align*}

\item[(b)]
We define 
$$
\Skew_{A}(S)
= \{ \xi \in \mathrm{End}^1_{A}(S) \ | \ \xi^* = - \xi
\}.
$$
We also define the following subspaces 
\begin{align*}
\Skew_{A}^*(S) &= \{ \xi \in \Skew_{A}(S) |\ \xi \mbox{ is invertible} \}, \\
\Skew_{A}^\dagger(S) &= \{ \xi \in \Skew_{A}(S) |\ \xi^2 = -1 \}.
\end{align*}

\end{itemize}
\end{defn}

\begin{defn}[{$\mathcal{A}$-endomorphism bundles}]\label{End_A_bundle}
Let $X$ be a CW complex and $\mathcal{A}$ be a bundle of simple central graded $*$-algebras over $X$. 
For an $\mathcal{A}$-module bundle with inner product $\slashed S$, we apply Definition \ref{def_End_A} fiberwise and define the bundles $\mathrm{End}^0_{\mathcal{A}}(\slashed S)$ and $\mathrm{End}^1_{\mathcal{A}}(\slashed S)$. 
If $\mathcal{A}$ is a bundle of nondegenerate algebras, we define a bundle of $\Z_2$-graded algebras over $X$ by
\begin{align*}
    \mathrm{End}_{\mathcal{A}}(\slashed S) := \mathrm{End}^0_{\mathcal{A}}(\slashed S) \oplus \mathrm{End}^1_{\mathcal{A}}(\slashed S).
\end{align*}
 In the degenerate case, we set
\begin{align*}
    \mathrm{End}_{\mathcal{A}}(\slashed S) := \mathrm{End}^0_{\mathcal{A}}(\slashed S) = \mathrm{End}^1_{\mathcal{A}}(\slashed S).
\end{align*}
Again note that, even though $\slashed S$ is ungraded, the bundle $\mathrm{End}_{\mathcal{A}}(\slashed S)$ is $\Z_2$-graded. 

We define the bundles $\Self_{\mathcal{A}}(\slashed S)$, $\Self^*_{\mathcal{A}}(\slashed S)$, $\Self^\dagger_{\mathcal{A}}(\slashed S)$, $\Skew_{\mathcal{A}}(\slashed S)$, $\Skew^*_{\mathcal{A}}(\slashed S)$ and $\Skew^\dagger_{\mathcal{A}}(\slashed S)$ by applying Definition \ref{def_skew_self_module_A} fiberwise. 
\end{defn}

We have the following result, generalizing the periodicities of Clifford modules: 
First remark that, for a $\Z_2$-graded real vector bundle $E = E^0 \oplus E^1$ with positive definite inner product on each $E^i$, the associated negligible bundle $\mathrm{End}(E)$ is equipped with a canonical choice of fiberwise volume element, namely the grading operator $\gamma_E := \mathrm{id}_{E^0} \oplus (-\mathrm{id}_{E^1})$.

\begin{lem}[{\cite[Theorem 16]{DonovanKaroubi}}]\label{lem_periodicity_negligible_bundle}
Let $X$ be a CW-complex. 
Let $E = E^0 \oplus E^1$ be a $\Z_2$-graded vector bundle over $\mathbb{K}$ on $X$ with positive definite inner product on each $E^i$. 
Let $\gamma_E:= \mathrm{id}_{E^0} \oplus (-\mathrm{id}_{E^1})$ be its grading operator. 
Let $\mathcal{A}$ be any bundle of simple central graded $*$-algebras over $X$. 
\begin{enumerate}
    \item For any $\mathcal{A}$-module bundle with inner product $\slashed S$, we introduce an $\mathrm{End}(E) \widehat{\otimes}\mathcal{A}$-module structure on $E \otimes \slashed S$ by
    \begin{align}\label{eq_periodicity_negligible_module}
    \mathrm{End}(E) \widehat{\otimes} \mathcal{A} &\to \mathrm{End}(E \otimes \slashed S) \\
    \xi \widehat{\otimes} 1 & \mapsto \xi \otimes \mathrm{id}_{\slashed S} \notag\\
    1 \widehat{\otimes} b & \mapsto (\gamma_E)^{|b|} \otimes b. \notag
\end{align}
    The assignment $\slashed S \mapsto E \otimes \slashed S$ gives a bijection between the isomorphism classes of $\mathcal{A}$-module bundles with inner products and those of $\mathrm{End}(E) \widehat{\otimes}\mathcal{A}$-module bundles with inner products. 
\item When $\mathcal{A}$ is a bundle of nondegenerate algebras, we have an isomorphism of bundles of $\Z_2$-graded algebras,
\begin{align}
    \psi^E \colon \mathrm{End}_{\mathcal{A}}(\slashed S)  &\simeq \mathrm{End}_{\mathrm{End}(E) \widehat{\otimes} \mathcal{A}}(E \otimes \slashed S), \\
    \xi &\mapsto (\gamma_E)^{|\xi|} \otimes \xi.\notag
\end{align}
This isomorphism restricts to ones on $\Self$, $\Self^*$, $\Self^\dagger$, $\Skew$, $\Skew^*$ and $\Skew^\dagger$. 
\item When $\mathcal{A}$ is a bundle of degenerate algebras, we have isomorphisms of vector bundles, 
\begin{align}\label{eq_End_isom_degenerate}
    \psi^E \colon \End_{\mathcal{A}}(\slashed S) = \End^1_{\mathcal{A}}(\slashed S)  &\simeq \End^1_{\mathrm{End}(E) \widehat{\otimes} \mathcal{A}}(E \otimes \slashed S), \\
    \xi &\mapsto \gamma_E \otimes \xi.\notag
\end{align}
This isomorphism restricts to ones on $\Self$. $\Self^*$, $\Self^\dagger$, $\Skew$, $\Skew^*$ and $\Skew^\dagger$. 
\end{enumerate}
\end{lem}

Lemma \ref{lem_periodicity_negligible_bundle} applied to $X = \pt$ recovers the $(1, 1)$-periodicity of modules of Clifford algebras over $\R$ by setting $E = \R \oplus \R$, the $(8, 0)$-periodicity by setting $E = \R^8 \oplus \R^8 $, and the  $2$-periodicity over $\C$ by $E = \C \oplus \C$.

\subsection{Karoubi's \texorpdfstring{$KO$}{KO}-theory}

In this subsection we review the formulation of topological $KO$-theory given by Karoubi \cite{KaroubiKtheory} and Donovan-Karoubi \cite{DonovanKaroubi}, in terms of {\it gradations}. 

First we fix our notations. 
The $KO$-groups of a point is presented as
\begin{align}\label{eq_KO_pt}
    KO^*(\pt) = \Z[\eta, v, u, u^{-1}] / (2\eta, \eta^3, \eta v, v^2-4u)
\end{align}
as a $\Z$-graded algebra, where $\eta \in KO^{-1}(\pt)$, $v \in KO^{-4}(\pt)$ and $u \in KO^{-8}(\pt)$. 
This implies that
\begin{align}\label{eq_KO_R_pt}
    KO^*(\pt)\otimes \R = \R[v, v^{-1}]. 
\end{align}
The $K$-groups of a point is presented as 
\begin{align}\label{eq_K_pt}
    K^*(\pt) = \Z[b, b^{-1}]
\end{align}
as a $\Z$-graded algebra, where $b \in K^{-2}(\pt)$. 
We have
\begin{align}\label{eq_K_R_pt}
    K^*(\pt)\otimes \R  = \R[b, b^{-1}]. 
\end{align}
We have the complexification map $KO \to K$. 
We choose the sign of $v$ so that $v$ maps to $2u^2$ under the complexification. 

In the rest of this subsection, we work over $\R$ unless otherwise stated (e.g., in Remark \ref{rem_Karoubi_K}).

\subsubsection{Karoubi's \texorpdfstring{$KO$}{KO}-theory}\label{subsubsec_Karoubi}

First we recall the definition of untwisted $KO$-groups. 
\begin{defn}\label{def_KO_karoubi_untwisted_triple}
Let $(X, Y)$ be a finite CW-pair and $A$ be simple central graded $*$-algebra. 
\begin{itemize}
\item
A {\it triple} $( S, h_0, h_1)$ on $(X, Y)$ consists of an $A$-module $ S$ with an inner product and two continuous maps $h_0, h_1 \in \mathrm{Map} (X,  \Self_{A}^*(S))$ with $h_0|_Y = h_1|_Y$.  
Such $h_i$'s are called {\it gradations} on $S$. 
\item 
Triples $( S, h_0, h_1)$ and $( S', h'_0, h'_1)$ on $(X, Y)$ are {\it isomorphic} if there exists an isometric isomorphism of $A$-modules $f \colon  S \simeq  S'$ such that $f \circ h_i = h'_i\circ f$ for $i = 0, 1$. 
\end{itemize}
\end{defn}

\begin{defn}[{$KO_+^{A}(X, Y)$, \cite[Chapter III, Proposition 4.26]{KaroubiKtheory}}]\label{def_karoubi_untwisted_KO}
Let $(X, Y)$ and $A$ as above. 
\begin{itemize}
\item
On the set $M^{A}_+(X, Y)$ of isomorphism classes of triples $( S, h_0, h_1)$ on $(X, Y)$, we introduce an abelian monoid structure by the direct sum. 

\item
We define $Z^{A}_+(X, Y)$ to be the submonoid of $M^{p, q}_+(X, Y)$ consisting of isomorphism classes of triples $( S, h_0, h_1)$ on $(X, Y)$ such that there exists a homotopy between $h_0$ and $h_1$ which is constant on $Y$, i.e., a map $h_I \in \mathrm{Map}(I \times X, \Self^*_{A}( S))$ with $h_I|_{\{i\} \times X} = h_i$ for $i = 0, 1$ and $h_I|_{\{t\} \times Y} = h_0|_Y$ for all $t \in I$. 

\item
We define $KO^{A}_+(X, Y) = M^{A}_+(X, Y)/Z^{A}_+(X, Y)$ to be the quotient monoid.

\end{itemize}
\end{defn}

\begin{rem}
Actually we are using a slightly different formulation from \cite{KaroubiKtheory}, but they are equivalent. 
In \cite{KaroubiKtheory}, we do not consider inner products on modules $S$, and instead of $\Self_*^A(S)$ above we use
\begin{align}
    \mathrm{Grad}_A(S) := \{\xi \in \End^1_A(S) \ | \ \xi^2 = 1\}. 
\end{align}
Elements in $\mathrm{Grad}_A(S)$, or maps to it, are called gradations there. 
If $S$ is equipped with an inner product, we have inclusions $\Self^\dagger_A(S) \hookrightarrow \mathrm{Grad}_A(S)$ and $\Self^\dagger_A(S) \hookrightarrow \Self^*_A(S)$. 
Since both of them are homotopy equivalences, our definition is equivalent to the formulation in \cite{KaroubiKtheory}. 
We choose this formulation because we need self-adjointness for the differential refinement below. 
By the same reason, we may as well use $\Self^\dagger_A(S)$ instead of $\Self_A^*(S)$ in the above definition. 
Similar remarks also apply to the twisted case. 
\end{rem}

As is shown in \cite{KaroubiKtheory}, the abelian monoid $KO^{A}_+(X, Y)$ gives rise to an abelian group, where the additive inverse given by $-[ S, h_0, h_1] = [ S, h_1, h_0]$. 
$KO_+^A$'s are models for the $KO$-theory, as follows. 
See Subsubsection \ref{subsubsec_KO_spectrum} for further details. 

\begin{fact}[{\cite[Chapter III, Theorem 4.29]{KaroubiKtheory}}]\label{fact_karoubi_KO}
We have a natural isomorphism on the category of finite CW-pairs, 
\begin{align}
    KO^{A}_+ \simeq KO^{\type(A)}. 
\end{align}
\end{fact}

By a slight generalization of the above constructions, we also get a model for the twisted $KO$-theories. 
The cases of twists given by Clifford module bundles of the form $C(V, Q)$ are given in \cite[Chapter III, Section 4]{KaroubiKtheory}. 
The case of twists given by bundles of simple central graded $\R$-algebras in general are given in \cite{DonovanKaroubi}.  
We explain twists in $KO$-theory in detail in Subsection \ref{subsec_twist} below. 

Following \cite[Section 7]{BunkeSchickDiffKsurvey}, in this paper we introduce the categories of twists for $KO_+$ and for its differential refinement. 
Here we introduce the category $\mathfrak{Twist}^2_{KO_+}$ of twists on $KO_+$\footnote{
The ``$2$'' indecates that we are dealing with CW-pairs. 
}.

\begin{defn}[{$\Tw^2_{KO_+}$}]\label{def_twist_cat_KO_+}
\begin{enumerate}
    \item For each finite CW-complex $X$ we define $\Tw_{KO_+, X}$ to be the groupoid of bundles of simple central graded $*$-algebras $\mathcal{A}$ over $X$, where morphisms are ismorphisms between such bundles. 
    \item For each morphism $f \colon X \to X'$ in $\mathrm{CW}_f$, we define a functor $f^* \colon \Tw_{KO_+, X'} \to \Tw_{KO_+, X}$ by the pullback. 
    \item We define $\Tw^2_{KO_+}$ to be the category such that an object $(X, Y, \mathcal{A})$ consists of $(X, Y) \in \mathrm{CWPair}_f$ and $\mathcal{A} \in \Tw_{KO_+, X}$, and a morphism from $(X, Y, \mathcal{A})$ to $(X', Y', \mathcal{A}')$ consists of a morphism $f \colon (X, Y) \to (X', Y')$ in $\mathrm{CWPair}_f$ and an isomorphism $\mathcal{A} \simeq f^*\mathcal{A}'$. 
\end{enumerate}
\end{defn}

\begin{defn}
Let $(X, Y, \mathcal{A}) \in \Tw^2_{KO_+}$. 
\begin{itemize}
\item
A {\it triple} $(\slashed S, h_0, h_1)$ on $(X, Y, \mathcal{A})$ consists of an $\mathcal{A}$-module bundle $\slashed S$ with an inner product and two continuous sections $h_0, h_1 \in \Gamma (X; \Self_\mathcal{A}^*(\slashed S))$ with $h_0|_Y = h_1|_Y$.  
\item 
Triples $(\slashed S, h_0, h_1)$ and $(\slashed S', h'_0, h'_1)$ on $(X, Y, \mathcal{A})$ are {\it isomorphic} if there exists an isometric isomorphism of $\mathcal{A}$-module bundles $f \colon \slashed S \simeq \slashed S'$ such that $f \circ h_i = h'_i\circ f$ for $i = 0, 1$. 
\end{itemize}
\end{defn}

\begin{defn}[{$KO_+^\mathcal{A}(X, Y)$}]\label{def_karoubi_twisted_KO}
Let $(X, Y, \mathcal{A}) \in \Tw^2_{KO_+}$. 
\begin{itemize}
\item
On the set $M^{\mathcal{A}}_+(X, Y)$ of isomorphism classes of triples $(\slashed S, h_0, h_1)$ on $(X, Y, \mathcal{A})$, we introduce an abelian monoid structure by the direct sum. 

\item
We define $Z^{\mathcal{A}}_+(X, Y)$ to be the submonoid of $M^{\mathcal{A}}_+(X, Y)$ consisting of isomorphism classes of triples $(\slashed S, h_0, h_1)$ on $(X, Y)$ such that there exists a homotopy between $h_0$ and $h_1$ which is constant on $Y$, i.e., a continuous section $h_I \in \Gamma(I \times X; \mathrm{pr}_X^*\Self^*_\mathcal{A}(\slashed S))$ with $h_I|_{\{i\} \times X} = h_i$ for $i = 0, 1$ and $h_I|_{\{t\} \times Y} = h_0|_Y$ for all $t \in I$. 

\item
We define $KO^{\mathcal{A}}_+(X, Y) = M^{\mathcal{A}}_+(X, Y)/Z^{\mathcal{A}}_+(X, Y)$ to be the quotient monoid.

\end{itemize}
\end{defn}

By the functoriality of the above definition, we get the functor
\begin{align}\label{eq_twisted_KO_+_functor}
    KO_+ \colon \Tw_{KO_+}^2 \to \mathrm{Ab}, 
\end{align}
where $\mathrm{Ab}$ is the category of abelian groups.


Let us consider the case where $\mathcal{A} = \underline{A}$ is the product bundle associated to an algebra $A$. In this case, the group $KO^{\underline{A}}_+(X, Y)$ differs from $KO^{A}_+(X, Y)$ on the nose, because the former also considers nontrivial $A$-module bundles. 
However, by \cite[Chapter III, Proposition 4.26]{KaroubiKtheory}, any element in the former group can be represented by a triple in Definition \ref{def_KO_karoubi_untwisted_triple}, and
we actually have a natural isomorphism
\begin{align}\label{eq_isom_karoubi_untwisted}
    KO^{A}_+(X, Y) \simeq KO^{\underline{A}}_+(X, Y). 
\end{align}

The suspension isomoprhism in this model is given as follows. 
\begin{fact}[{\cite[Section 6]{DonovanKaroubi}}]\label{fact_susp_twisted_+}
For any object $(X, Y, \mathcal{A}) \in \Tw^2_{KO_+}$,
the suspension isomorphism
$$
KO^{\Sigma^{0, 1}\mathcal{A}}_+(X, Y) \simeq
KO^{\mathcal{A}}_+(X \times I, X \times \partial I \cup Y \times I),
$$
is given by the assignment $[\slashed S, h_0, h_1] \mapsto [\slashed S, \widetilde{h}_0, \widetilde{h}_1]$. Here the $\Sigma^{0, 1}\mathcal{A}$-module bundle $\slashed S$ is regarded as an $\mathcal{A}$-module bundle in the latter triple, and $\widetilde{h}_i : X \times I \to \mathrm{Self}_{\mathcal{A}}^*(S)$ is given by
\begin{align}\label{eq_prop_periodicity_twisted}
    \widetilde{h}_i(x, t) = 
\beta \cos \pi \theta + h_i(x) \sin \pi \theta, 
\end{align}
where $\beta$ denotes the action of the generator of $Cl_{0, 1}$. 
\end{fact}

Actually the twists of $KO_+$ are classified by $\mathrm{GBrO}(X)$. 
To show this, we first prove that twists given by negligible bundles are trivial, as follows. 
\begin{fact}[{\cite[Theorem 16]{DonovanKaroubi}}]\label{fact_isom_negligible}
Let $(X, Y) \in \mathrm{CWPair}_f$. 
Let $E = E^0 \oplus E^1$ be a $\Z_2$-graded vector bundle over $X$ with fiberwise positive definite inner product.
For any bundle of simple central graded $*$-algebras $\mathcal{A}$, we have a canonical isomorphism
\begin{align}\label{eq_isom_negligible}
    KO^{\mathcal{A}}_+(X, Y) \simeq KO^{\mathrm{End}(E) \widehat{\otimes} \mathcal{A}}_+(X, Y). 
\end{align}
\end{fact}

\begin{proof}
By Lemma \ref{lem_periodicity_negligible_bundle}, we have a canonical bijective correspondence of $\mathcal{A}$-module bundles and $\mathrm{End}(E) \widehat{\otimes} \mathcal{A}$-module bundles, and $\Self^*$ on them. 
Using the notation there, the isomorphism \eqref{eq_isom_negligible} is explicitely given by
\begin{align}\label{eq_isom_negligible_map}
    [\slashed S, h_0, h_1] \mapsto [E \otimes \slashed S, \psi^E(h_0), \psi^E(h_1)]. 
\end{align}
\end{proof}

By Fact \ref{fact_isom_negligible}, we see the following. 

\begin{cor}\label{cor_KO_twist_classification}
For any object $(X, Y, \mathcal{A}) \in \Tw^2_{KO_+}$, the isomorphism class of the group $KO_+^{\mathcal{A}}(X, Y)$ only depends on the class of $\mathcal{A}$ in $\mathrm{GBrO}(X)$. 
\end{cor}
We say more about the twists in Subsection \ref{subsec_twist}. 

\begin{rem}\label{rem_Karoubi_K}
In the $\C$-linear settings, 
by only replacing the coefficients from $\R$ to $\C$, the same definition as Definition \ref{def_karoubi_untwisted_KO} constructs a functor $K^A_+$. 
We have the natural isomorphism on the category of finite CW-pairs, 
\begin{align}\label{eq_rem_karoubi_K}
    K_+^A \simeq K^{\type(A)}. 
\end{align}
Also, by the same replacement of Definitions \ref{def_twist_cat_KO_+} and \ref{def_karoubi_twisted_KO} we define the category of twists $\Tw^2_{K_+}$ and the twisted $K$-theory as a functor 
\[
K_+ \colon \Tw^2_{K_+} \to \mathrm{Ab}.
\]
\end{rem}

\subsubsection{$KO$-spectrum in terms of $\Self^*$}\label{subsubsec_KO_spectrum}
As explained in \cite[Chapter III, 4.24--4.30]{KaroubiKtheory}, the formulation of $KO$-theory in Subsubsections \ref{subsubsec_Karoubi} leads to a corresponding model of the $KO$-spectrum. 
In this subsubsection we review this construction.

Let $A$ be a simple central graded $*$-algebra. 
Given a $\Sigma^{0, 1} A$-module $S$ with an inner product, let $\beta$ denote the action of the generator of $Cl_{0, 1} \subset \Sigma^{0, 1} A$. 
We regard $S$ as an $A$-module by the inclusion $A \subset \Sigma^{0, 1} A$, and consider the space $\Self_{A}^*(S)$. 
Regarding $\beta \in \Self_{A}^*(S)$ as the basepoint of the space $\Self_{A}^*(S)$, we consider the following pointed space
\begin{align*}
    (\Self_{A}^*(S), \{\beta\}). 
\end{align*}
A sequence $\{S_n\}_n$ of $\Sigma^{0, 1} A$-modules are called {\it cofinal system} if we have a fixed isomorphism $S_n \oplus S_1 \simeq S_{n+1}$ for each $n$ and every $\Sigma^{0, 1} A$-module is a direct factor of some $S_n$. 
For such a sequence we get the sequence of pointed spaces, 
\begin{align}\label{eq_Skew_inclusion}
    \cdots \hookrightarrow (\Self_{A}^*(S_n), \{\beta\}) &\hookrightarrow (\Self_{A}^*(S_{n+1}), \{\beta\}) \hookrightarrow \cdots \\
    m &\mapsto m \oplus \beta|_{S_1}. \notag
\end{align}
The direct limit of \eqref{eq_Skew_inclusion} gives a classifying space of $KO_+^{A}$, as follows. 

\begin{fact}[{\cite[Theorem 4.27]{KaroubiKtheory}}]\label{fact_mass_classifying}
Let $\{S_n\}_n$ be any cofinal system of $\Sigma^{0, 1} A$-modules. 
Then we have a weak homotopy equivalence, 
\begin{align}\label{eq_fact_mass_classifying_spectrum}
    (KO_{\mathrm{type}(A)}, \{*\}) \sim \varinjlim_n (\Self_{A}^*(S_n), \{\beta\}). 
\end{align}
Upon a choice of a volume element $u$ in $A$, there is a preferred choice of the weak homotopy equivalence \eqref{eq_fact_mass_classifying_spectrum} up to homotopy. 
In particular if $A = Cl_{p, q}$, then it is of type $(p-q)$, so that we have
\begin{align*}
    (KO_{p-q}, \{*\}) \sim \varinjlim_n (\Self_{p,q}^*(S_n), \{\beta\}). 
\end{align*}
Moreover, for any finite CW-pair $(X, Y)$, we have a natural isomorphism
\begin{align}\label{eq_fact_mass_classifying}
    KO^{A}_+(X, Y) \simeq [(X, Y), \varinjlim_n (\Self_{A}^*(S_n), \{\beta\})]. 
\end{align}
\end{fact}
The dependence of the weak homotopy equivalence \eqref{eq_fact_mass_classifying_spectrum} on a volume element is implicit in \cite{KaroubiKtheory}, so we explain it below in the description of the homeomorphism \eqref{eq_homeo_Gr_Skew}. 
Modifying the construction in Fact \ref{fact_susp_twisted_+} (as carried out in \cite[III. 6]{KaroubiKtheory}), we can get the structure map for the above model of the $KO$-spectrum, making it into an $\Omega$-spectrum. 

Here we explain the isomorphism \eqref{eq_fact_mass_classifying}. 
Given an $A$-module $S$, we can embed $S \hookrightarrow S_n$ as a direct summand for some $n$ as an $A$-module. 
Then we get the corresponding embedding $\Self^*_{A}(S) \hookrightarrow \Self^*_{A}(S_n)$ (as an unbased space) by using $\beta$ on the orthogonal complement. 
The resulting homotopy class $[\Self^*_{A}(S) ,\varinjlim_n (\Self_{A}^*(S_n))]$ does not depend on the choice. 
Fixing such an embedding, given a triple $(S, h_0, h_1)$ which represents a class $[S, h_0, h_1]\in KO^{A}_+(X, Y)$, we get the homotopy classes $[h_i] \in [X, \varinjlim_n (\Self_{A}^*(S_n))]$ for $i = 0, 1$. The condition $h_0|_Y = h_1|_Y$ implies that the difference class can be regarded as an element
\begin{align}\label{eq_difference_class}
    [h_1] - [h_0] \in [(X, Y), \varinjlim_n (\Self_{A}^*(S_n), \{\beta\})]
\end{align}
which is independent of the choice of the embedding. 
The isomorphism \eqref{eq_fact_mass_classifying} is given by this procedure. 

Finally we explain the relation between the formulation in Subsubsection \ref{subsubsec_Karoubi} with the formulation of $KO^0$ in terms of real vector bundles. 
The corresponding relations in the other degrees can be found in \cite[Chapter III, Section 4.28--4.30]{KaroubiKtheory}. 

For the zeroth space of $KO$-spectrum, the following weak homotopy equivalences are well-known
\begin{align}\label{eq_KO_0_Grassmannian}
    (KO_0, \{*\}) \sim (\Z \times BO, \{*\}) \sim \varinjlim_{N}(\mathrm{Gr}(\R^{2N}), \{\R^N\}). 
\end{align}
Here, in the last term in \eqref{eq_KO_0_Grassmannian}, $\mathrm{Gr}(\R^{2N})$ is Grassmannian consisting of linear subspaces in $\R^{2N}$, and the basepoint
$\R^N \in \mathrm{Gr}_N(\R^{2N}) \subset \mathrm{Gr}(\R^{2N})$ is the point specified by the $N$-dimensional subspace $\R^N = \{0\}\oplus \R^N \subset \R^N \oplus \R^N = \R^{2N}$. 
The direct limit is taken by the inclusions
\begin{align}\label{eq_Gr_inclusion}
   \cdots \hookrightarrow (\mathrm{Gr}(\R^{2N}), \{\R^{N}\}) &\hookrightarrow (\mathrm{Gr}(\R \oplus \R^{2N} \oplus \R), \{\R^N \oplus \R\}) \hookrightarrow \cdots\\
    V &\mapsto V \oplus \R. \notag
\end{align}
This inclusion restricts to the one $\mathrm{Gr}_m(\R^{2N}) \hookrightarrow \mathrm{Gr}_{m+1}(\R^{2(N+1)})$ for each $m \le 2N$, and preserves the basepoints.
We have the following decomposition into the connected components, 
\begin{align}\label{eq_Grassmannian_decomposition}
   \varinjlim_{N}\mathrm{Gr}(\R^{2N}) =  \sqcup_{m \in \Z} \varinjlim_{N}\mathrm{Gr}_{m+N}(\R^{2N}) 
\end{align}
so that the each components are weakly homotopy equivalent to $BO$. 
The $\Z$-component in \eqref{eq_KO_0_Grassmannian} corresponds to the label $m$ in \eqref{eq_Grassmannian_decomposition}. 
The description \eqref{eq_KO_0_Grassmannian} directly leads to the famous definition of $KO^0$ as the Grothendieck group of the isomorphism classes of real vector bundles. 
Namely, if $X$ is a finite CW-complex, given an element $[f] \in [X, KO_0] = KO^0(X)$, we can represent it as a map $f \colon X \to  \mathrm{Gr}(\R^{2N})$ for $N$ large enough. 
Denote by $\theta_{2N} \to \mathrm{Gr}(\R^{2N})$ the tautological real vector bundle over $\mathrm{Gr}(\R^{2N})$. Then the element $[f] \in KO^0(X)$ corresponds to the following formal difference class of the real vector bundles over $X$, 
\begin{align}\label{eq_Gr_vec_bdle}
    [f^*\theta_{2N}] - [\underline{\R^N}]. 
\end{align}

Now, recalling Fact \ref{fact_mass_classifying}, let us consider a simple central graded $*$-algebra $A$ of type $0$, i.e., $[8; k, l] = M(k+l, \R)$ in Table \ref{table:gscR} for some $(k, l)$. Then we have
\begin{align}\label{eq_KO_0_Skew}
    (KO_0, \{*\})
    \sim \varinjlim_n (\Self_{A}^*(S_n), \{\beta\}) \sim \varinjlim_n (\Self_{A}^\dagger(S_n), \{\beta\}), 
\end{align}
where $\{S_n\}_n$ is a cofinal system of $\Sigma^{0, 1} A$-modules. 
The models \eqref{eq_KO_0_Grassmannian} and \eqref{eq_KO_0_Skew} are related as follows. 
Fix a volume element $u$ for $A$. 
We have $A = [8; k, l]\simeq M(k+l, \R)$ and $\Sigma^{0, 1} A \simeq [7; k+l] = M(k+l, \R) \oplus M(k+l, \R)$ with the $\Z_2$-gradings as in Table \ref{table:gscR}. 
Thus there are exactly two isomorphism classes of irreducible $\Sigma^{0, 1} A$-modules $S_+$ and $S_-$ each of which is $(k+l)$-dimensional and restricts to the same unique irreducible $A$-module. 
The modules $S_\pm$ are characterized by the property that $\beta \in \Sigma^{0, 1} A$ acts by the multiplication by $\pm u \in A \subset \Sigma^{0, 1} A$, respectively (here we use the choice of the volume element). 
Any $\Sigma^{0, 1} A$-module $S$ can be written as
\begin{align*}
    S = (S_+ \otimes \R^{N_+}) \oplus (S_- \otimes \R^{N_-}), 
\end{align*}
where $\R^{N_\pm}$ are the trivial $\Sigma^{0, 1} A$-modules, for some nonnegative integers $N_+$ and $N_-$. 
Using $\Self^\dagger_A(S_+) = \Self^\dagger_A(S_-) = \{\pm u\}$, we easily see that
\begin{align}\label{eq_Gr_vs_Skew}
    &\Self_{A}^\dagger\left((S_+ \otimes \R^{N_+} )\oplus (S_- \otimes \R^{N_-} )\right) \\
    &\quad = \left\{u \otimes a \ | \ a \in \mathrm{End}(\R^{N_+} \oplus \R^{N_-}), a^2 = 1, a^* = a
    \right\} \nonumber \\
    &\quad \simeq \left\{P \in \mathrm{End}(\R^{N_+} \oplus \R^{N_-}) \ | \  P^2 = 1, P^* = P
    \right\} \notag\\
    &\quad\simeq \mathrm{Gr}(\R^{N_+} \oplus \R^{N_-}). \notag
\end{align}
Here the second and the third maps are diffeomorphisms. 
The second one is given by $a \mapsto P = (1-a)/2$, and the third one maps the orthogonal projection $P$ to $\mathrm{Im}(P)$, in other words the $(+1)$-eigenspace of $P$. 
Under the diffeomorphism \eqref{eq_Gr_vs_Skew}, the basepoint 
\begin{align*}
    \beta = ((u\otimes \mathrm{id}_{\R^{N_+}}) \oplus (-u\otimes \mathrm{id}_{\R^{N_-}})) \in \Self_{A}^\dagger\left((S_+ \otimes \R^{N_+} )\oplus (S_- \otimes \R^{N_-} )\right)
\end{align*}
maps to
\begin{align*}
    \R^{N_+} =\{0 \}\oplus \R^{N_+ }\in \mathrm{Gr}_{N_+}(\R^{N_+} \oplus \R^{N_-}) \subset \mathrm{Gr}(\R^{N_+} \oplus \R^{N_-}). 
\end{align*}

As a cofinal sequence realizing \eqref{eq_KO_0_Skew}, we may take $(N_+, N_-) = (n, n)$, i.e., 
\begin{align}
    S_n = (S_+ \otimes \R^{n}) \oplus (S_- \otimes \R^{n}). 
\end{align}
We use the isomorphism $S_n \oplus S_1 \simeq S_{n+1}$ given as follows. 
Define $f \colon \R^n \oplus \R \simeq \R^{n+1}$ by $f(v \oplus w) = (w, v)$ and $g \colon \R^n \oplus \R \simeq \R^{n+1}$ by $g(v \oplus w) = (v, w)$. 
Then the required isomorphism is
\begin{align}
    (\mathrm{id}_{S_+} \otimes f) \oplus (\mathrm{id}_{S_-} \otimes g) \colon S_n \oplus S_1 \simeq S_{n+1}. 
\end{align}
Using this choice of $\{S_n\}_n$, we see that under the isomorphism \eqref{eq_Gr_vs_Skew} the sequence \eqref{eq_Skew_inclusion} coincides with the sequence \eqref{eq_Gr_inclusion}. 
Thus we get the explicit homeomorphism between the two models for $KO_0$ (depending on the choice of $u$), 
\begin{align}\label{eq_homeo_Gr_Skew}
    \varinjlim_{N}(\mathrm{Gr}(\R^{2N}), \{\R^N\}) \simeq \varinjlim_n (\Self_{A}^\dagger((S_+ \otimes \R^{n}) \oplus (S_- \otimes \R^{n})), \{\beta\}). 
\end{align}

\subsection{The Pontryagin character}\label{subsec_Ph_top}
For any spectrum $E$, there is a canonical map of spectra to the Eilenberg-MacLane spectrum theory called the {\it Chern-Dold homomorphism} \cite[Chapter II, 7.13]{rudyak1998}
\begin{align*}
    \mathrm{chd} \colon E \to H (\pi_{-\bullet}(E) \otimes_\Z\R), 
\end{align*}
characterized by the property that the induced map on the homotopy groups, 
\begin{align*}
   \mathrm{chd} \colon  \pi_k(E) \to \pi_k(H (\pi_{-\bullet}(E) \otimes_\Z\R)) = \pi_k(E) \otimes_\Z \R
\end{align*}
is given by the map $e \mapsto e \otimes 1$. 
In the cases $E = K$ and $E = KO$, the Chern-Dold homomorphisms are called the {\it Chern character} and the {\it Pontryagin character}, respectively, and denoted by (recall our notations \eqref{eq_KO_R_pt} and \eqref{eq_K_R_pt})
\begin{align}
    \mathrm{Ch}_{\mathrm{top}} &\colon K \to H\R[b, b^{-1}], \label{eq_def_Ch_top} \\
    \mathrm{Ph}_{\mathrm{top}} &\colon KO \to H\R[v, v^{-1}].   \label{eq_def_Ph_top}
\end{align}
They are related by the complexification map $KO \to K$. 
In this paper we use the canonical idendifications
\begin{align}\label{eq_isom_HR_b}
    H^*(X; \R[b, b^{-1}]) \simeq H^{2\Z + *}(X; \R) := \oplus_{k \in \Z}H^{2k+*}(X; \R), 
\end{align}
which sends $b \in H^{-2}(\pt; \R[b, b^{-1}])$ to $1 \in H^0(\pt; \R)$, and
\begin{align}\label{eq_isom_HR_v}
    H^*(X; \R[v, v^{-1}]) \simeq H^{4\Z + *}(X; \R) := \oplus_{k \in \Z}H^{4k+*}(X; \R),
\end{align}
which sends $v \in H^{-4}(\pt; \R[v, v^{-1}])$ to $-2 \in H^0(\pt; \R)$. 
Then the identifications \eqref{eq_isom_HR_b} and \eqref{eq_isom_HR_v} send the complexification $KO \to K$ to the map that is identity on $H^{8\Z}(-; \R)$ and the multiplication by $(-1)$ on $H^{8\Z + 4}(-; \R)$, because the complexification of $v$ is $2b^2 \in K^{-4}(\pt)$. 
We use this sign convention so that the formula for $\mathcal{R}$ in \eqref{eq_def_R} below becomes simple.

On manifolds, using the models of $K^0$ and of $KO^0$ by complex and real vector bundles respectively, we can realize the degree-zero part of the homomorphisms \eqref{eq_def_Ch_top} and \eqref{eq_def_Ph_top} as the homomorphism to the de Rham cohomologies by the Chern-Weil construction.  
Here we explain the case of $\mathrm{Ph}_{\mathrm{top}}$, 
\begin{align*}
    \mathrm{Ph}_{\mathrm{top}} \colon KO^0(X) \to H^0(X; \R[v, v^{-1}]) \simeq H^{4\Z}(X; \R) . 
\end{align*}
For a manifold $X$, we define a linear endomorphism on differential forms $\mathcal{R} \colon \Omega^*(X) \to \Omega^*(X)$ by, on the homogeneous elements, 
\begin{align}\label{eq_def_R}
    \mathcal{R}(\omega) := \begin{cases}
    (2\pi)^{-|\omega|/2}\omega & \mbox{ if }|\omega| \equiv 0 \pmod 2 \\
    \pi^{1/2}(2\pi)^{-(|\omega|+1)/2} \omega  & \mbox{ if }|\omega| \equiv 1 \pmod 2. 
    \end{cases} 
\end{align} 
Assume $X$ has a homotopy type of a finite CW-complex. 
Given an element in $KO^0(X)$, we can represent it by a smooth real vector bundle $E$ over $X$. 
Take any fiberwise inner product on $E$ and any connection $\nabla^E$ preserving it. 
Then the form
\begin{align*}
   \mathrm{Ph}(\nabla^E) := \mathcal{R} \circ \mathrm{Tr}\left(e^{(\nabla^E)^2} \right)
\end{align*}
lies in $\Omega^{4\Z}_{\mathrm{clo}}(X)$, and is called the {\it Pontryagin character form} for $\nabla^E$. 
Its cohomology class represents the image of $[E]$ under $\mathrm{Ph}_{\mathrm{top}}$, 
\begin{align*}
    \mathrm{Ph}_{\mathrm{top}}([E]) = [\mathrm{Ph}(\nabla^E)]. 
\end{align*}

The complex version \eqref{eq_def_Ch_top} is given by using the following endomorphism $\mathcal{R}_\C \colon \Omega^*(X; \C) \to \Omega^*(X; \C)$, 
\begin{align}\label{eq_def_R_complex}
    \mathcal{R}_\C(\omega) := \begin{cases}
    (-2\pi\sqrt{-1})^{-|\omega|/2}\omega  & \mbox{ if } |\omega| \equiv 0 \pmod 2 \\
    \pi^{1/2}(-2\pi \sqrt{-1})^{-(|\omega|+1)/2} \omega  & \mbox{ if } |\omega| \equiv 1 \pmod 2,
    \end{cases} 
 \end{align} 
and applying the corresponding formula for hermitian vector bundles with unitary connections.

\subsection{Twists in \texorpdfstring{$KO$}{KO}-theory}\label{subsec_twist}

We explain here the twisted groups $KO^{\mathcal{A}}_+$ in the general framework of twisted cohomology theory. 
In this subsection we work over $\R$. 

For the general theory of twisted cohomology theory, we refer to \cite[Section 22]{MaySigurdssonParametrizedHomotopy} and \cite{AndoBlumbergGepnerTwists}. 
In general, roughly speaking, given a generalized cohomology theory $E$ and a topological space $X$, a {\it twist} is given by a bundle of spectra $\mathcal{E} \to X$ with fiber $E$. 
Given such a twist, the twisted $E$-cohomology group is defined as the abelian group consisting of homotopy classes of sections, 
\begin{align}\label{eq_twist_general}
    E^\mathcal{E}(X) := \pi_0(\Gamma(X; \mathcal{E})). 
\end{align}
In particular we recover untwisted group $E^n(X)$ by setting $\mathcal{E} = \underline{\Sigma^n E}$. 
The appropriate choice of the class of ``bundles of spectra'' depends on the context. 
We may consider the twists given by maps to $B\mathrm{Aut}(E)$, where $\mathrm{Aut}(E)$ is the topological monoid of self-weak equivalences of $E$. 
In the case that $E$ is a ($E_\infty$-) ring $\Omega$-spectrum, a natural choice of a twist is a map to $B\mathrm{GL}_1(E)$, where $\mathrm{GL}_1(E)$ is the unit components of $E_0$ with respect to the ring structure of $\pi_0(E)$. 

Now we explain the formulation of twisted $KO$-theory group $KO^{\mathcal{A}}_+$ in Subsubsection \ref{subsubsec_Karoubi} from this viewpoint. 
The reference is \cite{DonovanKaroubi}. 
Also see \cite{AtiyahSegalTwistKandCohomology} for the case of $K$-theory, based on the model of the $K$-theory spectrum in terms of Fredholm operators. 

Let $A$ be one of the simple central graded $*$-algebras in Table \ref{table:gscR}. 
Recall that we explained a model of the $\mathrm{type}(A)$-th space of the $KO$-spectrum in terms of $\Self_{A}^*$ in Subsubsection \ref{subsubsec_KO_spectrum}. 
Using this model, we have an explicit homomorphism
\begin{align}\label{eq_hom_O_to_AutKO}
    \mathrm{Aut}_{*, \Z_2}(A) \to \mathrm{Aut}(KO_{\type(A)}). 
\end{align}

For this, we choose the following cofinal system $\{S_n\}_n$ of $\Sigma^{0, 1} A$-modules. 
\begin{align}\label{S_n_choice}
    S_n := (\Sigma^{0, 1} A) \otimes \R^n. 
\end{align}
Here an element $a \in \Sigma^{0, 1} A$ acts by $(a \cdot )\otimes \mathrm{id}_{\R^n}$. 
Recall that $A$ is a matrix algebra or a direct sum of two compies of them. 
We introduce an inner product on $\Sigma^{0, 1} A$ by the entry-wise $L^2$-inner product, i.e., $\langle a, b \rangle = \mathrm{Tr}(a^* b)$. 
It makes the action $*$-preserving. 
As an $A$-module, it is just a $2n$-copies of $A$. 

The automorphism group of $A$ acts on $\mathrm{End}(S_n)$ by the adjoint. 
We can check that it restricts to a homomorphism
\begin{align}\label{eq_aut_action_Skew}
    \mathrm{Aut}_{*, \Z_2}(A) \to \mathrm{Aut}(\Self_{A}^*(S_n)), 
\end{align}
for each $n$, and passing to the direct limit we get a homomorphism by Fact \ref{fact_mass_classifying},
\begin{align}\label{eq_aut_action_KO_1}
    \mathrm{Aut}_{*, \Z_2}(A) \to \mathrm{Aut}\left(\varinjlim_n (\Self_{A}^*(S_n), \{\beta\})\right) = \mathrm{Aut}\left(KO_{\mathrm{type}(A)}\right). 
\end{align}
This gives the desired homomorphism \eqref{eq_hom_O_to_AutKO}. 

Given a connected finite CW-complex $X$ with a bundle of simple central graded $*$-algebras $\mathcal{A} = P \times_{\mathrm{Aut}_{*, \Z_2}(A)} A$ associated to a principal $\mathrm{Aut}_{*, \Z_2}(A)$-bundle $P$ over $X$, take a classifying map $X \to B\mathrm{Aut}_{*, \Z_2}(A)$ for $P$. 
Using \eqref{eq_aut_action_KO_1} we define the following bundle of based spaces over $X$, 
\begin{align}\label{eq_Skew_bundle_V}
    (\mathcal{KO}_{\mathcal{A}}, \{\beta\}) := P \times_{\mathrm{Aut}_{*, \Z_2}(A)}\varinjlim_n (\Self_{A}^*(S_n), \{\beta\}). 
\end{align}
Then the relative version of \eqref{eq_twist_general} applied here gives
\begin{align}\label{eq_twistedKO_limit}
    KO^\mathcal{A}(X, Y) = \pi_0\left(\Gamma\left(X, Y;\mathcal{KO}_{\mathcal{A}}, \{\beta\}\right)\right). 
\end{align}
In fact, we have an isomorphism
\begin{align}\label{eq_KO_m=KO_twisted}
    KO_+^{\mathcal{A}}(X, Y) \simeq\pi_0\left(\Gamma\left(X, Y;\mathcal{KO}_{\mathcal{A}}, \{\beta\}\right)\right), 
\end{align}
where the left hand side is defined in Definition \ref{def_karoubi_twisted_KO}. 
The map is given as follows. 
Any element in $KO^{\mathcal{A}}_+(X, Y)$ is represented by a triple $(\slashed S, h_0, h_1)$ with $h_0^2 = 1$. 
For such a triple, we regard $h_0$ as the basepoint of each fiber of the fiber bundle $\Self_{\mathcal{A}}^*(\slashed S)$ over $X$. 
We may also use $h_0$ to regard $\slashed S$ as a $\Sigma^{0, 1}\mathcal{A}$-module bundle, with the generator of $Cl_{0, 1}$ acting by $h_0$. 
Now notice that any $\Sigma^{0, 1}\mathcal{A}$-module bundle can be embedded as a direct summand into $(\Sigma^{0, 1}\mathcal{A})\otimes \R^n$ for some $n$, and any such embeddings are homotopic to each other upon taking the direct limit in $n$. 
Recalling our choice \eqref{S_n_choice}, we have an embedding of bundles of based spaces, 
\begin{align}\label{eq_emb_Skew_bundle}
    \left(\Self_{\mathcal{A}}^*(\slashed S), \{h_0\} \right) \hookrightarrow (\mathcal{KO}_{\mathcal{A}}, \{\beta\})
\end{align}
uniquely up to homotopy. 
Fixing such an embedding we can regard $h_1$ as a section $h_1 \in \Gamma\left(X, Y; \mathcal{KO}_{\mathcal{A}}, \{\beta\}\right)$. 
The homotopy class of this $h_1$ only depends on the class $[\slashed S, h_0, h_1] \in KO_+^\mathcal{A}(X, Y)$. 
The fact that this map is an isomorphism follows by the following lemma, which is obvious by using the choice \eqref{S_n_choice}. 

\begin{lem}\label{lem_univ_Skew_bundle}
Fixing a connected $X \in \mathrm{CW}_f$ and $\mathcal{A} \in \Tw_{KO_+, X}$, we have a weak equivalence of fiber bundles, 
\begin{align}\label{eq_univ_Skew_bundle}
    \varinjlim_{(\slashed S, h_0)}\left(\Self_\mathcal{A}^*(\slashed S), \{h_0\} \right) \simeq (\mathcal{KO}_{\mathcal{A}}, \{\beta\}). 
\end{align}
Here the direct limit is taken over the direct system consisting of the pairs $(\slashed S, h_0)$ of $\mathcal{A}$-module bundles with inner product and continuous sections $h_0 \in \Gamma(X; \Self_\mathcal{A}^\dagger(\slashed S))$ and of their inclusions. 
The map from the left to the right is induced by \eqref{eq_emb_Skew_bundle}. 
\end{lem}

Recall that the isomorphism class of the group $KO^\mathcal{A}_+(X, Y)$ dependends on the class of $\mathcal{A}$ in $\mathrm{GBrO}(X)$ (Corollary \ref{cor_KO_twist_classification}). 
Indeed the map \eqref{eq_hom_O_to_AutKO} factors through $B\mathrm{GL}_1(KO)$ up to homotopy, where the map to $B\mathrm{GL}_1(KO)$ factors as
\begin{align}\label{twist_mass_KO_classification}
    B\mathrm{Aut}_{*, \Z_2}(A) \xrightarrow{w} B\mathrm{O}\langle 0, 1, 2 \rangle \xrightarrow{\iota} B\mathrm{GL}_1(KO) . 
\end{align}
up to homotopy. 
Here $B\mathrm{O}\langle 0, 1, 2 \rangle$ is the second stage in the Postonikov system of $B\mathrm{O}$, which classifies $\mathrm{GBrO}(-)$.

\subsection{The twisted Pontryagin character}\label{subsec_twisted_Ph_top}
In this subsection we explain the twisted Pontryagin character homomorphism. 

Assume we are given two ring spectra $E$ and $F$ with multiplicative homomorphism $f \colon E \to F$. 
For a topological space $X$, given an element $\tau \in [X, B\mathrm{GL}_1(E)]$ specifying a twist for $E$, by composing it with the map $f \colon B\mathrm{GL}_1(E) \to B\mathrm{GL}_1(F)$ we have the corresponding twist $f_*\tau$ for $F$.  
It also induces the maps between associated bundles, so in view of the definition \eqref{eq_twist_general} we get a homomorphism
\begin{align}\label{eq_trans_twist}
    f \colon E^\tau(X) \to F^{f_*\tau }(X). 
\end{align}
It satisfies the obvious functoriality. 
It is useful to view \eqref{eq_trans_twist} in the following way. 
Let $\pi \colon \mathcal{E}_\tau \to X$ be the bundle with fiber $E$ associated with $\tau$. 
Then $\pi^*f_*\tau \in [\mathcal{E}_\tau, B\mathrm{GL}_1(F)]$ specifies a twist of $F$ for the total space $\mathcal{E}_\tau$, and we get an element
\begin{align}\label{eq_trans_twist_total}
    [f] \in F^{\pi^*f_*\tau}(\mathcal{E}_\tau). 
\end{align}
We recover the homomorphism \eqref{eq_trans_twist} from \eqref{eq_trans_twist_total} by the pullback with respect to a section in $E^\tau(X) = \pi_0(\Gamma(X; \mathcal{E}_\tau))$. 
Also note that, for each point $x\in X$, under a choice of an isomorphism $\mathcal{E}_\tau|_x \simeq E$, the element \eqref{eq_trans_twist_total} restricts on the fiber at $x$ to the natural transformation $f \colon E \to F$.  

Now we turn to the twisted Pontryagin character homomorphism. 
Recall that we are using the twists of $KO$-theory coming from the composition \eqref{twist_mass_KO_classification}. 
We denote by $(\mathcal{KO}_n)_\iota$ the $n$-th space component of the $KO$-spectrum bundle $\mathcal{KO}_\iota$ over $BO\langle 0, 1, 2 \rangle$. 
We denote by $\{\beta\} \subset (\mathcal{KO}_n)_\iota$ the fiberwise basepoint, in analogy with \eqref{eq_twistedKO_limit}. 
Applying the above procedure to the multiplicative homomorphism $\mathrm{Ph}_{\mathrm{top}} \colon KO \to H\R[v, v^{-1}]$, we get the {\it twisted Pontryagin character homomorphism},
\begin{align}\label{eq_twisted_Ph_top}
    \mathrm{Ph}_{\mathrm{top}} \colon KO^{\mathcal{A}}_+(X, Y) \to H^{4\Z + \type(\mathcal{A})}(X, Y; \mathrm{Ori}(\mathcal{A}) ),  
\end{align}
which is natural in $(X, Y, \mathcal{A}) \in \Tw^2_{KO_+}$. 
Also, by \eqref{twist_mass_KO_classification} and the above argument, the natural transformation \eqref{eq_twisted_Ph_top} is specified by the universal element for each $n$,
\begin{align}\label{eq_univ_twisted_Ph}
    [\mathrm{Ph}_{\mathrm{top}}] \in H^{4\Z + n }((\mathcal{KO}_{n})_\iota, \{\beta\}; E\Z_2 \times_{\Z_2}  \R). 
\end{align}
Here $E\Z_2 / \Z_2 = B\Z_2 \simeq K(\Z_2, 1)$ and we denote its pullback to $B\langle 0, 1, 2 \rangle$ by the same symbol. 
$\Z_2$ acts on $\R$ by the multiplication of $\{\pm 1\} \simeq \Z_2$. 
Each transformation \eqref{eq_twisted_Ph_top} is given by the pullback of the universal class \eqref{eq_univ_twisted_Ph} to $\mathcal{KO}_\mathcal{A}$. 
We observe the following. 

\begin{lem}\label{lem_twisted_Ph_characterization}
For each $n$, the element \eqref{eq_univ_twisted_Ph} is the unique element which restricts to $[\mathrm{Ph}_{\mathrm{top}}] \in H^{4\Z + n}(KO_{n}, \{*\};  \R)$ on each fiber. 
\end{lem}
\begin{proof}
This is because the restriction to a fiber gives an isomorphism in this case, since the higher integral cohomology groups of $BO\langle 0, 1, 2 \rangle$ are rationally trivial. 
We remark that the analogous argument appears in the paper by Atiyah and Segal \cite{AtiyahSegalTwistKandCohomology} in the case of the twisted $K$-theory. 
\end{proof}

\section{Generalized Clifford superconnections}\label{sec_superconn}
In this section we set up the {\it generalized Clifford superconnection formalism}.  
Actually, in order to define untwisted differential $KO$-groups $\widehat{KO}^{A}_+$, we only need a particularly simple class of superconnections defined on trivial bundles $\underline{S}$ of $A$-modules. 
Before developing the general formulation in Subsection \ref{subsec_superconn_general}, we first explain this simple case in Subsection \ref{subsec_superconn_triv}. 
The readers who are only interested in the untwisted case only need Subsection \ref{subsec_Clifford_end} and Subsection \ref{subsec_superconn_triv}, and can go directly to Section \ref{sec_diff_KO}. 
In this section we work over $\R$.

\subsection{Traces}\label{subsec_Clifford_end}

\begin{defn}[$\mathrm{Tr}_u$]\label{def_u_str}
Let $A$ be a simple central graded $*$-algebra, and $u \in A$ be a volume element. 
For an $A$-module $S$, we define an $\R$-linear map
\begin{align*}
    \mathrm{Tr}_u(\xi) \colon \mathrm{End}_{A}(S) \to \R
\end{align*}
by the following. 
\begin{itemize}
    \item If $A$ is of odd-type, set
    \begin{align*}
        \mathrm{Tr}_{u}(\xi) := 2^{1/2}(\dim_\R A)^{-1/2}\mathrm{Tr}_S(u \cdot \xi). 
    \end{align*}
    Note that since $u \in A^1$ we have $\mathrm{Tr}_u(\xi) = 0$ for $\xi \in \mathrm{End}^1_{A}(S)$. 
    \item If $A$ is of even-type and nondegenerate, set
    \begin{align*}
        \mathrm{Tr}_{u}(\xi) :=\begin{cases}
        (\dim_\R A)^{-1/2}\mathrm{Tr}_S(u\cdot \xi) & \mbox{ if } \xi \in \mathrm{End}^1_{A}(S), \\
        0 & \mbox{ if } \xi\in \mathrm{End}^0_{A}(S). 
        \end{cases}
    \end{align*}
    \item If $A$ is degenerate, set
    \begin{align*}
        \mathrm{Tr}_{u}(\xi) :=(\dim_\R A)^{-1/2}\mathrm{Tr}_S(u\cdot \xi) 
    \end{align*}
    for any $\xi \in \mathrm{End}_{A}(S)$. 
\end{itemize}
\end{defn}
When $A$ is nondegenerate, $\mathrm{Tr}_u$ is a supertrace on the $\Z_2$-graded algebra $\mathrm{End}_{A}(S)$, as follows. 
The proof is straightforward, by using $u \in A^{|\mathrm{type}(A)|}$. 

\begin{lem}\label{lem_Gamma_trace}
Let $S$ be an $A$-module and $\xi_1, \xi_2 \in \mathrm{End}_{A}(S)$. 
If $A$ is nondenerate, we have
\begin{align*}
\mathrm{Tr}_u(\{\xi_1, \xi_2\}) = 0. 
\end{align*}
If $A$ is degenerate, we have
\begin{align*}
\mathrm{Tr}_u([\xi_1, \xi_2]) = 0. 
\end{align*}
\end{lem}

Moreover, $\mathrm{Tr}_u$ is compatible with tensoring negligible modules, as follows.

\begin{lem}\label{lem_tr_u_negligible}
Let $V = V^0 \oplus V^1$ be a $\Z_2$-graded real vector space, and $\gamma_V$ the $\Z_2$-grading operator on $V$. Let $A$ and $S$ be as before, and
let $u$ be a volume element of $A$. 
Then $\gamma_V \widehat{\otimes} u$ is a volume element of $\mathrm{End}(V) \widehat{\otimes} A$. 
For any $\xi \in \End_A(S)$, we have
\begin{align*}
    \mathrm{Tr}_u(\xi) = \mathrm{Tr}_{\gamma_V \widehat{\otimes} u}(\psi^V(\xi)). 
\end{align*}
\end{lem}

For later use, we also regard Definition \ref{def_u_str} as giving a linear map
\begin{align}\label{eq_str_ori_A}
    \mathrm{Tr}_{A} \colon \mathrm{End}_{A}(S) \to \mathrm{Ori}(A). 
\end{align}

\subsection{The case for trivial bundles of \texorpdfstring{$A$}{A}-modules}\label{subsec_superconn_triv}

Let $A$ be a simple central graded $*$-algebra. 
Throughout this subsection, we assume that $A$ is nondegenerate. 
Let $S$ be an $A$-module with an inner product. 
Let $X$ be a manifold, and consider the trivial $A$-module bundle $\underline{S}$ over $X$. 
Here, we explain the generalized Clifford superconnection formalism in Subsection \ref{subsec_superconn_general} applied to these settings.

\subsubsection{Generalized Clifford superconnections associated to gradations}
Recall that we have introduced the $\Z_2$-graded algebra structure on $\mathrm{End}_{A}(S)$ in Definition \ref{def_End_A}. 
We introduce a $\Z_2$-graded algebra structure on $\Omega^*(X; \mathrm{End}_{A}(S))$ by the graded tensor product, 
\begin{align*}
    \Omega^*(X; \mathrm{End}_{A}(S)) &= \Omega^*(X) \widehat{\otimes}_{C^\infty(X)}C^\infty(X; \mathrm{End}_{A}(S)). 
\end{align*}
This means that the multiplication is given by the formula \eqref{eq_multi_form}: 
\begin{align*}
    (\omega \widehat{\otimes} \xi) \cdot (\omega' \widehat{\otimes} \xi') :=
   (-1)^{|\xi|\cdot| \omega'|} (\omega\wedge \omega') \otimes (\xi \cdot \xi'). 
\end{align*}
The even and odd part of $\Omega^*(X; \mathrm{End}_{A}(S))$ with respect to this grading are denoted by $\Omega^*(X; \mathrm{End}_{A}(S))^0$ and $\Omega^*(X; \mathrm{End}_{A}( S))^1$, respectively. 
We let the algebra $\Omega^*(X; \mathrm{End}_{A}(S))$ act on $\Omega^*(X; S) := C^\infty(X; \wedge T^*X \otimes S)$ by the formula \eqref{eq_action_form}, respecting the grading:
\begin{align*}
    (\omega \widehat{\otimes} \xi) \cdot (\eta \otimes \psi) := (-1)^{|\xi|\cdot|\eta|}(\omega \wedge \eta) \otimes (\xi \cdot \psi). 
\end{align*}

Given $h \in C^\infty(X, \Self_{A}(S))$, consider the following differential operator on $\Omega^*(X; S)$, 
\begin{align}\label{eq_Grad_triv}
    \Grad := d + h \colon \Omega^*(X; S) \to \Omega^*(X; S). 
\end{align}
This is an example of a {\it self-adjoint superconnection} (Definition \ref{def_superconn} and \ref{def_skew_superconn}) for the case $(\mathcal{A}, \nabla^{\mathcal{A}}) = (\underline{A}, d)$ and $\slashed S = \underline{S}$. 
The square of the operator \eqref{eq_Grad_triv} is given by the action of the following element in $\Omega^*(X; \mathrm{End}_{A}(S))^0$, 
\begin{align}
     F({\Grad}) = \Grad^2 = dh + h^2 \in \Omega^*(X; \mathrm{End}_{A}(S))^0. 
\end{align}
This is the {\it twisting curvature} for $\Grad$ in Lemma \ref{lem_curv_even}.

As in the case of usual connections, we apply polynomials to $F({\Grad})$ and get characteristic forms. 
Recall that we defined a linear map $\mathrm{Tr}_A \colon \End_A(S) \to \mathrm{Ori}(A)$ in \eqref{eq_str_ori_A}. 
We extend it to a linear map $\mathrm{Tr}_A \colon \Omega^*(X; \mathrm{End}_{A}(S)) \to \Omega^*(X; \mathrm{Ori}(A))$ by applying it to the coefficients (Definition \ref{def_V_str}). 
\begin{defn}[{Definition \ref{def_Ch} applied to $\Grad = d + h$}]\label{def_Ch_triv}
In the above settings, {\it Pontryagin character form} of $\Grad = d + h$ is defined by 
\begin{align}
    \mathrm{Ph}_\lself(d + h) :=   \mathrm{Tr}_A(e^{-F({\Grad})})
    = \mathrm{Tr}_A(e^{-dh - h^2})
    \in  \Omega_{\mathrm{clo}}^{4\Z +\mathrm{type}(A)+1}(X; \Ori(A)).  
\end{align}
The closedness follows from Proposition \ref{prop_transgression} (1), and the fact that the form appears in degrees $ \mathrm{type}(A) + 1 \pmod 4$ follows from Proposition \ref{prop_chracteristic_form_mod4} below. 
\end{defn}

The above Pontryagin character form $\mathrm{Ph}_\lself(d + h)$ depends on $h$, but it is always exact. 
Moreover, if we are given an element $h_I \in C^\infty(I \times X, \Self_{A}(S))$, which is regarded as a homotopy $h_I = \{h_t\}_{t \in I}$ from $h_0$ to $h_1$, we get an explicit form cobounding $\mathrm{Ph}(d + h_1) - \mathrm{Ph}(d + h_0)$, as follows. 
\begin{defn}[{Definition \ref{def_Ch} applied to $\Grad^{\slashed S} = d + h$}]\label{def_CS_triv}
In the above settings, for an element $h_I \in C^\infty(I \times X, \mathrm{Skew}_{A}(S))$, we define its {\it Chern-Simons form} 
$\mathrm{CS}_\lself( d_{I \times X} + h_I) \in\Omega^{4\Z +\mathrm{type}(A)}(X; \Ori(A))$ by
\begin{align}\label{eq_def_CS_triv}
    \mathrm{CS}_\lself( d_{I \times X} + h_I) &:= \int_{I} \mathrm{Ph}_\lself\left(d_{I \times X} + h_I\right) \notag \\
    &= -\int_I dt \wedge \mathrm{Tr}_A \left(\frac{dh_I}{dt} e^{-d_X h_I - (h_I)^2}
    \right). 
\end{align}
Here we denote by $d_{I \times X}$ and $d_X$ the de Rham differential on $I \times X$ and $X$, respectively. 
The last equality in \eqref{eq_def_CS_triv} follows from
\begin{align*}
    (d_{I \times X} + h_I)^2 = (d_X + h_I)^2 + dt \wedge \frac{dh_I}{dt}
\end{align*}
(also see \eqref{eq_cylinder_characteristic_form}). 
\end{defn}

We have \eqref{eq_transgression_Ph}
\begin{align}\label{eq_CS_transgression_triv}
    d\mathrm{CS}_\lself( d_{I \times X} + h_I) = \mathrm{Ph}_\lself(d + h_1) - \mathrm{Ph}_\lself(d + h_0). 
\end{align}
In particular, for any $h \in C^\infty(X, \mathrm{Self}_{A}(S))$, we have the homotopy $h_I = th$ from $0$ to $h$. 
Thus, applying \eqref{eq_CS_transgression_triv} to the homotopy, we see that $\mathrm{Ph}_\lself(d + h)$ is always exact, as stated above. 

\subsubsection{The Pontryagin character forms for gradations}\label{subsubsec_Ph_triv}
As we have seen in the end of the last subsubsection, the Pontryagin character form $\mathrm{Ph}_\lself(d + h)$ itself has no interesting cohomological information. 
However, the point is that we know a particular choice $\mathrm{CS}_\lself( d_{I \times X} + th)$ of the form cobounding it. 
Based on this observation, we introduce the {\it Pontryagin character form} $\mathrm{Ph}_\lself(h)$ for {\it invertible} sections (gradations) $h \in C^\infty(X, \Self^*_{A}(S))$, which captures the nontrivial cohomological information. 

Assume we are given an invertible section $h \in C^\infty(X, \Self^*_{A}(S))$. 
Consider the manifold $(0, \infty) \times X$ and the section $th \in C^\infty((0, \infty) \times X, \Self^*_{A}(S))$, where $t$ is the coordinate in $(0, \infty)$. 
Applying Definition \ref{def_Ch_triv} to this, we get
\begin{align}\label{eq_Ch_tm}
    \mathrm{Ph}_\lself\left(d_{(0, \infty) \times X} + th\right) &=  -dt \wedge \mathrm{Tr}_A\left(h e^{-td_X h - t^2h^2} \right)
    \\&\in  \Omega_{\mathrm{clo}}^{4\Z+\mathrm{type}(A)+1}((0, \infty) \times X ; \Ori(A)). \notag
\end{align}
The section $th \in C^\infty((0, \infty) \times X, \Self^*_{A}(S))$ can be considered as a family of sections $\{th \in C^\infty(X, \Self^*_{A}(S))\}_{t \in (0, \infty)}$ parametrized by $(0, \infty)$. 
The integration of the form \eqref{eq_Ch_tm} in the $(0, \infty)$-direction can be regarded as a limit of the Chern-Simons form in Definition \ref{def_CS_triv}. 
As a consequence of the invertibility of the section $h$, we can show that the integration indeed converges. 

\begin{lem}[{The special case of Lemma \ref{lem_Ph_convergence}}]\label{lem_Ph_convergence_triv}
In the above settings, the following integration converges pointwise, and the resulting form is closed. 
\begin{align}\label{eq_lem_Ph_triv}
    -\int_{(0, \infty)} \mathrm{Ph}_\lself\left(d_{(0, \infty) \times X} + th\right)
   & = \int_{(0, \infty)} dt \wedge \mathrm{Tr}_A\left(h e^{-td_X h - t^2h^2} \right) \\
    &\in \Omega_{\mathrm{clo}}^{4\Z +\mathrm{type}(A)}(X; \Ori(A)). \notag
\end{align}
\end{lem}
It turns out that the form \eqref{eq_lem_Ph_triv} captures the nontrivial cohomological information on $h$. 
Recall the homomorphism $\mathcal{R}$ given in \eqref{eq_def_R}. 

\begin{defn}[{$\mathrm{Ph}_\lself(h)$ : a special case of Definition \ref{def_Ph_m} (1)}]\label{def_Ph_m_triv}
For $h \in C^\infty(X, \Self^*_{A}(S))$, its {\it Pontryagin character form} 
\[
\mathrm{Ph}_\lself(h)
\in \Omega_{\mathrm{clo}}^{4\Z + \type(A)}(X; \Ori(A))
\]
is defined by
\begin{align}
    \mathrm{Ph}_\lself(h) &:=   -\pi^{-1/2}\mathcal{R} \circ\int_{(0, \infty)} \mathrm{Ph}_\lself(d_{(0, \infty) \times X} + th) \\
    &=\pi^{-1/2}\mathcal{R} \circ\int_{(0, \infty)} dt \wedge \mathrm{Tr}_A\left(h e^{-td_X h - t^2h^2} \right). \notag
\end{align}
\end{defn}

The cohomology class of $\mathrm{Ph}_\lself(h)$ only depends on the homotopy class of $h$ in $C^\infty(X, \Self^*_{A}(S))$. 
Indeed, given an element $h_I \in C^\infty(I \times X , \Self^*_{A}(S))$ regarded as a homotopy between $h_0 := h_I|_{\{0\} \times X}$ and $h_1 := h_I|_{\{1\} \times X}$, we get an explicit form cobounding $ \mathrm{Ph}_\lself(h_1) -  \mathrm{Ph}_\lself(h_0)$, as follows. 

\begin{defn}[{$\mathrm{CS}_\lself(h_I)$ : a special case for Definition \ref{def_Ph_m} (2)}]\label{def_CS_m_triv}
For $h_I \in C^\infty(I \times X , \Self^*_{A}(S))$, we define its {\it Chern-Simons form} by
\begin{align} 
    \mathrm{CS}_\lself(h_I) := \int_{I}\mathrm{Ph}_\lself\left(h_I \right) 
    \in \Omega^{4\Z + \type(A) - 1}(X;\Ori(A) ).  
\end{align}
\end{defn}

The fact that $\mathrm{CS}_\lself(h_I)$ satisfies
\begin{align}\label{eq_Ph=dCS_triv}
     \mathrm{Ph}_\lself(h_1) -  \mathrm{Ph}_\lself(h_0)  = d\mathrm{CS}_\lself(h_I) 
\end{align}
follows from the closedness of $\mathrm{Ph}_\lself(h)$. 

\subsubsection{Properties of the Pontryagin character forms}\label{subsubsec_properties_Ph}
Now we explain the properties of the Pontryagin character form $\mathrm{Ph}_\lself(h)$ for invertible sections $h \in C^\infty(X, \Self^*_{A}(S))$. 

To begin with, it is important to notice that $\mathrm{Ph}_\lself(h)$ is the pullback of the {\it universal Pontryagin form} on the space $\Self_{A}^*(S)$. 
Indeed, fixing $S$, we have the {\it universal gradation}, or the {\it tautological gradation} $h_{\mathrm{univ}}$ on $\Self^*_{A}(S)$, given by the identity map, 
\begin{align}
    h_{\mathrm{univ}} := \mathrm{id} \in C^\infty(\Self^*_{A}(S), \Self^*_{A}(S)). 
\end{align}
We get its Pontryagin character form, 
\begin{align}\label{eq_univ_Ph}
    \mathrm{Ph}_\lself(h_{\mathrm{univ}}) \in \Omega^{4\Z + \type(A)}_{\mathrm{clo}}(\Self^*_{A}(S); \Ori(A) ).
\end{align}
We easily see that Definition \ref{def_Ph_m_triv} is contravariant in $X$, and in particular for any $h \in C^\infty(X, \Self_{A}^*(S))$ we have
\begin{align}
    \mathrm{Ph}_\lself(h) = h^* \mathrm{Ph}_\lself(h_{\mathrm{univ}}). 
\end{align}
Thus we call the form $\mathrm{Ph}_\lself(h_{\mathrm{univ}})$ the {\it universal Pontryagin form}. 
The various statements on $\mathrm{Ph}_\lself(h)$ below are also functorial in $X$, so they can be treated as statements on each $h$, as well as statements on the universal one. 
Both viewpoints are useful. 

First, $\mathrm{Ph}_\lself(h)$ is invariant under tensoring negligible modules. 
Recall Lemma \ref{lem_periodicity_negligible_bundle}. 
Given a $\Z_2$-graded vector space $V = V^0 \oplus V^1$ with a positive definite inner product on each $V^i$ associating the negligible module $\End(V)$,
for an $A$-module $S$ we get the corresponding $\End(V) \widehat{\otimes}A$-module $V \otimes S$, and the isomorphism $\psi^V \colon \Self_A^*(S) \simeq \Self^*_{\End(V) \widehat{\otimes}A}(V \otimes S)$. 
The following result follows from a straightforward algebraic computation by using Lemma \ref{lem_tr_u_negligible}. 
\begin{lem}[{Invariance of $\mathrm{Ph}_\lself(h)$ under tensoring negligible modules}]\label{lem_Ph_negligible_triv}
In the above settings, for $h \in C^\infty(X; \Self_A^*(S))$ we have
    \begin{align*}
        \mathrm{Ph}_\lself(h) = \mathrm{Ph}_\lself(\psi^V(h)), 
    \end{align*}
where we use the isomorphism $\Ori(A) \simeq \Ori(\End(V)\widehat{\otimes} A)$ by $u \mapsto \gamma_V \widehat{\otimes} u$. 
\end{lem}

Next we look at the compatibility with the suspension isomorphism in $KO$-theory. 
Recall that, by Fact \ref{fact_susp_twisted_+}, the suspension isomorphism of $KO$-theory, 
\begin{align}
     KO^{\Sigma^{0, 1}A}_+(X, Y) \simeq
 KO^{A}_+(X \times I, X \times \partial I \cup Y \times I), 
 \end{align}
is realized by the map $(S, h_0, h_1) \mapsto (S, \widetilde{h}_0, \widetilde{h}_1)$.
There, we defined the map
 \begin{align}\label{eq_susp_gradation}
     C^\infty(X, \Self_{\Sigma^{0, 1}A}^*(S)) &\to C^\infty(I \times X, \Self_{A}^*(S)), \\
     h &\mapsto \widetilde{h}(\theta, x) = \beta \cos \pi \theta + h(x) \sin \pi \theta. 
 \end{align}
As remarked after Fact \ref{fact_mass_classifying}, it is essentially the structure map of the $KO$-spectrum explained in Subsubsection \ref{subsubsec_KO_spectrum}. 
We have the following. 
 
 \begin{prop}\label{prop_int_Ph_m}
 For any element $h \in C^\infty(X, \Self_{\Sigma^{0, 1}A}^\dagger (S))$, we have
  \begin{align*}
         \mathrm{Ph}_\lself(h) = \int_I \mathrm{Ph}_\lself(\widetilde{h}). 
     \end{align*}
     Here we use the isomorphism $\Ori(A) \simeq \Ori(\Sigma^{0, 1}A)$ given by $u \mapsto u\widehat{\otimes}\beta$. 
\end{prop}
The proof is done by a direct computation, which we give in Subsection \ref{subsec_proof_int_Ph_m}.

Finally, we identify the cohomology class of $\mathrm{Ph}_\lself(h)$, justifying the name ``Pontryagin character form''. 
For this, it is convenient to work with the universal one, $\mathrm{Ph}_\lself(h_{\mathrm{univ}}) \in \Omega_{\mathrm{clo}}^{4\Z + \type(A)}(\Self^*_{A}(S); \Ori(A))$, as explained at the beginning of this subsubsection. 

Recall that we have the model of the $KO$-spectrum realized as a direct limit of $\mathrm{Skew}_{A}^*$ given in Subsubsection \ref{subsubsec_KO_spectrum}. 
Letting $\{S_n\}_n$ be any cofinal system of $\Sigma^{0, 1}A$-modules, 
we have the weak homotopy equivalence (Fact \ref{fact_mass_classifying}), 
\begin{align*}
    (KO_{\mathrm{type}(A)}, \{*\}) \sim \varinjlim_n (\Self_{A}^*(S_n), \{\beta\}). 
\end{align*}
Also recall that this equivalence depends on the choice of a volume element $u \in A$. 
Any $\Sigma^{0, 1}A$-module $S$ admits an inclusion $S \hookrightarrow S_n$ of $\Sigma^{0, 1}A$-modules for $n$ large enough, and its image is a direct summand. Any such inclusions are homotopic in the direct limit. 
We denote by 
\begin{align}
    \iota_{S, u} \in \left[(\Self_{A}^*(S), \{\beta\}), (KO_{\type (A)}, \{*\}) \right]
\end{align}
the homotopy class of the induced map.  
The topological Pontryagin character homomorphism \eqref{eq_def_Ph_top} can be regarded as an element
\begin{align*}
    \mathrm{Ph}_{\mathrm{top}} \in  {H}^{4\Z +\type(A)}\left( KO_{\type(A)}, \{*\}; \R \right). 
\end{align*}
The universal Pontryagin character form \eqref{eq_univ_Ph} indeed represents the topological Pontryagin character, as follows. 
\begin{thm}\label{thm_Ph=Ph_top}
Let $A$ be a nondegenerate simple central graded $*$-algebra and $S$ be a $\Sigma^{0, 1}A$-module. 
Let $u$ be a volume element of $A$, and use it to trivialize $\Ori(A)$. 
Then we have the following equality in $H^{4\Z +{\type}(A)}\left( \Self_{A}^*(S), \{\beta\}; \R\right)$, 
\begin{align}
   \mathrm{Rham}\left( \mathrm{Ph}_\lself(h_{\mathrm{univ}})\right) = \iota_{S, u}^*  \mathrm{Ph}_{\mathrm{top}}. 
\end{align}
\end{thm}
The proof of Theorem \ref{thm_Ph=Ph_top} is given in Subsection \ref{subsec_proof_Ph=Ph_top}. 
Using this universal result, we see that $\mathrm{Ph}_\lself(h)$ realizes the topological Pontryagin character homomorphism for $KO_+$, as follows. 

\begin{cor}\label{cor_Ph=Ph_top}
Let $A$ be a nondegenerate simple central graded $*$-algebra and $(X, Y)$ be an object of $\mathrm{MfdPair}_f$. 
Representing classes of $KO_+^{A}(X, Y)$ by triples $(S, h_0, h_1)$ with $h_i \in C^\infty(X; \Self_A^*(S))$, we have the following realization of the topological Pontryagin character homomorphism
\begin{align*}
    \mathrm{Ph}_{\mathrm{top}} \colon KO_+^{A}(X, Y) &\to H^{4\Z +{\type}(A) }(X, Y; \Ori(A)), \\
    [S, h_0, h_1] &\mapsto \mathrm{Rham}\left(\mathrm{Ph}_\lself(h_1) - \mathrm{Ph}_\lself(h_0)  \right). 
\end{align*}
\end{cor}
\begin{proof}
This follows directly from Theorem \ref{thm_Ph=Ph_top} and the definition \eqref{eq_difference_class} of the isomorphism  \eqref{eq_fact_mass_classifying}. 
\end{proof}

\subsection{The generalized Clifford superconnection formalism}\label{subsec_superconn_general}

In this subsection, we develop the {\it generalized Clifford superconnection formalism}. 
The construction below can be regarded as a generalization of the (real variant of the) usual {\it superconnection formalism}, which was  developed by Quillen \cite{QuillenSuperconnection} and has been broadly used since then. 
We explain the relation with the Quillen's formalism in Subsubsection \ref{subsubsec_Quillen}. 
Here we just remark that the Quillen's {\it (even) superconnection formalism} corresponds to the case $\mathcal{A} = \underline{Cl_{0, 1}}$ and the Quillen's {\it odd superconnection formalism} corresponds to the case $\mathcal{A} = \underline{Cl_{1, 1}}$.

Let $X$ be a manifold and $\mathcal{A}$ be a bundle of simple central graded $*$-algebras over $\R$ on $X$. 
Throughout this subsection, we always assume that the fibers of $\mathcal{A}$ are {\it nondegenerate} algebras\footnote{
When we are interested in bundles of degenerate algebras and modules over them, we tensor negligible modules to produce bundles of nondegenerate algebras, it is enough to use the isomorphism in Lemma \ref{lem_periodicity_negligible_bundle} (3), and then apply the superconnection formalism to the latter. 
Actually, the Quillen's {\it odd superconnection formalism} can be regarded as such an example. 
See Subsubsection \ref{subsubsec_Quillen} for details. 
}.

Given an $\mathcal{A}$-module bundle $\slashed S$ over a manifold $X$, 
we introduce an $\mathcal{A}$-module bundle structure on $\wedge T^*X \otimes \slashed S$ as follows. 
For each point $x \in X$, let $\mathcal{A}_x$ act on $\wedge T^*_xX \otimes \slashed S_x$ by the formula
\begin{align}\label{eq_A_action_form}
    c \cdot (\omega \otimes \psi ) := (-1)^{|c| \cdot|\omega|}\omega \otimes (c \cdot \psi). 
\end{align}
This makes $\wedge T^*X \otimes \slashed S$ an $\mathcal{A}$-module bundle over $X$. 
Let us introduce $\Z_2$-graded algebra structures on $\Omega^*(X; \mathcal{A})$ and $\Omega^*(X; \mathrm{End}_{\mathcal{A}}(\slashed S))$ by the graded tensor products,  
\begin{align*}
 \Omega^*(X; \mathcal{A}) &= \Omega^*(X) \widehat{\otimes}_{C^\infty(X)}C^\infty(X; \mathcal{A}). \\
    \Omega^*(X; \mathrm{End}_{\mathcal{A}}(\slashed S)) &= \Omega^*(X) \widehat{\otimes}_{C^\infty(X)}C^\infty(X; \mathrm{End}_{\mathcal{A}}(\slashed S)). 
\end{align*}
This means that the multiplication in these algebras are given by
\begin{align}\label{eq_multi_form}
    (\omega \widehat{\otimes} \xi) \cdot (\omega' \widehat{\otimes} \xi') :=
   (-1)^{|\xi|\cdot| \omega'|} (\omega\wedge \omega') \otimes (\xi \cdot \xi'). 
\end{align}
The even and odd part of $\Omega^*(X; \mathrm{End}_{\mathcal{A}}(\slashed S))$ with respect to this grading are denoted by $\Omega^*(X; \mathrm{End}_{\mathcal{A}}(\slashed S))^0$ and $\Omega^*(X; \mathrm{End}_{\mathcal{A}}(\slashed S))^1$, respectively. 
We also use the corresponding notation $\Omega^*(X; \mathcal{A})^i$, $i \in \Z_2$. 

Let the algebras $\Omega^*(X; \mathcal{A})$ and $\Omega^*(X; \mathrm{End}_{\mathcal{A}}(\slashed S))$ act on the space of differential forms $\Omega^*(X; \slashed S) = C^\infty(X; \wedge T^*X \otimes \slashed S)$ by
\begin{align}\label{eq_action_form}
    (\omega \widehat{\otimes} \xi) \cdot (\eta \otimes \psi) := (-1)^{|\xi|\cdot|\eta|}(\omega \wedge \eta) \otimes (\xi \cdot \psi). 
\end{align}
We have 
\begin{align*}
    \mathrm{End}_{\mathcal{A}}(\wedge T^*X \otimes \slashed S) = \wedge T^*X \widehat{\otimes} \mathrm{End}_{\mathcal{A}}(\slashed S) . 
\end{align*}
as bundles of $\Z_2$-graded algebras.  

In the computations below we abuse the notation to use the graded commutators between elements of $\Omega^*(X; \mathcal{A})$ and $\Omega^*(X; \mathrm{End}_{\mathcal{A}}(\slashed S))$, namely, we denote like 
\begin{align}\label{eq_comm_A_c}
    \{\Xi, C\} := \Xi C - (-1)^{|\Xi||C|}C\Xi
\end{align}
for homogeneous $\Xi \in \Omega^*(X; \mathrm{End}_{\mathcal{A}}(\slashed S))$ and $C \in \Omega^*(X; \mathcal{A})$, as an operator on $\Omega^*(X; \slashed S)$. 
Of course, the operator \eqref{eq_comm_A_c} is zero by definition. 

In order to formulate the notion of superconnections on $\slashed S$, we need to fix a connection $\nabla^\mathcal{A}$ on $\mathcal{A}$ which preserves the graded $*$-algebra structures. 
Recall that $\mathcal{A}$ is associated to a principal $\mathrm{Aut}_{*, \Z_2}(A)$-bundle $P$ over $X$, where $A$ denotes the typical fiber of $\mathcal{A}$. 
By a {\it connection on $\mathcal{A}$} we mean a connection $\nabla^\mathcal{A}$ on $\mathcal{A}$ induced by an $\mathrm{Aut}_{*, \Z_2}(A)$-connection on $P$. 
By Lemma \ref{lem_Lie_aut_A}, we know that the set of connections on $\mathcal{A}$ is a torsor over $\Omega^1(X; \mathcal{A}^0_{\mathrm{skew}})$, where $\mathcal{A}^0_{\mathrm{skew}}$ is the subbundle of $\mathcal{A}^0$ consisting of skew-adjoint elements. 
Explicitly, given $\nabla^\mathcal{A}$ and $C \in \Omega^1(X; \mathcal{A}^0_{\mathrm{skew}})$, the operator
\begin{align*}
    \nabla^\mathcal{A} + \{C, -\}
\end{align*}
is another connection on $\mathcal{A}$, and any two connections are related in this way. 

\begin{defn}[{$\mathcal{A}$-superconnections}]\label{def_superconn}
Let $\mathcal{A}$ be a bundle of nondegenerate simple central graded $*$-algebras over a manifold $X$, equipped with a connection $\nabla^\mathcal{A}$. 
Let $\slashed S$ be an $\mathcal{A}$-module bundle over $X$. 
An {\it $\mathcal{A}$-superconnection $\Grad^{\slashed S}$ on $\slashed S$ compatible with $\nabla^\mathcal{A}$ }is a linear map
\begin{align*}
    \Grad^{\slashed S} \colon \Omega^*(X; \slashed S) \to \Omega^*(X; \slashed S)
\end{align*}
such that, for any homogeneous element $C \in \Omega^*(X; \mathcal{A})$, we have
\begin{align}\label{eq_Cliff_superconn}
   \Grad^{\slashed S} \circ (C \cdot) - (-1)^{|C|}(C \cdot) \circ \Grad^{\slashed S} = (\nabla^{\mathcal{A}} C )\cdot, 
\end{align}
where $\Omega^*(X; \mathcal{A})$ acts on $\Omega^*(X; \slashed S)$ by \eqref{eq_action_form}. 
In particular, if $\Grad^{\slashed S}$ is a usual connection, i.e., increases the form-degree by one, we call it an {\it $\mathcal{A}$-connection}. 
\end{defn}

In the following, given an $\mathcal{A}$-superconnection $\Grad^{\slashed S}$, for an element $\Xi \in \Omega^*(X; \mathrm{End}_{\mathcal{A}}(\slashed S))$ or $\Omega^*(X; \mathcal{A})$, we define the operator $\{\Grad^{\slashed S}, \Xi\}$ on $\Omega^*(X; \slashed S)$ by
\begin{align}
    \{\Grad^{\slashed S}, \Xi\} := \Grad^{\slashed S} \circ \Xi - (-1)^{|\Xi|}\Xi \circ \Grad^{\slashed S}
\end{align}
if $\Xi$ is homogeneous, and extend this definition linearly for general $\Xi$. 
Although we are not regarding $\Grad^{\slashed S}$ as an element in any $\Z_2$-graded algebra, this convention is useful. 

As usual, we can easily show the following. 
\begin{lem}\label{lem_superconn_affine}
In the settings of Definition \ref{def_superconn}, for a fixed $ \nabla^\mathcal{A}$, the set of $\mathcal{A}$-superconnections on $\slashed S$ compatible with $ \nabla^\mathcal{A}$ is an affine space over $\Omega^*(X; \mathrm{End}_{\mathcal{A}}(\slashed S))^1$. 
\end{lem}

An $\mathcal{A}$-superconnection can be decomposed into a finite sum as
\begin{align}\label{eq_superconn_sum}
    \Grad^{\slashed S} = \nabla + \sum_j \omega_j \widehat{\otimes}\xi_j , 
\end{align}
where $\nabla$ is an $\mathcal{A}$-connection, each differential form $\omega_j$ is homogeneous and $\omega_j \widehat{\otimes}\xi_j  \in \Omega^*(X; \mathrm{End}_{\mathcal{A}}(\slashed S))^1$. 
We are mainly interested in superconnections of the form
\begin{align*}
    \Grad^{\slashed S} = \nabla + \xi_0, 
\end{align*}
for $\xi_0 \in C^\infty(X; \mathrm{End}^1_{\mathcal{A}}(\slashed S))$. 

\begin{lem}\label{lem_existense_conn}
Let us choose and fix $(\mathcal{A},  \nabla^\mathcal{A})$ in Definition \ref{def_superconn}. For any $\mathcal{A}$-module bundle $\slashed S$, there exists an $\mathcal{A}$-connection $\nabla$ compatible with $ \nabla^\mathcal{A}$. 
\end{lem}

\begin{proof}
The lemma is shown by taking local trivializations. The key is the fact that any $\mathcal{A}$-module bundle is locally of the form $\underline{S_0 \otimes W_0}$ or $\underline{S_0 \otimes W_0} \oplus \underline{S_1 \otimes W_1}$ depending on the type of fibers $A$, where $\{S_i\}_i$ is the list of the isomorphism classes of irreducible $A$-modules, and $W_i$ are some trivial $A$-modules. 
\end{proof}

It is useful to have the following local description. 
Given $X$ and $\mathcal{A}$, we can locally (say, on $U$) choose a trivialization $\mathcal{A} \simeq \underline{A}$. 
Then any $ \nabla^\mathcal{A}$ can be written as
\begin{align*}
     \nabla^\mathcal{A} = d + \{C, -\}, 
\end{align*}
for some $C \in \Omega^1(U; A^0_{\mathrm{skew}})$. 
Given any $\mathcal{A}$-module bundle $\slashed S$, by replacing $U$ smaller if necessary, we can trivialize it as $\slashed S = \underline{S}$ for some $A$-module $S$. 
Then the operator
\begin{align}
    d + C 
\end{align}
on $\Omega^*(U; S) \simeq \Omega^*(U; \slashed S)$ defines an $\mathcal{A}$-superconnection compatible with $ \nabla^\mathcal{A}$ on $U$. 
By Lemma \ref{lem_superconn_affine}, any other $\mathcal{A}$-superconnection $\Grad^{\slashed S}$ on $\slashed S$ can be written as
\begin{align}\label{eq_superconn_local}
    \Grad^{\slashed S} = d + C + B
\end{align}
for some $B \in \Omega^*(U; \mathrm{End}_\mathcal{A}(\slashed S))^1$ on $U$. 

Now we look at the effects of changes of the connection $ \nabla^\mathcal{A}$ on $\mathcal{A}$. 

\begin{lem}\label{lem_connection_S_V}
Let $X$ and $\mathcal{A}$ be as above. 
Let $\slashed S$ be an $\mathcal{A}$-module bundle, and assume we are given an $\mathcal{A}$-superconnection $\Grad^{\slashed S}$ on $\slashed S$ compatible with $ \nabla^\mathcal{A}$. 
Then, for any $C \in \Omega^1(X; \mathcal{A}^0_{\mathrm{skew}})$, $\Grad^{\slashed S} + C$ is an $\mathcal{A}$-superconnection compatible with $ \nabla^\mathcal{A} + \{C, -\}$. 
In other words, $\Omega^1(X; \mathcal{A}^0_{\mathrm{skew}})$ acts on the set of pairs $( \nabla^\mathcal{A}, \Grad^{\slashed S})$ consisting of a connection $\nabla^\mathcal{A}$ on $\mathcal{A}$ and an $\mathcal{A}$-superconnection $\Grad^{\slashed S}$ on $\slashed S$ compatible with $\nabla^\mathcal{A}$. 
Here $C \in \Omega^1(X; \mathcal{A}^0_{\mathrm{skew}})$ acts as
\begin{align}\label{eq_connection_action}
    ( \nabla^\mathcal{A}, \Grad^{\slashed S}) \mapsto ( \nabla^\mathcal{A} + \{C, -\}, \Grad^{\slashed S} + C). 
\end{align}
In particular, for any two connections $ \nabla^\mathcal{A}_0$ and $ \nabla^\mathcal{A}_1$ on $\mathcal{A}$, the correspondence \eqref{eq_connection_action} induces a bijection between the space of $\mathcal{A}$-superconnections on $\slashed S$ compatible with $ \nabla^\mathcal{A}_0$ and that with $ \nabla^\mathcal{A}_1$. 
\end{lem}
\begin{proof}
It is easily shown by the local description above. 
\end{proof}

\begin{lem}\label{lem_curv_even}
Let $X$, $\mathcal{A}$, $\slashed S$, $ \nabla^\mathcal{A}$ and $\Grad^{\slashed S}$ be as above. 
We have
\begin{align*}
    F({\Grad^{\slashed S}};  \nabla^\mathcal{A}) := (\Grad^{\slashed S})^2 - ( \nabla^\mathcal{A})^2 \in \Omega^*(X; \mathrm{End}_{\mathcal{A}}(\slashed S))^0. 
\end{align*}
We call $F({\Grad^{\slashed S}};  \nabla^\mathcal{A})$ the {\it twisting curvature} of $\Grad^{\slashed S}$. 
\end{lem}

\begin{proof}
It is enough to show that, for any $C \in \Omega^*(X; \mathcal{A})$, we have
\begin{align*}
    F(\Grad^{\slashed S};  \nabla^\mathcal{A}) \circ (C \cdot) - (C\cdot ) \circ F(\Grad^{\slashed S};  \nabla^\mathcal{A}) = 0
\end{align*}
as an operator on $\Omega^*(X; \slashed S)$. 
But this is checked directly using \eqref{eq_Cliff_superconn}. 
\end{proof}

\begin{rem}\label{rem_triv_notation}
In the case of trivial bundles in Subsection \ref{subsec_superconn_triv}, we denoted $F(\Grad) := F(\Grad; d)$, where $d$ means the trivial connection on $\underline{A}$.
But we omit the reference to $d$ when it is obvious. 
Similar remarks about the notations also apply to the other objects introduced in Subsection \ref{subsec_superconn_triv}. 
\end{rem}

In the local description \eqref{eq_superconn_local}, $F(\Grad^{\slashed S};  \nabla^\mathcal{A})$ is given by
\begin{align}\label{eq_curv_local}
    F(\Grad^{\slashed S};  \nabla^\mathcal{A}) 
    &= (d+C +B)^2 - (d+\{C, -\})^2 \\
    &= \{d+C, B\} + B^2
    = dB + B^2. \nonumber
\end{align}
Note that, in particular when $\Grad^{\slashed S}$ is an $\mathcal{A}$-{\it connection}, $F(\Grad^{\slashed S};  \nabla^\mathcal{A}) $ is not the curvature of the connection because we are subtracting $( \nabla^\mathcal{A})^2$. 
In fact, it generalizes the {\it twisting curvature} in the sense of \cite[Proposition 3.43]{BGVheatkernel}.

\begin{defn}[$\mathrm{Tr}_\mathcal{A}$]\label{def_V_str}
We define a bundle map over $X$, 
\begin{align}\label{eq_V_str}
   \mathrm{Tr}_\mathcal{A} \colon \mathrm{End}_{\mathcal{A}}(\slashed S) \to \mathrm{Ori}(\mathcal{A}),
\end{align}
by applying the linear map \eqref{eq_str_ori_A} fiberwise. 
Also extend this map left $\Omega^*(X)$-linearly and define
\begin{align}\label{eq_def_Gamma_tr}
    \mathrm{Tr}_{\mathcal{A}} \colon \Omega^*(X; \mathrm{End}_{\mathcal{A}}(\slashed S))& \to \Omega^*(X; \mathrm{Ori}(\mathcal{A})). \\
    \omega \widehat{\otimes} \xi &\mapsto \omega \otimes \mathrm{Tr}_{\mathcal{A}}(\xi) \notag. 
\end{align}
\end{defn}
Note that the map \eqref{eq_def_Gamma_tr} preserves the $\Z_2$-grading if $\mathrm{type}(\mathcal{A})$ is odd, and reverses the $\Z_2$-grading if $\mathrm{type}(\mathcal{A})$ is even. 

By Lemma \ref{lem_Gamma_trace}, the $u$-trace $\mathrm{Tr}_u$ is a supertrace on the $\Z_2$-graded algebra $\mathrm{End}_{A}(S)$. 
Using the formula \eqref{eq_def_Gamma_tr}, we easily deduce that
\begin{align}\label{eq_tr_vanish_supercommutator}
    \mathrm{Tr}_{\mathcal{A}}(\{\Xi_1, \Xi_2\}) = 0
\end{align}
for any $\Xi_1, \Xi_2 \in \Omega^*(X; \mathrm{End}_{\mathcal{A}}(\slashed S))$. 
We have the following lemma corresponding to \cite[Proposition 2]{QuillenSuperconnection} and \cite[Lemma 1.42]{BGVheatkernel}. 

\begin{lem}\label{lem_dtr}
Let $X$, $\mathcal{A}$, $\slashed S$, $ \nabla^\mathcal{A}$ and $\Grad^{\slashed S}$ be as above. 
\begin{enumerate}
    \item For any element $\Xi \in \Omega^*(X; \mathrm{End}_{\mathcal{A}}(\slashed S))$, we have
    \[ \{\Grad^{\slashed S}, \Xi\} \in \Omega^*(X; \mathrm{End}_{\mathcal{A}}(\slashed S)).\] 
    Moreover the homomorphism
    \begin{align}\label{eq_comm_Grad}
        \{\Grad^{\slashed S}, \cdot\} \colon \Omega^*(X; \mathrm{End}_{\mathcal{A}}(\slashed S)) \to \Omega^*(X; \mathrm{End}_{\mathcal{A}}(\slashed S))
    \end{align}
    is an odd graded derivation with respect to the $\Z_2$ grading given on $\Omega^*(X; \mathrm{End}_{\mathcal{A}}(\slashed S))$. 
    
    \item For any element $\Xi \in \Omega^*(X; \mathrm{End}_{\mathcal{A}}(\slashed S))$, we have
    \begin{align*}
        d \mathrm{Tr}_{\mathcal{A}} (\Xi) = \mathrm{Tr}_{\mathcal{A}}(\{\Grad^{\slashed S}, \Xi\}).
    \end{align*}
\end{enumerate}
\end{lem}

\begin{proof}
First we prove (1). 
To show the first statement and the fact that \eqref{eq_comm_Grad} is an odd homomorphism, it is enough to show that
\begin{align}
    \{\Grad^{\slashed S}, \Xi\}C - (-1)^{|C|(|\Xi|+1)}C\{\Grad^{\slashed S}, \Xi\} = 0
\end{align}
for any homogeneous $\Xi \in \Omega^*(X; \mathrm{End}_{\mathcal{A}}(\slashed S))$ and $C \in \Omega^*(X; \mathcal{A})$. 
Indeed, we have
\begin{align*}
    &\{\Grad^{\slashed S}, \Xi\}C - (-1)^{|C|(|\Xi|+1)}C\{\Grad^{\slashed S}, \Xi\} \\
    &=
    (-1)^{|C||\Xi|}\{\Grad^{\slashed S}, C\}\Xi - (-1)^{|\Xi|}\Xi \{\Grad^{\slashed S}, C\} 
    + \Grad^{\slashed S}\{\Xi, C\} - (-1)^{|\Xi|+|C|}\{\Xi, C\} \Grad^{\slashed S}\\
    &= (-1)^{|C||\Xi|}( \nabla^\mathcal{A} C)\Xi - (-1)^{|\Xi|}\Xi ( \nabla^\mathcal{A} C) \\
    &= (-1)^{|C||\Xi|}\{ \nabla^\mathcal{A} C, A\} \\
    &= 0. 
\end{align*}
The graded derivation property is obvious. 

Next we prove (2). 
We use the local description of $\mathcal{A}$-superconnections \eqref{eq_superconn_local}. 
Using the notations there, we have
\begin{align*}
    \{\Grad^{\slashed S}, \Xi\} = \{d + C + B, \Xi\} = \{d, \Xi\} + \{B, \Xi\}, 
\end{align*}
since $C \in \Omega^*(U; \mathcal{A})$. 
By \eqref{eq_tr_vanish_supercommutator}, we have
\begin{align*}
    \mathrm{Tr}_{\mathcal{A}}(\{\Grad^{\slashed S}, \Xi\}) = \mathrm{Tr}_{\mathcal{A}}(\{d, \Xi\}) = d \mathrm{Tr}_{\mathcal{A}}(\Xi), 
\end{align*}
so we get the desired result. 
\end{proof}

Using Lemma \ref{lem_dtr}, by the similar argument as in \cite{QuillenSuperconnection} and \cite{BGVheatkernel}, we get the following transgression formula, corresponding to \cite[Proposition 1.41]{BGVheatkernel}. 

\begin{prop}\label{prop_transgression}
Let $X$, $\mathcal{A}$ and $\slashed S$ be as above, and $f(z) \in \R[[z]]$. 
\begin{enumerate}
    \item For any connection $ \nabla^\mathcal{A}$ on $\mathcal{A}$ and $\mathcal{A}$-superconnection $\Grad^{\slashed S}$ on $\slashed S$ compatible with $ \nabla^\mathcal{A}$, we have
    \begin{align*}
        \mathrm{Tr}_{\mathcal{A}}(f(F(\Grad^{\slashed S};  \nabla^\mathcal{A}))) \in \Omega^*_{\mathrm{clo}}(X; \mathrm{Ori}(\mathcal{A})). 
    \end{align*}
    \item  Consider the manifold $I \times X$ and the bundles $\mathrm{pr}_X^*\mathcal{A}$ and $\mathrm{pr}_X^*\slashed S$ on it. 
    Let $\nabla^{\mathcal{A}, I}$ be a connection on $\mathrm{pr}_X^* \mathcal{A}$ and $\Grad^{\slashed S, I}$ be an $\mathrm{pr}_X^*\mathcal{A}$-superconnection on $\mathrm{pr}_X^*\slashed S$ compatible with $\nabla^{\mathcal{A}, I}$. 
    For $i = 0, 1$, we set $\nabla^{\mathcal{A}, I}|_{\{i\} \times X} =  \nabla^\mathcal{A}_i$ and $\Grad^{\slashed S, I}|_{\{i\} \times X} = \Grad^{\slashed S}_i$. Then we have
    \begin{multline*}
        d \int_{I} \mathrm{Tr}_{\mathrm{pr}_X^*\mathcal{A}} \left(f\left(F({\Grad^{\slashed S, I}}; \nabla^{\mathcal{A}, I})\right)\right) \\
        = \mathrm{Tr}_{\mathcal{A}}(f(F({\Grad^{\slashed S}_1};  \nabla^\mathcal{A}_1))) - \mathrm{Tr}_{\mathcal{A}}(f(F({\Grad^{\slashed S}_0};  \nabla^\mathcal{A}_0))). 
    \end{multline*}
    \item Let $\{ \nabla^\mathcal{A}_t\}_t$ and $\{\Grad^{\slashed S}_t\}_t$ be a smooth $1$-parameter family of connections on $\mathcal{A}$ and $\mathcal{A}$-superconnections on $\slashed S$ compatible with it. 
    We have
    \begin{align*}
    \frac{d}{dt}\mathrm{Tr}_{\mathcal{A}}(f(F({\Grad^{\slashed S}_t};  \nabla^\mathcal{A}_t))) = d\mathrm{Tr}_{\mathcal{A}}\left(\left(\frac{d\Grad^{\slashed S}_t}{dt}- \frac{d \nabla^\mathcal{A}_t}{dt}\right) f'(F({\Grad^{\slashed S}_t};  \nabla^\mathcal{A}_t)) \right). \end{align*}
\end{enumerate}
\end{prop}

\begin{proof}
For (1), using the local description \eqref{eq_superconn_local} we have
\begin{align*}
    \{\Grad^{\slashed S}, F(\Grad^{\slashed S};  \nabla^\mathcal{A})\} 
    &=- \{\Grad^{\slashed S}, ( \nabla^\mathcal{A})^2\} = -\{d+C+B, (d+\{C, -\})^2\} \\
    &= -\{B,( \nabla^\mathcal{A})^2 \} = 0, 
\end{align*}
where the last equality is because $B \in \Omega^*(X; \mathrm{End}_{\mathcal{A}}(\slashed S))$ and $( \nabla^\mathcal{A})^2 \in \Omega^*(X; \mathcal{A})$. 
The formula above and Lemma \ref{lem_dtr} lead to (1). 

For (2), we can decompose $\mathrm{Tr}_{\mathrm{pr}_X^*\mathcal{A}} \left(f\left(F({\Grad^{\slashed S, I}}; \nabla^{\mathcal{A}, I})\right)\right)$ as
\begin{align*}
    \mathrm{Tr}_{\mathrm{pr}_X^*\mathcal{A}} \left(f\left(F({\Grad^{\slashed S, I}}; \nabla^{\mathcal{A}, I})\right)\right)
    = dt \wedge \alpha(t) + \beta(t), 
\end{align*}
where $t$ is the coordinate on $I = [0, 1]$ and $\alpha(t), \beta(t) \in \Omega^*(X; \mathrm{Ori}(\mathcal{A}))$. 
We have $\beta(i) = \mathrm{Tr}_{\mathcal{A}}(f(F_{\Grad_i}))$ for $i = 0, 1$. 
By (1) we know that 
\begin{align*}
    0 =-dt \wedge d_X \alpha  + dt \wedge \frac{d\beta}{dt}   + d_X\beta, 
\end{align*}
so that 
\begin{align}\label{eq_proof_transgression}
    d_X \alpha = \frac{d\beta}{dt}. 
\end{align}
We have
\begin{align*}
    d \int_{I} \mathrm{Tr}_{\mathrm{pr}_X^*\mathcal{A}} \left(f\left(F({\Grad^{\slashed S, I}}; \nabla^{\mathcal{A}, I})\right)\right)
    &= d \int_{I} dt \wedge \alpha  \\
    &= \int_{I}dt \wedge  d_X\alpha \\
    &= \int_{I}dt \wedge \frac{d\beta}{dt} \\
    &= \beta(1) - \beta(0) \\
    &= \mathrm{Tr}_{\mathcal{A}}(f(F({\Grad^{\slashed S}_1};  \nabla^\mathcal{A}_1))) - \mathrm{Tr}_{\mathcal{A}}(f(F({\Grad^{\slashed S}_0};  \nabla^\mathcal{A}_0))), 
\end{align*}
so we get (2). 

For (3), as a special case of (2) we consider the connection $\nabla^{\mathcal{A}, I} = d_t + \nabla_t^{\mathcal{A}}$ and $\mathcal{A}$-superconnection $\Grad^{\slashed S, I} = d_t + \Grad^{\slashed S}_t$, where $d_t$ is the de Rham differential on $[0,1]$. 
We have
\begin{align*}
   (\Grad^{\slashed S, I})^2  &= (\Grad^{\slashed S}_t)^2 +dt\wedge \frac{d\Grad^{\slashed S}_t}{dt}, \\
   (\nabla^{\mathcal{A}, I})^2 &= (\nabla_t^V)^2 + dt \wedge \frac{d \nabla^\mathcal{A}_t}{dt}, 
\end{align*}
which imply
\begin{multline}\label{eq_cylinder_characteristic_form}
    \mathrm{Tr}_{\mathcal{A}}\left(f\left(F({\Grad^{\slashed S, I}}; \nabla^{\mathcal{A}, I})\right)\right)
    = \\
    dt \wedge \mathrm{Tr}_{\mathcal{A}}\left(\left(\frac{d\Grad^{\slashed S}_t}{dt}- \frac{d \nabla^\mathcal{A}_t}{dt}\right) f'(F({\Grad^{\slashed S}_t};  \nabla^\mathcal{A}_t)) \right) + \mathrm{Tr}_{\mathcal{A}}(f(F({\Grad^{\slashed S}_t};  \nabla^\mathcal{A}_t))). 
\end{multline}
The computation is parallel to \cite[pp.48--49]{BGVheatkernel}. 
By \eqref{eq_proof_transgression} and \eqref{eq_cylinder_characteristic_form}, we get (3). 
\end{proof}

The form $\mathrm{Tr}_{\mathcal{A}}\left(f(F(\Grad^{\slashed S};  \nabla^\mathcal{A}))\right)$ is invariant under the action of $\Omega^1(X; \mathcal{A}^0_{\mathrm{skew}})$ in Lemma \ref{lem_connection_S_V}, as follows. 
\begin{lem}\label{lem_change_connection}
In the above settings, let $C \in \Omega^1(X; \mathcal{A}^0_{\mathrm{skew}})$. 
For any $f(z) \in \R[[z]]$ we have
\begin{align*}
    \mathrm{Tr}_{\mathcal{A}}\left(f(F(\Grad^{\slashed S};  \nabla^\mathcal{A}))\right) = \mathrm{Tr}_{\mathcal{A}}\left(f(F(\Grad^{\slashed S} + C;  \nabla^\mathcal{A} + \{C, -\}))\right). 
\end{align*}

\end{lem}

\begin{proof}
It follows from the application of Proposition \ref{prop_transgression} (3) to the linear path from $( \nabla^\mathcal{A}, \Grad^{\slashed S})$ to $( \nabla^\mathcal{A} + \{C, -\}, \Grad^{\slashed S} + C)$. 
\end{proof}


Next we introduce the self and skew-adjointness condition on $\mathcal{A}$-superconnections. 
\begin{defn}[{Self-adjointness and Skew-adjointness of generalized Clifford superconnections}]\label{def_skew_superconn}
In the above settings, assume that $\slashed S$ is equipped with an inner product. 
Let $\Grad^{\slashed S}$ be an $\mathcal{A}$-superconnection. 
\begin{enumerate}
 \item We say that $\Grad^{\slashed S}$ is {\it self-adjoint} if, in a decomposition \eqref{eq_superconn_sum}, $\nabla^{\slashed S}$ preserves the inner product, $\xi_j$ is skew-adjoint if $|\omega_j| \equiv 1, 2 \pmod 4$ and self-adjoint if $|\omega_j| \equiv 0,3 \pmod 4$. 
\item We say that $\Grad^{\slashed S}$ is {\it skew-adjoint} if, in a decomposition \eqref{eq_superconn_sum}, $\nabla^{\slashed S}$ preserves the inner product, $\xi_j$ is skew-adjoint if $|\omega_j| \equiv 0, 1 \pmod 4$ and self-adjoint if $|\omega_j| \equiv 2,3 \pmod 4$. 
\end{enumerate}
This definition does not depend on the choice of the decomposition. 
\end{defn}

\begin{rem}\label{rem_self_skew_Dirac}
Definition \ref{def_skew_superconn} is motivated by the self/skew-adjointness of the associated Dirac operators, as follows. 
Let $X$ be a manifold with a Riemannian metric $g$. 
First we consider the Clifford algebra bundle $C(T^*X, g)$, in which the Clifford multiplications of cotangent vectors are skew-adjoint (see \eqref{eq_sign_C(V)}). 
Let $\slashed S$ be a $C(T^*X, g)$-module bundle with an inner product. 
Assume we have a {\it self-adjoint} $C(T^*X, g)$-superconnection $\Grad^{\slashed S}$ on $\slashed S$ compatible with the Levi-Civita connection. 
Then the associated {\it Dirac operator} is given by the following composition\footnote{
The complex analogue of \eqref{eq_def_Dirac} generalizes the twisted Dirac operators defined in \cite{Kahle2011}. 
In that paper, the author considers the case where we have a $\Z_2$-graded hermitian vector bundle $E$ equipped with a Quillen's superconnection $\Grad^E$, 
and Spin$^c$-Dirac operators twisted by $\Grad^E$. 
In our language, such $\Grad^E$ induces a $C(T^*X, g)$-superconnection on $\slashed S \otimes E$, where $\slashed S$ is the Spinor bundle on a Spin$^c$-manifold, and \eqref{eq_def_Dirac} recovers the one in \cite{Kahle2011}. 
}. 
\begin{align}\label{eq_def_Dirac}
    \slashed D(\Grad^{\slashed S}) \colon C^\infty(X; \slashed S) \xrightarrow{\Grad^{\slashed S}} \Omega^*(X; \slashed S) \xrightarrow{c} C^\infty(X; \slashed S). 
\end{align}
Here the second map is the Clifford multiplication. 
Then it is easily checked that the self-adjointness of $\Grad^{\slashed S}$ implies that $\slashed D(\Grad^{\slashed S})$ is formally self-adjoint. 

Next we consider the Clifford algebra bundle $(T^*X, -g)$, where now the Clifford multiplications of cotangent vectors are self-adjoint. 
Let $\slashed S$ be a $C(T^*X, -g)$-module bundle with an inner product, and assume we have a {\it skew-adjoint} $C(T^*X, -g)$-superconnection $\Grad^{\slashed S}$ on $\slashed S$ compatible with the Levi-Civita connection. 
Then the associated Dirac operator $\slashed D(\Grad^{\slashed S})$ is defined by the same formula as \eqref{eq_def_Dirac}. 
In this case, the skew-adjointness of $\Grad^{\slashed S}$ implies that $\slashed D(\Grad^{\slashed S})$ is formally skew-adjoint. 
We also note that the {\it massive Dirac operator} in \eqref{eq_massive_dirac} can be written in terms of this construction. 
\end{rem}

\begin{prop}\label{prop_chracteristic_form_mod4}
Assume that an $\mathcal{A}$-module bundle $\slashed S$ is equipped with an inner product and an $\mathcal{A}$-superconnection $\Grad^{\slashed S}$ compatible with $\nabla^\mathcal{A}$. 
Let $f(z) \in \R[[z]]$. 
\begin{enumerate}
    \item If $\Grad^{\slashed S}$ is self-adjoint, we have
    \begin{align*}
        \mathrm{Tr}_{\mathcal{A}}(f(F(\Grad^{\slashed S};  \nabla^\mathcal{A}))) \in  \Omega_{\mathrm{clo}}^{4\Z +\mathrm{type}(\mathcal{A})+1}(X; \mathrm{Ori}(\mathcal{A})). 
    \end{align*}
    \item If $\Grad^{\slashed S}$ is skew-adjoint, we have
\begin{align*}
    \mathrm{Tr}_{\mathcal{A}}(f(F(\Grad^{\slashed S};  \nabla^\mathcal{A}))) \in  \Omega_{\mathrm{clo}}^{4\Z -\mathrm{type}(\mathcal{A})-1}(X; \mathrm{Ori}(\mathcal{A})). 
\end{align*}
\end{enumerate}
\end{prop}

The proof is by a direct computation, which we give in Subsection \ref{subsec_proof_mod4}. 

\begin{defn}[{$\mathrm{Ph}_{\lself/ \lskew}(\Grad^{\slashed S};  \nabla^\mathcal{A})$, $\mathrm{CS}_{\lself/ \lskew}( \Grad^{\slashed S, I};  \nabla^\mathcal{A})$}]\label{def_Ch}
Let $X$, $\mathcal{A}$, $\nabla^\mathcal{A}$ and $\slashed S$ be as above. 
\begin{enumerate}
    \item For a self-adjoint $\mathcal{A}$-superconnection $\Grad^{\slashed S}$ on $\slashed S$ compatible with $ \nabla^\mathcal{A}$, we define its {\it Pontryagin character form} by
    \begin{align}
    \mathrm{Ph}_\lself(\Grad^{\slashed S};  \nabla^\mathcal{A}) :=   \mathrm{Tr}_{\mathcal{A}}(e^{-F(\Grad^{\slashed S};  \nabla^\mathcal{A})})
    \in   \Omega_{\mathrm{clo}}^{4\Z +\mathrm{type}(\mathcal{A})+1}(X; \mathrm{Ori}(\mathcal{A})). 
\end{align}
For a self-adjoint $\mathrm{pr}_X^*\mathcal{A}$-superconnection $\Grad^{\slashed S, I}$ on the $\mathrm{pr}_X^*\mathcal{A}$-module bundle $\mathrm{pr}_X^*\slashed S$ over $I \times X$ compatible with $\mathrm{pr}_X^*  \nabla^\mathcal{A}$, we define its {\it Chern-Simons form} by
\begin{align}
    \mathrm{CS}_\lself( \Grad^{\slashed S, I};  \nabla^\mathcal{A}) := \int_{I}\mathrm{Ph}_\lself\left(\Grad^{\slashed S, I}; \mathrm{pr}_X^*  \nabla^\mathcal{A}\right)  \in  \Omega^{4\Z +\mathrm{type}(\mathcal{A})}(X; \mathrm{Ori}(\mathcal{A})). 
\end{align}

\item In the same settings, if $\Grad^{\slashed S}$ is skew-adjoint, we define its {\it Pontryagin character form} by
\begin{align}\label{eq_Ph_skew}
    \mathrm{Ph}_\lskew(\Grad^{\slashed S};  \nabla^\mathcal{A}) :=   \mathrm{Tr}_{\mathcal{A}}(e^{F(\Grad^{\slashed S};  \nabla^\mathcal{A})})
    \in  \Omega_{\mathrm{clo}}^{4\Z -\mathrm{type}(\mathcal{A})-1}(X; \mathrm{Ori}(\mathcal{A})). 
\end{align}
If $\Grad^{\slashed S, I}$ is skew-adjoint, we define its {\it Chern-Simons form} by
\begin{align}
    \mathrm{CS}_\lskew( \Grad^{\slashed S, I};  \nabla^\mathcal{A}) := \int_{I}\mathrm{Ph}_\lskew\left(\Grad^{\slashed S, I}; \mathrm{pr}_X^*  \nabla^\mathcal{A}\right) \in   \Omega^{4\Z -\mathrm{type}(\mathcal{A})-2}(X; \mathrm{Ori}(\mathcal{A})). 
\end{align}
\end{enumerate}
\end{defn}
By Proposition \ref{prop_transgression}, setting $\Grad_i^{\slashed S}:= \Grad^{\slashed S, I}|_{\{i\}\times X}$ we have
\begin{align}\label{eq_transgression_Ph}
    \Ph_{\lself / \lskew}(\Grad_1^{\slashed S};  \nabla^\mathcal{A}) - \Ph_{\lself / \lskew}(\Grad_0^{\slashed S};  \nabla^\mathcal{A})
    = d\CS_{\lself / \lskew}( \Grad^{\slashed S, I};  \nabla^\mathcal{A}). 
\end{align}

By Lemma \ref{lem_change_connection}, we get the following. 

\begin{lem}\label{lem_change_connection_Ch}
For any $C \in \Omega^1(X; \mathcal{A}^0_{\mathrm{skew}})$, we have
\begin{align*}
\mathrm{Ph}_{\lself / \lskew}(\Grad^{\slashed S};  \nabla^\mathcal{A}) = \mathrm{Ph}_{\lself / \lskew}(\Grad^{\slashed S} + C;  \nabla^\mathcal{A} + \{C, -\}). 
\end{align*}
Moreover, if we have an $\mathcal{A}$-superconnection $\Grad^{\slashed S, I}$ as in Definition \ref{def_Ch}, we have
\begin{align*}
    \mathrm{CS}_{\lself / \lskew}(\Grad^{\slashed S, I};  \nabla^\mathcal{A}) = \mathrm{CS}_{\lself / \lskew}(\Grad^{\slashed S, I} + \mathrm{pr}_X^*C;  \nabla^\mathcal{A} + \{C, -\}).  
\end{align*}
\end{lem}

\subsubsection{The relation with the superconnection formalism by Quillen \cite{QuillenSuperconnection}}\label{subsubsec_Quillen}
Here we explain how to recover the Quillen's superconnection formalism \cite{QuillenSuperconnection} as a special case of the formalisms developed in this subsection. 

First of all, Quillen works in the $\C$-linear setting. 
Thus, precisely speaking we recover the $\R$-linear version of Quillen's superconnection formalism (by the obvious restriction of the coefficients $\R \subset \C$) as a special case of the formalisms in this subsection. 
As we will explain in Subsection \ref{subsec_complex_superconn}, the formalism in this subsection can be modified to the $\C$-linear setting in a straightforward way. 
Based on the modified formalism, we can generalize the following discussion in a straightforward manner to recover the formalism in \cite{QuillenSuperconnection}. 

Recall that in \cite{QuillenSuperconnection} the superconnections are considered for $\Z_2$-graded vector bundles (which we call the {\it even superconnection formalism}) and ungraded vector bundles (which we call the {\it odd superconnection formalism}). 

For the former, the relation with our formulation is very simple. 
The even superconnection formalism is our superconnection formalism in the case where $\mathcal{A}$ is the trivial bundle with fiber $Cl_{0, 1}$. 
Indeed, a $\Z_2$-grading is equivalent to a $Cl_{0, 1}$-module structure, and the trace $\mathrm{Tr}_u$ for a volume element $u \in Cl_{0, 1}$ recovers the usual supertrace. 

Now we explain the latter. 
The odd superconnection formalism can be regarded as our superconnection formalism in the case where $\mathcal{A}$ is the trivial bundle with fiber $Cl_{1, 1}$, as follows. 
In \cite[Section 5]{QuillenSuperconnection}, we start with an ungraded vector bundle $\slashed S$, form a vector bundle $\slashed S' :=  Cl_{0, 1} \otimes \slashed S$, regard $\slashed S'$ as a right $Cl_{0, 1}$-module, and consider the bundle $\End_\sigma(\slashed S')$ of endomorphisms on $\slashed S'$ commuting with the right $Cl_{0, 1}$-action, where $\sigma \in Cl_{0, 1}$ denotes a generator. 
We have
\begin{align*}
    \End_\sigma(\slashed S') = \{\id \otimes a +  (\sigma \cdot) \otimes b\ | \ a, b \in \End({\slashed S}) \}. 
\end{align*}
Then the {\it odd supertrace} is defined as
\begin{align}\label{eq_odd_supertrace}
    \mathrm{tr}_\sigma \colon \Omega^*(X;  \End_\sigma(\slashed S')) &\to \Omega^*(X) \\
    \id \otimes a +  (\sigma \cdot) \otimes b &\mapsto \mathrm{Tr}_{\slashed S}(b).  \notag
\end{align}

Now we turn to our formalism. 
Recall that in the generalized Clifford superconnection formalism of this subsection, we required that $\mathcal{A}$ is a bundle of nondegenerete algebras. 
Thus, for an ungraded vector bundle ${\slashed S}$ regarded as an $\R = Cl_{0, 0}$-module bundle, we cannot apply the formalism directly. 
In such a case, we tensor a negligible module to produce a nondegenerate algebra, use the isomorphism in Lemma \ref{lem_periodicity_negligible_bundle} (3), and apply the superconnection formalism to the latter. 
Actually, the above process of producing $\mathrm{tr}_\sigma$ on $\End_\sigma(\slashed S')$ can be regarded as performing this procedure. 
Indeed, consider the $\Z_2$-graded vector space $V = V^0 \oplus V^1 = \R \oplus \R$ and the associated negligible module $\End(V) \simeq Cl_{1, 1}$. 
The grading operator $\gamma_V$ on $V$ is a volume element for $\End(V)$. 
Then, given an ungraded vector bundle ${\slashed S}$, as in Lemma \ref{lem_periodicity_negligible_bundle} (3) we consider the $\End(V)$-module bundle $V \otimes {\slashed S}$. 
The bundle $V \otimes {\slashed S}$ corresponds to the above $\slashed S' =Cl_{0, 1} \otimes \slashed S$, and the action by $\gamma_V \otimes \id_{\slashed S}$ corresponds to $(\sigma  \cdot) \otimes \id_{\slashed S}$\footnote{
Be careful that in this correspondence the $\Z_2$-grading on $V$ does not correspond to the $\Z_2$-grading on $Cl_{0, 1}$. 
In our generalized Clifford superconnection formalism, as an $\End(V)$-module we just regard $V \otimes {\slashed S}$ as an ungraded module, so this difference is irrelevant. 
}. 
We have
\begin{align}
    \End^0_{\End(V)} (V \otimes \slashed S) &= \{\id_V \otimes a \ | \ a \in \End(\slashed S)\}, \\
    \End^1_{\End(V)} (V \otimes \slashed S) &= \{\gamma_V \otimes b \ | \ b \in \End(\slashed S)\}, 
\end{align}
and the $\mathrm{Tr}_{\gamma_V}$ in Definition \ref{def_u_str} is given by
\begin{align}
    \mathrm{Tr}_{\gamma_V} \colon \Omega^*(X;  \End^1_{\End(V)} (V \otimes \slashed S)) &\to \Omega^*(X) \\
    \id_V \otimes a +  \gamma_V \otimes b &\mapsto \mathrm{Tr}_{\slashed S}(b).  \notag
\end{align}
Thus we see that we recover the odd supertrace $\mathrm{tr}_\sigma$ in \eqref{eq_odd_supertrace}, as desired.

\subsection{The Pontryagin character forms for gradations}\label{subsec_Ph_form}

Let $X$, $\mathcal{A}$, $\nabla^{\mathcal{A}}$ and $\slashed S$ with an inner product be as in the last subsection. 
Note that we are assuming that the fibers of $\mathcal{A}$ are nondegenerate. 
Let $\nabla^{\slashed S}$ be a self-adjoint $\mathcal{A}$-{\it connection} (i.e., one  which increases the form degree by one) on $\slashed S$ compatible with $\nabla^{\mathcal{A}}$. 
In this setting, given a smooth invertible section $h \in C^\infty( X ; \mathrm{Self}_{\mathcal{A}}^*(\slashed S))$, we are going to define its {\it Pontryagin character form} $\mathrm{Ph}_\lself(h; \nabla^{\mathcal{A}}, \nabla^{\slashed S})$ (Definition \ref{def_Ph_m}), generalizing Definition \ref{def_Ph_m_triv}. 

Consider the manifold $(0, \infty)_t \times X$, and the bundles  $\mathrm{pr}_X^*\mathcal{A}$ and $\mathrm{pr}_X^*\slashed S$ on it. 
We consider the following $\mathrm{pr}_X^*\mathcal{A}$-superconnection on $\mathrm{pr}_X^*\slashed S$ compatible with $\mathrm{pr}_X^*\nabla^{\mathcal{A}}$, 
\begin{align}\label{eq_massive_conn}
     \mathrm{pr}_X^*\nabla^{\slashed S} + th, 
\end{align}
where $t$ denotes the corrdinate on $(0, \infty)$. 
Then \eqref{eq_massive_conn} is self-adjoint in the sense of Definition \ref{def_skew_superconn}, so we have
\begin{align}
   \mathrm{Ph}_\lself(\mathrm{pr}_X^*\nabla^{\slashed S} + th; \mathrm{pr}_X^*\nabla^{\mathcal{A}}) \in  \Omega_{\mathrm{clo}}^{4\Z+\mathrm{type}(\mathcal{A})+1}((0, \infty) \times X ; \mathrm{pr}_X^*\mathrm{Ori}(\mathcal{A})). 
\end{align}

\begin{lem}\label{lem_Ph_convergence}
In the above settings, the following integration converges pointwise, 
\begin{align}\label{eq_lem_Ph_convergence_1}
    \int_{(0, \infty)} \mathrm{Ph}_\lself(\mathrm{pr}_X^*\nabla^{\slashed S} + th; \mathrm{pr}_X^*\nabla^{\mathcal{A}}) \in \Omega^{4\Z +\mathrm{type}(\mathcal{A})}(X; \mathrm{Ori}(\mathcal{A})). 
\end{align}
We have
\begin{align}\label{eq_lem_Ph_convergence_2}
    d\int_{(0, \infty)} \mathrm{Ph}_\lself(\mathrm{pr}_X^*\nabla^{\slashed S} + th; \mathrm{pr}_X^*\nabla^{\mathcal{A}}) = -\mathrm{Ph}_\lself(\nabla^{\slashed S}; \nabla^{\mathcal{A}}). 
\end{align}
\end{lem}

\begin{proof}
Working locally, we may trivialize $\mathcal{A}$ and $\slashed S$. 
By Lemma \ref{lem_change_connection_Ch} and the fact that under the action \eqref{eq_connection_action} the superconnection \eqref{eq_massive_conn} transforms as $\mathrm{pr}_X^*\nabla^{\slashed S} + th \mapsto \mathrm{pr}_X^*(\nabla^{\slashed S} + C) + th$, we may assume $\nabla^{\mathcal{A}} = d$. 
In this case we have $(\mathrm{pr}_X^*\nabla^{\slashed S} + th)^2 \in \Omega^*((0, \infty) \times X; \mathrm{End}_{\mathcal{A}}(\slashed S))$, and $\mathrm{Ph}_\lself(\mathrm{pr}_X^*\nabla^{\slashed S} + th) = \mathrm{Tr}_{\mathcal{A}}(e^{-(\mathrm{pr}_X^*\nabla^{\slashed S} + th)^2})$. 
Fix any Riemannian metric on $X$. 
It is enough to show that, for each $x \in X$, there exists $c > 0$ and $t_0 \in (0, \infty)$ such that 
\begin{align*}
    \|e^{-(\mathrm{pr}_X^*\nabla^{\slashed S} + th)^2}\| < e^{-t^2c}
\end{align*}
on $[t_0, \infty) \times \{x\}$. 
We have
\begin{align*}
    (\mathrm{pr}_X^*\nabla^{\slashed S} + th)^2
    = h + (\nabla^{\slashed S})^2 + t\{\nabla^{\slashed S}, h\} + t^2 h^2. 
\end{align*}
Since $h(x)$ is self-adjoint and invertible, there exists a constant $a > 0$ such that $h(x)^2  >a \cdot \mathrm{Id}$. 
From this we get the convergence of \eqref{eq_lem_Ph_convergence_1}. 
\eqref{eq_lem_Ph_convergence_2} follows from \eqref{eq_transgression_Ph} and the convergence $\lim_{t \to \infty}\mathrm{Ph}_\lself(\nabla^{\slashed S} + th; \nabla^{\mathcal{A}}) = 0$ shown similarly. 
\end{proof}

By Lemma \ref{lem_Ph_convergence}, we get
\begin{align*}
   -\pi^{-1/2}\mathcal{R} \circ \int_{(0, \infty)} \mathrm{Ph}_\lself(\mathrm{pr}_X^*\nabla^{\slashed S} + th; \mathrm{pr}_X^*\nabla^{\mathcal{A}}) \in \Omega^{4\Z + \type(\mathcal{A})}(X; \mathrm{Ori}(\mathcal{A})). 
\end{align*}
Here $\mathcal{R}$ is the endomorphism on differential forms defined in \eqref{eq_def_R}. 

\begin{defn}[{$\mathrm{Ph}_\lself(h; \nabla^{\mathcal{A}}, \nabla^{\slashed S})$, $\mathrm{CS}_\lself(h_I; \nabla^{\mathcal{A}}, \nabla^{\slashed S})$}]\label{def_Ph_m}
Let $X$, $\mathcal{A}$, $\slashed S$, $\nabla^{\mathcal{A}}$ and $\nabla^{\slashed S}$ be as above. 
\begin{enumerate}
    \item For any smooth invertible section $h \in C^\infty( X ; \mathrm{Self}_{\mathcal{A}}^*(\slashed S))$, we define
\begin{multline}
    \mathrm{Ph}_\lself(h; \nabla^{\mathcal{A}}, \nabla^{\slashed S}) :=   -\pi^{-1/2}\mathcal{R}  \circ\int_{(0, \infty)} \mathrm{Ph}_\lself(\mathrm{pr}_X^*\nabla^{\slashed S} + th; \mathrm{pr}_X^*\nabla^{\mathcal{A}}) \\
      \in \Omega^{4\Z + \type(\mathcal{A})}(X; \mathrm{Ori}(\mathcal{A})). 
\end{multline}
The convergence of the integration follows from Lemma \ref{lem_Ph_convergence}. 
Note that the form above is not a closed one (see \eqref{eq_lem_Ph_convergence_2}). 
\item  
For any smooth invertible section $h_I \in C^\infty(I \times X ; \mathrm{pr}_X^*\mathrm{Self}_{\mathcal{A}}^*(\slashed S))$, we define
\begin{multline} 
    \mathrm{CS}_\lself(h_I; \nabla^{\mathcal{A}}, \nabla^{\slashed S}) := \int_{I}\mathrm{Ph}_\lself\left(h_I; \mathrm{pr}_X^*\nabla^{\mathcal{A}}, \mathrm{pr}_X^*\nabla^{\slashed S} \right)  \\
    \in \Omega^{4\Z + \type(\mathcal{A})-1}(X; \mathrm{Ori}(\mathcal{A})). 
\end{multline}
\end{enumerate}
\end{defn}

\begin{lem}\label{lem_Ph=dCS}
We have
\begin{multline}\label{eq_Ph=dCS}
     d\mathrm{CS}_\lself\left(h_{I}; \nabla^{\mathcal{A}}, \nabla^{\slashed S}\right) \\
     =
     \mathrm{Ph}_\lself\left(h_I|_{\{1\} \times X}; \nabla^{\mathcal{A}}, \nabla^{\slashed S}\right) -  \mathrm{Ph}_\lself\left(h_I|_{\{0\} \times X}; \nabla^{\mathcal{A}}, \nabla^{\slashed S}\right). 
\end{multline}
\end{lem}

\begin{proof}
Based on \eqref{eq_lem_Ph_convergence_2}, the difference of the left and the right hand side of \eqref{eq_Ph=dCS} is computed as 
\begin{align*}
    \int_I d \mathrm{Ph}_\lself\left(h_I; \mathrm{pr}_X^*\nabla^{\mathcal{A}}, \mathrm{pr}_X^*\nabla^{\slashed S} \right)  &= (\mbox{const}) \cdot \mathcal{R} \circ \int_I\mathrm{Ph}_\lself(\mathrm{pr}_X^*\nabla^{\slashed S}; \mathrm{pr}_X^*\nabla^{\mathcal{A}}) \\
    &= (\mbox{const}) \cdot \mathcal{R} \circ \int_I\mathrm{pr}_X^*\mathrm{Ph}_\lself(\nabla^{\slashed S}; \nabla^{\mathcal{A}}) = 0. 
\end{align*}
\end{proof}

Note that, in Definition \ref{def_Ph_m} (1), if we have a submanifold $Y \subset X$ and two smooth sections $h_0, h_1\in C^\infty( X ; \mathrm{Self}_{\mathcal{A}}^*(\slashed S))$ such that $h_0|_Y = h_1|_Y$, we have
\begin{align}\label{eq_difference_Ph_closed}
    \mathrm{Ph}_\lself(h_1; \nabla^{\mathcal{A}}, \nabla^{\slashed S}) - \mathrm{Ph}_\lself(h_0; \nabla^{\mathcal{A}}, \nabla^{\slashed S}) \in \Omega^{4\Z + \type(\mathcal{A})}_{\mathrm{clo}}(X, Y; \mathrm{Ori}(\mathcal{A})). 
\end{align}
The closedness follows from \eqref{eq_lem_Ph_convergence_2}. 
Similarly, in Definition \ref{def_Ph_m} (2), if $h_I|_{I \times Y}$ is constant in the $I$-direction, we have
$\mathrm{CS}_\lself(h_I; \nabla^{\mathcal{A}}, \nabla^{\slashed S})|_Y = 0$, so that we have
\begin{align}\label{eq_CS_relative}
     \mathrm{CS}_\lself(h_I; \nabla^{\mathcal{A}}, \nabla^{\slashed S}) \in \Omega^{4\Z + \type(\mathcal{A})-1}(X, Y; \mathrm{Ori}(\mathcal{A})). 
\end{align}
Moreover, we can also easily show that the above Chern-Simons form modulo exact forms only depends on the homotopy class of the homotopy $h_I$, that is, if we have two choices of homotopies $h_I$ and $h'_I$ which are homotopic in the obvious sense, then we have
\begin{align}\label{eq_CS_mod_exact}
    \mathrm{CS}_\lself(h_I; \nabla^{\mathcal{A}}, \nabla^{\slashed S}) - \mathrm{CS}_\lself(h'_I; \nabla^{\mathcal{A}}, \nabla^{\slashed S})  \in \mathrm{Im}(d). 
\end{align}

Now we investigate into properties of $\mathrm{Ph}_\lself(h; \nabla^{\mathcal{A}}, \nabla^{\slashed S})$. 
Our goal is Theorem \ref{thm_Ph=Ph_top_twisted}, which says that $\mathrm{Ph}_\lself(h; \nabla^{\mathcal{A}}, \nabla^{\slashed S})$ realizes the twisted topological Pontryagin character homomorphism in Subsection \ref{subsec_twisted_Ph_top}. 
It generalizes Corollary \ref{cor_Ph=Ph_top}. 

First, by Lemma \ref{lem_change_connection_Ch} we get the following. 
\begin{lem}\label{lem_change_connection_Ph}
For any $C \in \Omega^1(X; \mathcal{A}^0_{\mathrm{skew}})$,
we have
\begin{align*}
    \mathrm{Ph}_\lself(h; \nabla^{\mathcal{A}}, \nabla^{\slashed S}) &= \mathrm{Ph}_\lself(h; \nabla^{\mathcal{A}} + C, \nabla^{\slashed S} + C),  \\
     \mathrm{CS}_\lself(h_I; \nabla^{\mathcal{A}}, \nabla^{\slashed S}) &= \mathrm{CS}_\lself(h_I; \nabla^{\mathcal{A}} + C, \nabla^{\slashed S} + C). 
\end{align*}
\end{lem}

Next, we show that the cohomology class of \eqref{eq_difference_Ph_closed} does not depend on the connections, as follows. 
\begin{lem}\label{lem_Ph_indep_conn}
Let $X$, $\mathcal{A}$ and $\slashed S$ be as above. 
Assume we have a submanifold $Y \subset X$ and two smooth sections $h_0, h_1\in C^\infty( X ; \mathrm{Self}_{\mathcal{A}}^*(\slashed S))$ such that $h_0|_Y = h_1|_Y$. 
Then the cohomology class of the difference element \eqref{eq_difference_Ph_closed}, 
\begin{multline}\label{eq_Ph_indep_conn}
    \mathrm{Rham}\left(\mathrm{Ph}_\lself(h_1; \nabla^{\mathcal{A}}, \nabla^{\slashed S}) - \mathrm{Ph}_\lself(h_0; \nabla^{\mathcal{A}}, \nabla^{\slashed S}) \right) \\
    \in H^{4\Z + {\mathrm{type}}(\mathcal{A})}(X, Y; \mathrm{Ori}(\mathcal{A})),  
\end{multline}
does not depend on the compatible pair $(\nabla^{\mathcal{A}}, \nabla^{\slashed S})$. 
\end{lem}
\begin{proof}
Any two compatible pairs $(\nabla_0^V, \nabla_0^{\slashed S})$ and $(\nabla_1^V, \nabla_1^{\slashed S})$ are connected by a homotopy. 
As always, we can transgress the difference of the forms \eqref{eq_difference_Ph_closed} corresponding to those connections, by integrating the form \eqref{eq_difference_Ph_closed} constructed by using the homotopy in the $I$-direction. 
\end{proof}

Recall that twists in $KO_+$ by negligible bundles are trivial (Fact \ref{fact_isom_negligible}). 
Here we show the corresponding statement for the Pontryagin character forms. 
Let $E = E^0 \oplus E^1$ be a smooth $\Z_2$-graded vector bundle over $X$ equipped with a fiberwise positive definite inner product on each $E^i$, associating the negligible bundle $\End(E)$. 
Recall Lemma \ref{lem_periodicity_negligible_bundle}. 
Given any other $\mathcal{A}$, for an $\mathcal{A}$-module bundle $\slashed S$ with an inner product we get an $\End(E) \widehat{\otimes} \mathcal{A}$-module bundle $E \otimes\slashed S$, and we have a canonical isomorphism $\psi^E \colon \Self^*_{\mathcal{A}}(\slashed S) \simeq \Self^*_{\End(E) \widehat{\otimes} \mathcal{A}}(E \otimes\slashed S)$.

\begin{lem}\label{lem_Ph_spin}
In the above setting,
let $\nabla^\mathcal{A}$ be a connection on $\mathcal{A}$ and $\nabla^{\slashed S}$ be an $\mathcal{A}$-connection on $\slashed S$ compatible with $\nabla^\mathcal{A}$. 
Let $\nabla^E$ be any orthogonal connection on $E$, and we denote by $\nabla^{\mathrm{End}(E)}$ the induced connection on $\mathrm{End}(E)$. 
Then $\nabla^{E} \otimes 1 + 1 \otimes \nabla^{\slashed S}$ is an $\End(E) \widehat{\otimes}\mathcal{A}$-connection on $\End(E) \otimes \slashed S$ compatible with $\nabla^{\mathrm{End}(E)} \otimes 1 + 1 \otimes \nabla^{\mathcal{A}}$. Furthermore, for any $h \in C^\infty(X; \Self^*_{\mathcal{A}}(\slashed S))$, we have
\begin{multline}
    \mathrm{Ph}_\lself(h; \nabla^{\mathcal{A}}, \nabla^{\slashed S}) \\ 
    = \mathrm{Ph}_\lself(\psi^E(h); \nabla^{\mathrm{End}(E)} \otimes 1 + 1 \otimes \nabla^{\mathcal{A}}, \nabla^{E} \otimes 1 + 1 \otimes \nabla^{\slashed S}). 
\end{multline}
Here we are using $\mathrm{Ori}(\mathcal{A}) \simeq \mathrm{Ori}(\End(E) \widehat{\otimes}\mathcal{A})$ given by $u \mapsto \gamma_E \widehat{\otimes}u$. 
\end{lem}
\begin{proof}
Locally we can trivialize $E$ and write $\nabla^E = d + C^E$ for some $\End(E)^0$-valued one-form $C^E$. 
Then we have $\nabla^{\End(E)} = d + \{C^E, -\}$, so the result follows from Lemma \ref{lem_change_connection_Ph} and Lemma \ref{lem_tr_u_negligible}. 
\end{proof}

Now we turn to the relation with the twisted topological Pontryagin character homomorphism in Subsection \ref{subsec_twisted_Ph_top}. 
For this, we use the explanation of the twisted $KO$-theory in Subsection \ref{subsec_twist} in terms of $\Self_{A}^*(-)$-bundles.  

Recall that in the untwisted case, Pontryagin character forms can be regarded as the pullbacks of the universal Pontryagin form \eqref{eq_univ_Ph}. 
Here in the twisted case, we also have the corresponding universal form. 
Let $X$, $\mathcal{A}$ and $\slashed S$ be as before, and consider the fiber bundle $\pi \colon \Self_{\mathcal{A}}^*(\slashed S) \to X$. 
We consider the pullback of $\mathcal{A}$ and $\slashed S$ on $\Self_{\mathcal{A}}^*(\slashed S)$. 
Then we have the {\it universal gradation}, which is given by the tautological section, 
\begin{align*}
    h^{\mathcal{A}, \slashed S}_{\mathrm{univ}} \in C^\infty(\Self_{\mathcal{A}}^*(\slashed S); \Self_{\pi^*\mathcal{A}}^*(\pi^*\slashed S)). 
\end{align*}
Given a pair $(\nabla^{\mathcal{A}}, \nabla^{\slashed S})$ as before, we pull them back by $\pi$, and get the {\it universal Pontryagin form}, 
\begin{align}\label{eq_univ_Ph_twisted}
    \mathrm{Ph}_\lself\left(h^{\mathcal{A}, \slashed S}_{\mathrm{univ}}; \pi^* \nabla^{\mathcal{A}}, \pi^* \nabla^{\slashed S}\right) \in \Omega^{4\Z + \type(\mathcal{A})}(\Self_{\mathcal{A}}^*(\slashed S); \mathrm{Ori}(\mathcal{A})), 
\end{align}
where we abuse the notation to write the pullback of $\mathrm{Ori}(\mathcal{A})$ by the same symbol. 
Note that it is not closed \eqref{eq_lem_Ph_convergence_2}. 
It is universal in the following sense. 
Any section $h \in  C^\infty( X ; \mathrm{Self}_{\mathcal{A}}^*(\slashed S))$ is in particular a smooth map from $X$ to $\Self_{\mathcal{A}}^*(\slashed S)$, and the gradation $h$ is the pullback of the universal gradation $h^{\mathcal{A}, \slashed S}_{\mathrm{univ}}$ by this map. 
By the obvious naturality of the Pontryagin character forms, we see that
\begin{align*}
   \mathrm{Ph}_\lself(h; \nabla^{\mathcal{A}}, \nabla^{\slashed S}) = h^* \mathrm{Ph}_\lself\left(h^{\mathcal{A}, \slashed S}_{\mathrm{univ}}; \pi^* \nabla^{\mathcal{A}}, \pi^* \nabla^{\slashed S}\right). 
\end{align*}
Recall that, as explained in Subsection \ref{subsec_twist}, twisted $KO$-theory groups are given in terms of sections of bundles of based spaces. 
As explained around \eqref{eq_KO_m=KO_twisted}, it corresponds to considering pairs $(\slashed S, h_0)$ of an $\mathcal{A}$-module bundle and a section $h_0 \in C^\infty(X; \mathrm{Self}_{\mathcal{A}}^*(\slashed S))$. 
For such a pair, we get a closed form, 
\begin{multline}\label{eq_twisted_Ph_univ_form}
    \mathrm{Ph}_\lself\left(h^{\mathcal{A}, \slashed S}_{\mathrm{univ}}; \pi^* \nabla^{\mathcal{A}}, \pi^* \nabla^{\slashed S}\right) - \mathrm{Ph}_\lself\left(\pi^*h_0; \pi^* \nabla^{\mathcal{A}}, \pi^* \nabla^{\slashed S}\right) \\
    \in \Omega_{\mathrm{clo}}^{4\Z + \type(\mathcal{A})}(\Self_{\mathcal{A}}^*(\slashed S), \{h_0\}; \mathrm{Ori}(\mathcal{A}) ). 
\end{multline}
The cohomology class of the form \eqref{eq_twisted_Ph_univ_form} is independent of $(\nabla^{\mathcal{A}}, \nabla^{\slashed S})$ by Lemma \ref{lem_Ph_indep_conn}, which we denote by
\begin{align}\label{eq_twisted_Ph_univ_rham}
    \mathrm{Ph}_\lself(\mathcal{A}, \slashed S, h_0) \in H^{4\Z +{\mathrm{type}}(\mathcal{A})}(\Self_{\mathcal{A}}^*(\slashed S), \{h_0\}; \mathrm{Ori}(\mathcal{A})). 
\end{align}

\begin{thm}\label{thm_Ph=Ph_top_twisted}
Let $X$ be connected. 
The class \eqref{eq_twisted_Ph_univ_rham} is the pullback of the universal topological Pontryagin character class
\begin{align*}
    [\mathrm{Ph}_{\mathrm{top}}] \in H^{4\Z + {\mathrm{type}}(\mathcal{A})}((\mathcal{KO}_{\mathrm{type}(\mathcal{A})})_\iota, \{\beta\};E\Z_2 \times_{\Z_2} \R)
\end{align*}
in \eqref{eq_univ_twisted_Ph} under the map given as the composition
\begin{align*}
    (\Self_{\mathcal{A}}^*(\slashed S), \{h_0\}) \hookrightarrow (\mathcal{KO}_{\mathcal{A}}, \{\beta\}) \to ((\mathcal{KO}_{\mathrm{type}(\mathcal{A})})_{\iota}, \{\beta\}), 
\end{align*}
where the first inclusion map is \eqref{eq_emb_Skew_bundle} and the second map comes from the composition of the classifying map $X \to B\mathrm{Aut}_{*, \Z_2}(A)$ of $\mathcal{A}$ and the map $(w_1, w_2) \colon B\mathrm{Aut}_{*, \Z_2}(A) \to B\mathrm{O}\langle 0, 1, 2 \rangle$. 
\end{thm}
\begin{proof}
We use the notations in Subsection \ref{subsec_twisted_Ph_top}. 
By the naturality of the Pontryagin character forms, the class \eqref{eq_twisted_Ph_univ_rham} gives an element
\begin{align}\label{eq_Ph_V_rham}
\{\mathrm{Ph}_\lself(V, \slashed S, h_0)\}_{(S, h_0)} &\in \varprojlim_{(\slashed S, h_0)}H^{4\Z + \type(\mathcal{A})}\left(\Self_{\mathcal{A}}^*(\slashed S), \{h_0\}; \mathrm{Ori}(\mathcal{A})  \right)\\
&\simeq H^{4\Z + \type(\mathcal{A})}(\mathcal{KO}_{\mathcal{A}}, \{\beta\}; \mathrm{Ori}(\mathcal{A})  ). \notag
\end{align}
Here the inverse limit runs over the pairs $(\slashed S, h_0)$ of $\mathcal{A}$-module bundles with inner product and sections $h_0 \in C^\infty(X; \Self_{\mathcal{A}}^\dagger(\slashed S))$ with respect to their inclusions. 
The isomorphism in \eqref{eq_Ph_V_rham} is given by Lemma \ref{lem_univ_Skew_bundle}. 
Moreover, by the same naturality, we see that the class \eqref{eq_Ph_V_rham} is natural in $(X, \mathcal{A})$. 
Since $B\mathrm{Aut}_{*, \Z_2}(A)$ can be realized as a direct limit of closed manifolds, we get the element
\begin{align}\label{eq_Ph_EO}
    \{\mathrm{Ph}(V, \slashed S, h_0)\}_{(X, V, S, h_0)} \in H^{4\Z + \type(\mathcal{A})}(\mathcal{KO}_{E\mathrm{Aut}_{*, \Z_2}(A)}, \{\beta\}; \mathrm{Ori}(\mathcal{A})  ).
\end{align}
Here we abuse the notation to denote by $\mathcal{KO}_{E\mathrm{Aut}_{*, \Z_2}(A)}$ the bundle over $B\mathrm{Aut}_{*, \Z_2}(A)$ defined by applying
\eqref{eq_Skew_bundle_V} to $P = E\mathrm{Aut}_{*, \Z_2}(A)$. 
In turn, by Lemma \ref{lem_Ph_spin} this element is the pullback under $(w_1, w_2) \colon \mathcal{KO}_{E\mathrm{Aut}_{*, \Z_2}(A)} \to (\mathcal{KO}_{\type(A)})_{\iota}$ of an element on $(\mathcal{KO_{\type(\mathcal{A})}})_{\iota}$ denoted temporarily as
\begin{align}\label{eq_Ph'}
    [\mathrm{Ph}'] \in H^{4\Z + \type(\mathcal{A})}((\mathcal{KO}_{\type(A))})_{\iota}, \{\beta\}; E\Z_2 \times_{\Z_2} \R ). 
\end{align}

Moreover, the element \eqref{eq_Ph'} restricts to the universal class for the untwisted topological Pontryagin character 
\[
[\mathrm{Ph}_{\mathrm{top}}] \in H^{4\Z + \type(\mathcal{A})}(KO_{\type(A)}, \{*\}; \mathrm{Ori}(\mathcal{A}) )
\] 
on each fiber. 
To see this, it is enough to show the corresponding statement for any of $\mathcal{KO}_{\mathcal{A}}$. 
When $X = \pt $, the element \eqref{eq_twisted_Ph_univ_rham} coincides with the untwisted universal Pontryagin form in \eqref{eq_univ_Ph}. 
Using Theorem \ref{thm_Ph=Ph_top} and taking the inverse limit, we get the desired claim. 

By Lemma \ref{lem_twisted_Ph_characterization}, this property establishes
\begin{align*}
    [\mathrm{Ph}'] = [\mathrm{Ph}_{\mathrm{top}}], 
\end{align*}
which leads to the theorem. 
\end{proof}

Recall that, for a CW-pair $(X, Y)$, given $\mathcal{A}$ over $X$, we have \eqref{eq_twisted_Ph_top}
\begin{align*}
    \mathrm{Ph}_{\mathrm{top}} \colon KO^{\mathcal{A}}_+(X, Y) \to H^{4\Z + \type(\mathcal{A})}(X, Y; \mathrm{Ori}(\mathcal{A}) ). 
\end{align*}
Theorem \ref{thm_Ph=Ph_top_twisted} and the definition of the isomorphism \eqref{eq_KO_m=KO_twisted} imply the following. 
\begin{cor}\label{cor_Ph=Ph_top_twisted}
In the setting of Definition \ref{def_Ph_m}, 
let $(X, Y)$ be an object of $\mathrm{MfdPair}_f$. 
Suppose we have two sections $h_0, h_1 \in C^\infty( X ; \mathrm{Self}_{\mathcal{A}}^*(\slashed S))$ such that $h_0|_Y = h_1|_Y$. 
Then we have
\begin{align} 
    \mathrm{Ph}_{\mathrm{top}}([\slashed S, h_0, h_1]) = \mathrm{Rham}\left( \mathrm{Ph}_\lself(h_1; \nabla^{\mathcal{A}}, \nabla^{\slashed S}) - \mathrm{Ph}_\lself(h_0; \nabla^{\mathcal{A}}, \nabla^{\slashed S})\right). 
\end{align}
Here we note that the right hand side is well-defined by \eqref{eq_difference_Ph_closed}. 
\end{cor}

\section{The definition of differential \texorpdfstring{$KO$}{KO}-theory \texorpdfstring{$\widehat{KO}_+$}{}}\label{sec_diff_KO}
In this section we define our model of differential $KO$ groups. 
We deal with the untwisted case in Subsection \ref{subsec_untwisted_KO_hat}, and the twisted case in Subsection \ref{subsec_twisted_KO_hat}. 
In this section we work over $\R$. 

\subsection{The untwisted groups \texorpdfstring{$\widehat{KO}^{A}_+$}{}}\label{subsec_untwisted_KO_hat}
In this subsection we give a model $\widehat{KO}^{A}_+$ for the untwisted differential $KO$-theory groups, by refining Definition \ref{def_karoubi_untwisted_KO}. 
In this subsection we use the notations in Subsection \ref{subsec_superconn_triv}. 

\begin{defn}\label{def_KO_hat_quadruple}
Let $A$ be a nondegenerate simple central graded $*$-algebra. 
Let $(X, Y)$ be an object in $\mathrm{MfdPair}_f$. 
\begin{itemize}
    \item A {\it $\widehat{KO}_+$-cocycle} $(S, h_0, h_1, \eta)$ on $(X, Y)$ consists of an $A$-module $S$ with an inner product, two smooth maps $h_0, h_1 \in C^\infty( X , \mathrm{Self}_{A}^*(S))$ such that $h_0|_Y = h_1|_Y$, and an element
    \[
        \eta \in \Omega^{4\Z + \type(A)-1}(X, Y; \Ori(A)) / \mathrm{Im}(d).
   \]
    \item Two $\widehat{KO}_+$-cocycles $(S, h_0, h_1, \eta)$ and $(S', h'_0, h'_1, \eta')$ are {\it isomorphic}
    if there exists an isometric isomorphism of $A$-modules $f \colon S \simeq  S'$ such that $f\circ h_i = h'_i \circ f$ for $i = 0, 1$, and we have $\eta = \eta'$. 
\end{itemize}
\end{defn}

\begin{defn}[{$\widehat{KO}_+^{A}(X, Y)$}]\label{def_untwisted_KO_hat_m}
Let $A$ and $(X, Y)$ be as in Definition \ref{def_KO_hat_quadruple}
\begin{itemize}
    \item
We introduce an abelian monoid structure on the set $\widehat{M}^{A}_+(X, Y)$ of isomorphism classes of $\widehat{KO}_+$-cocycles $(S, h_0, h_1, \eta)$ on $(X, Y)$ by
\begin{align*}
    [S, h_0, h_1, \eta] + [S', h'_0, h'_1, \eta']
    = [S \oplus  S', h_0 \oplus h'_0,h_1 \oplus h'_1, \eta +\eta']. 
\end{align*}
\item We define $\widehat{Z}^{A}_+(X, Y)$ to be the submonoid of $\widehat{M}^{A}_+(X, Y)$ consisting of elements of the form
\begin{align*}
    [S, h_0, h_1, \mathrm{CS}_\lself(h_I)], 
\end{align*}
where $h_I$ is a smooth homotopy between $h_0$ and $h_1$ which is constant on $Y$, i.e., a smooth map $h_I \in C^\infty(I \times X, \Self^*_{A}(S))$ with $h_I|_{\{i\} \times X} = h_i$ for $i = 0, 1$ and $h_I|_{\{t\} \times Y} = h_0|_Y$ for all $t \in I$. 
\item We define $\widehat{KO}^{A}_+(X, Y):=\widehat{M}^{A}_+(X, Y) / \widehat{Z}^{A}_+(X, Y)$. 
\end{itemize}
\end{defn}

\begin{lem}\label{lem_untwisted_additive_inverse}
$\widehat{KO}_+^{A}(X, Y)$ is an abelian group. 
\end{lem}

\begin{proof}
The additive inverse of an element $ [S, h_0, h_1, \eta]\in \widehat{KO}_+^{A}(X, Y)$ is given as follows. 
Consider the $A$-module bundle $ S \oplus S$, and note that $h_0 \oplus h_1$ and $ h_1 \oplus h_0$ are homotopic relative to $Y$ in $C^\infty(X; \mathrm{Self}_{A}^*(S \oplus S))$. 
Take any smooth homotopy $h_I$ between them, which is constant on $Y$. 
Then we have the following equality in $\widehat{KO}_+^{A}(X, Y)$. 
\begin{align*}
   \left[ S \oplus  S, h_0 \oplus h_1, h_1 \oplus h_0, \mathrm{CS}_\lself\left(h_I\right)\right] = 0. 
\end{align*}
This means that we have
\begin{align} 
    -[ S,  h_0, h_1, \eta] = \left[ S, h_1 , h_0, -\eta + \mathrm{CS}_\lself\left(h_I\right)\right]. 
\end{align}
\end{proof}

The abelian group $\widehat{KO}_+^{A}(X, Y)$ has the obvious naturality in $(X, Y)$, so we regard $\widehat{KO}_+^{A}$ as a contravariant functor, 
\begin{align}
    \widehat{KO}_+^{A} \colon \mathrm{MfdPair}^{\mathrm{op}}_f \to \mathrm{Ab}. 
\end{align}
The remainder of this subsection is devoted to showing that $\widehat{KO}_+^A$ satisfies the axioms of differential extension as developed by Bunke and Schick \cite{BunkeSchicksmoothK, BunkeSchickDiffKsurvey}. First of all, we define the structure homomorphisms for $\widehat{KO}^{A}_+$ as follows. 

\begin{defn}[{Structure homomorphisms for $\widehat{KO}^{A}_+$}]\label{def_str_hom_untwisted}
We define the following structure homomorphisms, which are natural in $(X, Y)$. 
\begin{align*}
        R &\colon \widehat{KO}_+^{A}(X, Y) \to \Omega_{\mathrm{clo}}^{4\Z + \type(A)}(X, Y; \Ori(A))\\
       &[S, h_0, h_1, \eta] \mapsto \mathrm{Ph}_\lself(h_1) - \mathrm{Ph}_\lself(h_0) + d\eta.\\
        I &\colon \widehat{KO}_+^{A}(X, Y) \to KO_+^{A}(X, Y) \\
       & [S, h_0, h_1, \eta] \mapsto [S, h_0, h_1]. \\
        a &\colon \Omega^{4\Z + \type(A) - 1}(X, Y; \Ori(A)) / \mathrm{Im}(d) \to \widehat{KO}_+^{A}(X, Y)  \\
       & \eta \mapsto [0, 0, 0, \eta]. 
\end{align*}
The well-definedness of $R$ follows from \eqref{eq_Ph=dCS_triv}.
\end{defn}

Now we check that the functor $\widehat{KO}_+^{A}$ satisfies the axioms for the differential $KO^{\type(A)}$.  
\begin{thm}\label{thm_axiom_diffcoh_untwisted_KO}
In the notations of Definition \ref{def_str_hom_untwisted}, we have the following. 
\begin{enumerate}
\item We have $R \circ a = d$. 
\item The following diagram commutes. 
\begin{align*}
    \xymatrix{
    \widehat{KO}_+^{A}(X, Y)\ar[r]^-{R} \ar[d]^I & \Omega^{4\Z + {\type}(A)}_{\mathrm{clo}}(X, Y;\Ori(A)) \ar[d]^{\mathrm{Rham}} \\
    KO_+^{A}(X, Y)\ar[r]^-{\mathrm{Ph}_{\mathrm{top}}} & H^{4\Z + {\type}(A)}(X, Y; \Ori(A)).
    } 
\end{align*}

    \item The following sequence is exact. 
\begin{align}\label{eq_prop_axiom_untwisted}
   \xymatrix@R=10pt{
   KO_+^{\Sigma^{0,1}A}(X, Y) 
   \ar[r]^-{\mathrm{Ph}_{\mathrm{top}}} & 
    \Omega^{4\Z + \type(A) - 1}(X, Y; \Ori(A)) / \mathrm{Im}(d) 
    \ar[ld]_{a} \\
    \widehat{KO}_+^{A}(X, Y)
    \ar[r]_{I} &
    KO_+^{A}(X, Y) \ar[r]
    & 0. 
    }
\end{align}
\end{enumerate}

\end{thm}
\begin{proof}
(1) is obvious. 
(2) follows from Corollary \ref{cor_Ph=Ph_top}. 
We prove (3). 
The exactness of \eqref{eq_prop_axiom_untwisted} at ${KO}_+^{A}(X, Y)$, i.e., the surjectivity of $I$, is obvious. 
We then check the exactness at $\Omega^{4\Z + \type(A) - 1}(X, Y; \Ori(A)) / \mathrm{Im}(d)$. 
We have the following commutative diagram. 
\begin{align}\label{diag_proof_axiom_1_triv}
    \xymatrix{
    KO_+^{A}(I \times X, \del I \times X \cup I \times Y)\ar[r]^-{\mathrm{Ph}_{\mathrm{top}}}\ar[d]^-{\mathrm{susp}}_-{\simeq} & H^{4\Z + \type(A)}(I \times X, \del I \times X \cup I \times Y; \Ori(A))\ar[d]^-{\mathrm{susp}}_-{\simeq} \\
    KO_+^{\Sigma^{0, 1}A}(X, Y)\ar[r]^-{\mathrm{Ph}_{\mathrm{top}}} & H^{4\Z + \type(A)-1}(X, Y; \Ori(A))
    }
\end{align}
We would like to apply (2) to the top row of the above diagram, but a technical point is that the subset $\del I \times X \cup I \times Y$ is not a manifold. 
So we choose smooth neighborhood $U$ of $\del I \times X \cup I \times Y$ in $I \times X$ which has a deformation retraction to $\del I \times X \cup I \times Y$. 
Then by (2) and the surjectivity of $I$, we have the following commutative diagram. 
\begin{align}\label{diag_proof_axiom_2_triv}
    \xymatrix{
     \widehat{KO}_+^{A}(I \times X, \overline{U})\ar[r]^-{R} \ar@{->>}[d]^I & \Omega_{\mathrm{clo}}^{4\Z + \type(A)}(I \times X, \overline{U}; \Ori(A)) \ar[d]^{\mathrm{Rham}} \\
    KO_+^{A}(I \times X, \overline{U})\ar[r]^-{\mathrm{Ph}_{\mathrm{top}}} \ar[d]^-{\simeq}& H^{4\Z + \type(A)}(I \times X, \overline{U}; \Ori(A)) \ar[d]^-{\simeq}\\
     KO_+^{A}(I \times X, \del I \times X \cup I \times Y)\ar[r]^-{\mathrm{Ph}_{\mathrm{top}}}& H^{4\Z + \type(A)}(I \times X, \del I \times X \cup I \times Y; \Ori(A))
    }
\end{align}
Thus, noting that the suspension in the de Rham cohomology is given by the integration $\int_I$ on closed forms, the desired exactness is equivalent to the exactness of the following sequence. 
\begin{align}\label{eq_proof_axioh_1_triv}
    \widehat{KO}_+^{A}(I \times X, \overline{U}) \xrightarrow{\mathrm{Rham} \circ \int_I \circ R}\Omega^{4\Z + \type(A) - 1}(X, Y; \Ori(A)) / \mathrm{Im}(d) \xrightarrow{a}\widehat{KO}_+^{A}(X, Y). 
\end{align}
Now we prove that the composition $a \circ\mathrm{Rham} \circ  \int_I \circ R$ in \eqref{eq_proof_axioh_1_triv} is zero. 
By the commutativity of the upper square of \eqref{diag_proof_axiom_2_triv}, the image under $\mathrm{Rham} \circ \int_I \circ R$ of any element in $\widehat{KO}_+^{A}(I \times X, \overline{U})$ only depends on its image in $KO_+^{A}(I \times X, \overline{U})$ under $I$. 
Moreover, by the excision, we know that the group $KO_+^{A}(I \times X, \overline{U})$ is generated by triples of the form 
\begin{align}
    (S, h_0, h_1), 
\end{align}
where $h_0, h_1$ are smooth and $h_0|_{I \times Y} = h_1|_{I \times Y}$ is {\it constant} in the $I$-direction, i.e., there exists a smooth map $h_Y \in C^\infty(Y, \Self_{A}^*( S))$ such that $h_i|_{\{t\} \times Y} = h_Y$ for all $t \in [0, 1]$, $i = 0, 1$. 
It is enough to prove that the element 
\begin{align*}
    [ S, h_0, h_1, 0] \in \widehat{KO}_+^{A}(I \times X, \overline{U})
\end{align*}
maps to zero under the composition \eqref{eq_proof_axioh_1_triv}. 
We have
\begin{align*}
    \int_I \circ R [ S, h_0, h_1, 0] = \int_I \mathrm{Ph}_\lself(h_1) - \int_I \mathrm{Ph}_\lself(h_0). 
\end{align*}
So we have
\begin{align*}
   a \circ \mathrm{Rham} \int_I \circ R [ S, h_0, h_1, 0] = \left[0, 0, 0,  \int_I \mathrm{Ph}_\lself(h_1) - \int_I \mathrm{Ph}_\lself(h_0)\right]. 
\end{align*}
To see that this element is zero, notice that, for each $i = 0, 1$ the following $\widehat{KO}^A_+$-cocycle is an element in $\widehat{Z}^{A}_+(X, Y)$, 
\begin{align*}
    [S, h_i|_{\{0\} \times X}, h_i|_{\{1\} \times X}, \mathrm{CS}_\lself(h_i)] \in \widehat{Z}^{A}_+(X, Y), 
\end{align*}
where $h_i \in C^\infty(I \times X, \Self^*_A(S))$ is regarded as a homotopy from $h_i|_{\{0\} \times X}\in C^\infty(X, \Self^*_A(S))$ to $h_i|_{\{1\} \times X} \in C^\infty(X, \Self^*_A(S))$. 
For each $i = 0, 1$ we have
\begin{align*}
    \mathrm{CS}_\lself(h_i) = \int_I \mathrm{Ph}_\lself(h_i). 
\end{align*}
Also notice that for $i = 0, 1$ we have
\begin{align*}
    h_0|_{\{i\} \times X} = h_1|_{\{i\} \times X}, 
\end{align*}
since $\del I \times X \subset \overline{U}$. 
Thus we get
\begin{align*}
&\left[0, 0, 0, \int_I \mathrm{Ph}_\lself(h_1) - \int_I \mathrm{Ph}_\lself(h_0)\right]\\
   & =[S, h_1|_{\{0\} \times X}, h_1|_{\{1\} \times X}, \mathrm{CS}_\lself(h_1)] - [ S, h_0|_{\{0\} \times X}, h_0|_{\{1\} \times X}, \mathrm{CS}_\lself(h_0)] \\
   &= 0, 
\end{align*}
as desired. 
This proves that the composition in \eqref{eq_proof_axioh_1_triv} is zero. 

To complete the proof of exactness of \eqref{eq_proof_axioh_1_triv}, take any element $\eta \in \Omega^{4\Z + \type(A) - 1}(X, Y; \Ori(A))/ \mathrm{Im}(d)$ such that $a(\eta) = [0, 0, 0,  \eta] = 0 $ in the group $\widehat{KO}_+^{A}(X, Y)$. 
This means that there exists an element
$x\in \widehat{Z}_+^{A}(X, Y)$ such that $x + a(\eta)$ is also an element in $\widehat{Z}_+^{A}(X, Y)$. 
This implies that, under the description $x = [S,  h_0, h_1, \mathrm{CS}_\lself(h_I)]$ where $h_I$ is a homotopy between $h_0$ and $h_1$, there exists another homotopy $h'_I$ between $h_0$ and $h_1$ such that 
\begin{align}\label{eq_proof_axiom_2_triv}
    \mathrm{CS}_\lself(h'_I) = \mathrm{CS}_\lself(h_I) + \eta \mod(\mathrm{Im}(d)). 
\end{align}
Notice that we have 
\begin{align*}
    h_I|_{\del I \times X \cup I \times Y} = h'_I|_{\del I \times X \cup I \times Y}. 
\end{align*}
We may replace each of $h_I$ and $h'_I$ by a homotopic one so that they satisfy 
\begin{align*}
    h_I|_{\overline{U}} = h'_I|_{\overline{U}}. 
\end{align*}
By \eqref{eq_CS_mod_exact}, the equality \eqref{eq_proof_axiom_2_triv} is preserved by this replacement. 
Now we have the element
\begin{align}\label{eq_proof_axiom_4_triv}
    [ S,  h'_I, h_I, 0] \in \widehat{KO}_+^{A}(I \times X, \overline{U}). 
\end{align}
We have
\begin{align}\label{eq_proof_axiom_3_triv}
    R[S,  h'_I, h_I, 0] = \mathrm{Ph}_\lself(h'_I)-\mathrm{Ph}_\lself(h_I). 
\end{align}
By \eqref{eq_proof_axiom_2_triv}, \eqref{eq_proof_axiom_3_triv} and the definition of $\mathrm{CS}_\lself$, we see that the element \eqref{eq_proof_axiom_4_triv} maps to $\eta$ under the left arrow in \eqref{eq_proof_axioh_1_triv}. 
This completes the proof of the exactness of \eqref{eq_prop_axiom_untwisted} at $\Omega^{4\Z + \type(A) - 1}(X, Y; \Ori(A)) / \mathrm{Im}(d) $. 

Finally we prove the exactness of \eqref{eq_prop_axiom_untwisted} at $\widehat{KO}_+^{A}(X, Y)$. 
The equality $I \circ a = 0$ is obvious. 
Take any element $[ S,  h_0, h_1, \eta] \in \widehat{KO}_+^{A}(X, Y)$ such that $I[ S,  h_0, h_1, \eta] = [S, h_0, h_1] = 0$ in ${KO}_+^{A}(X, Y)$. 
It is enough to consider the case $(S, h_0, h_1) \in Z_+^A(X, Y)$. 
Then there exists a homotopy $h_I$ between $h_0$ and $h_1$ which is constant on $Y$. 
Moreover, we can choose $h_I$ so that it is smooth. 
Then we have
\begin{align*}
    [ S,  h_0, h_1, \eta]
    = [ S,  h_0, h_1, \mathrm{CS}_\lself(h_I)] + a(\eta - \mathrm{CS}_\lself(h_I)) 
    = a(\eta - \mathrm{CS}_\lself(h_I)). 
\end{align*}
This completes the proof of the exactness of \eqref{eq_prop_axiom_untwisted} at $\widehat{KO}_+^{A}(X, Y)$ and finishes the proof of the theorem. 
\end{proof}

Thus, together with Fact \ref{fact_karoubi_KO} we get the following. 
\begin{thm}\label{thm_KO_hat_untwisted}
The quadruple $\left(\widehat{KO}^A_+, R, I, a\right)$ is a differential extension of the $KO$-theory $KO^A_+ \simeq KO^{\type(A)}$ on $\mathrm{MfdPair}_f$. 
\end{thm}

Let $V := V^0 \oplus V^1$ be a $\Z_2$-graded real vector space with a positive definite inner product on each $V^i$. 
We denote the $\Z_2$-grading operator on $V$ by $\gamma_V$. 
Since the algebra $\End(V)$ is negligible, 
Theorem \ref{thm_KO_hat_untwisted} in particular means that $\widehat{KO}_+^{A}$ and $\widehat{KO}_+^{\End(V) \widehat{\otimes}A}$ are both differential extensions of $KO^{\type(A)}$. 
Actually, we have a canonical isomorphism $\widehat{KO}_+^{A} \simeq \widehat{KO}_+^{\End(V) \widehat{\otimes}A}$ as follows. 
\begin{prop}\label{prop_isom_negligible_untwisted}
In the above settings, we have a natural isomorphism
\begin{align}\label{eq_isom_negligible_hat_untwisted}
    \widehat{KO}^{A}_+ \simeq \widehat{KO}^{\mathrm{End}(V) \widehat{\otimes}A}_+
\end{align}
which is compatible with the corresponding isomorphisms of other functors appearing in Definition \ref{def_str_hom_untwisted} via the structure homomorphisms (under the isomorphism $\Ori(A) \simeq \Ori(\End(V) \widehat{\otimes} A)$ given by $u \mapsto \gamma_V \widehat{\otimes} u$). 
\end{prop}
\begin{proof}
The natural isomorphism \eqref{eq_isom_negligible_hat_untwisted} on an object $(X, Y)$ in $\mathrm{MfdPair}_f$ is given by the following refinement of the untwisted version of the map \eqref{eq_isom_negligible_map}, 
\begin{align}\label{eq_isom_negligible_map_hat_untwisted}
    [S,  h_0, h_1, \eta] \mapsto [E \otimes  S, \psi^E(h_0), \psi^E(h_1), \eta]. 
\end{align}
The well-definedness and the compatibility with the structure maps follow from Lemma \ref{lem_Ph_spin}. 
By the exactness of \eqref{eq_prop_axiom_untwisted} and the five lemma, we see that \eqref{eq_isom_negligible_map_hat_untwisted} defines an isomorphism, as desired. 
\end{proof}

In particular we get
\begin{cor}\label{cor_KO_degree_classification_hat}
The isomorphism class of the differential extension \[\left(\widehat{KO}_+^{A}, R, I, a\right)\] of $KO^{\type(A)}$ only depends on the type of $A$.  
\end{cor}

\subsection{The twisted groups \texorpdfstring{$\widehat{KO}^{\mathcal{A}}_+$}{}}\label{subsec_twisted_KO_hat}
In this subsection we give a model $\widehat{KO}^\mathcal{A}_+$ for the twisted differential $KO$-theory groups, by refining Definition \ref{def_karoubi_twisted_KO}.  
First we introduce the category $\Tw^2_{\widehat{KO}_+}$ of twists on $\widehat{KO}_+$. 

\begin{defn}[{$\hatTw$}]\label{def_twist_cat_hat_KO_+}
\begin{enumerate}
    \item For each object $X \in \mathrm{Mfd}_f$, we define $\Tw_{\widehat{KO}_+, X}$ to be the groupoid of smooth bundles of nondegenerate simple central graded $*$-algebras $\mathcal{A}$ over $X$, where morphisms are ismorphisms between such bundles. 
    \item For each morphism $f \colon X \to X'$ in $\mathrm{Mfd}_f$, we define a functor $f^* \colon \Tw_{\widehat{KO}_+, X'} \to \Tw_{\widehat{KO}_+, X}$ by the pullback. 
    \item We define the category $\hatTw$ such that an object $(X, Y, \mathcal{A})$ consists of $(X, Y) \in \mathrm{MfdPair}_f$ and $\mathcal{A} \in \Tw_{\widehat{KO}_+, X}$ and a morphism from $(X, Y, \mathcal{A})$ to $(X', Y', \mathcal{A}')$ consists of a morphism $f \colon (X, Y) \to (X', Y')$ in $\mathrm{MfdPair}_f$ and an isomorphism $\mathcal{A} \simeq f^*\mathcal{A}'$. 
\end{enumerate}
\end{defn}

We have the forgetful functor
\begin{align*}
   \hatTw \to \Tw_{KO_+}^2 . 
\end{align*}

\begin{defn}
Let $(X, Y, \mathcal{A}) \in \hatTw$. 
\begin{itemize}
    \item A {\it $\widehat{KO}_+$-cocycle} $(\slashed S, \nabla^{\mathcal{A}}, \nabla^{\slashed S}, h_0, h_1, \eta)$ on $(X, Y, \mathcal{A})$ consists of an $\mathcal{A}$-module bundle $\slashed S$ with an inner product, a connection $\nabla^{\mathcal{A}}$ on $\mathcal{A}$, a self-adjoint $\mathcal{A}$-connnection $\nabla^{\slashed S}$ on $\slashed S$ compatible with $\nabla^{\mathcal{A}}$, two smooth sections $h_0, h_1 \in C^\infty( X ; \mathrm{Self}_{\mathcal{A}}^*(\slashed S))$ such that $h_0|_Y = h_1|_Y$, and an element
    $
        \eta \in \Omega^{4\Z + \type(\mathcal{A}) - 1}(X, Y; \Ori(\mathcal{A})) / \mathrm{Im}(d) $. 
    \item Two $\widehat{KO}_+$-cocycle $(\slashed S, \nabla^{\mathcal{A}}, \nabla^{\slashed S}, h_0, h_1, \eta)$ and $(\slashed S',(\nabla^{\mathcal{A}})', \nabla^{\slashed S'}, h'_0, h'_1, \eta')$ are {\it isomorphic}
    if we have $\nabla^{\mathcal{A}} = (\nabla^{\mathcal{A}})'$ and $\eta = \eta'$, and there exists an isometric isomorphism of smooth $\mathcal{A}$-module bundles $f \colon \slashed S \simeq \slashed S'$ such that $f^* \nabla^{\slashed S'} = \nabla^{\slashed S}$ and $f\circ h_i = h'_i \circ f$ for $i = 0, 1$. 
\end{itemize}
\end{defn}

\begin{defn}[{$\widehat{KO}_+^{\mathcal{A}}(X, Y)$}]\label{def_twisted_KO_hat_m}
Let $(X, Y, \mathcal{A}) \in \hatTw$. 
\begin{itemize}
    \item
We introduce an abelian monoid structure on the set $\widehat{M}^{\mathcal{A}}_+(X, Y)$ of isomorphism classes of $\widehat{KO}_+$-cocycles $(\slashed S, \nabla^{\mathcal{A}}, \nabla^{\slashed S}, h_0, h_1, \eta)$ by
\begin{align*}
    &[\slashed S, \nabla^{\mathcal{A}}, \nabla^{\slashed S}, h_0, h_1, \eta] + [\slashed S',(\nabla^{\mathcal{A}})', \nabla^{\slashed S'}, h'_0, h'_1, \eta']\\
   & \quad = [\slashed S \oplus \slashed S', \nabla^{\mathcal{A}} \oplus (\nabla^{\mathcal{A}})',\nabla^{\slashed S} \oplus \nabla^{\slashed S'} , h_0 \oplus h'_0,h_1 \oplus h'_1, \eta +\eta']. 
\end{align*}
\item We define $\widehat{Z}^{\mathcal{A}}_+(X, Y)$ to be the submonoid of $\widehat{M}^{\mathcal{A}}_+(X, Y)$ consisting of elements of the form
\begin{align*}
    [\slashed S, \nabla^{\mathcal{A}}, \nabla^{\slashed S}, h_0, h_1, \mathrm{CS}_\lself(h_I; \nabla^{\mathcal{A}}, \nabla^{\slashed S})], 
\end{align*}
where $h_I$ is a smooth homotopy between $h_0$ and $h_1$ which is constant on $Y$, i.e., a smooth section $h_I \in C^\infty(I \times X; \mathrm{pr}_X^*\Self^*_{\mathcal{A}}(\slashed S))$ with $h_I|_{\{i\} \times X} = h_i$ for $i = 0, 1$ and $h_I|_{\{t\} \times Y} = h_0|_Y$ for all $t \in I$ (see \eqref{eq_CS_relative}). 
\item We define $\widehat{KO}_+^{\mathcal{A}}(X, Y):=\widehat{M}^{\mathcal{A}}_+(X, Y) / \widehat{Z}^{\mathcal{A}}_+(X, Y)$. 
\end{itemize}
\end{defn}


\begin{lem}
$\widehat{KO}_+^{\mathcal{A}}(X, Y)$ is an abelian group. 
\end{lem}

\begin{proof}
The proof is similar to the untwisted case (Lemma \ref{lem_untwisted_additive_inverse}). 
The additive inverse of an element $[\slashed S, \nabla^{\mathcal{A}}, \nabla^{\slashed S}, h_0, h_1, \eta] \in \widehat{KO}_+^{\mathcal{A}}(X, Y)$ is given by
\begin{multline} \label{eq_twisted_diff_KO_additive_inv}
    -[\slashed S, \nabla^{\mathcal{A}}, \nabla^{\slashed S}, h_0, h_1, \eta] \\
    = \left[\slashed S,\nabla^{\mathcal{A}}, \nabla^{\slashed S}, h_1 , h_0, -\eta + \mathrm{CS}_\lself\left(h_I; \nabla^{\mathcal{A}} \oplus  \nabla^{\mathcal{A}}, \nabla^{\slashed S}\oplus \nabla^{\slashed S}\right)\right]. 
\end{multline}
Here $h_I$ is a homotopy between $h_0 \oplus h_1$ and $h_1 \oplus h_0$ on $\slashed S \oplus \slashed S$. 
\end{proof}

Thus we get the functor
\begin{align}\label{eq_hat_KO_+_functor}
    \widehat{KO}_+ \colon \hatTw \to \mathrm{Ab}. 
\end{align}

The structure homomorphisms in the twisted case are given as follows. 

\begin{defn}[{Structure homomorphisms for $\widehat{KO}^\mathcal{A}_+$}]\label{def_str_hom}
We define the following structure homomorphisms, which are natural in $(X, Y, \mathcal{A}) \in \hatTw$. 
\begin{align*}
        R &\colon \widehat{KO}_+^\mathcal{A}(X, Y) \to \Omega^{4\Z + \type(\mathcal{A})}_{\mathrm{clo}}(X, Y; \Ori(\mathcal{A})) \\
       &[\slashed S, \nabla^\mathcal{A}, \nabla^{\slashed S}, h_0, h_1, \eta] \mapsto \mathrm{Ph}_\lself(h_1; \nabla^\mathcal{A}, \nabla^{\slashed S}) - \mathrm{Ph}_\lself(h_0; \nabla^\mathcal{A}, \nabla^{\slashed S}) + d\eta.\\
        I &\colon \widehat{KO}_+^\mathcal{A}(X, Y) \to KO_+^\mathcal{A}(X, Y) \\
       & [\slashed S, \nabla^\mathcal{A}, \nabla^{\slashed S}, h_0, h_1, \eta] \mapsto [\slashed S, h_0, h_1]. \\
        a &\colon \Omega^{4\Z + \type(\mathcal{A}) - 1}(X, Y; \Ori(\mathcal{A})) / \mathrm{Im}(d) \to \widehat{KO}_+^\mathcal{A}(X, Y)  \\
       & \eta \mapsto [0, 0, 0, 0, 0, \eta]. 
\end{align*}
The well-definedness of $R$ follows from \eqref{eq_Ph=dCS} and \eqref{eq_difference_Ph_closed}. 
The well-definedness of $a$ and $I$ is obvious. 
\end{defn}

Recall that, by \eqref{eq_connection_action}, 
$\Omega^1(X; \mathcal{A}^0_{\mathrm{skew}})$ acts on the set of pairs $(\nabla^\mathcal{A}, \nabla^{\slashed S})$. 
\begin{lem}\label{lem_twisted_KO_connection_V}
For any $KO^\mathcal{A}_+$-cycle $(\slashed S, \nabla^\mathcal{A}, \nabla^{\slashed S}, h_0, h_1, \eta)$ and any element $C \in \Omega^1(X; \mathcal{A}^0_{\mathrm{skew}})$, we have the following equality in $\widehat{KO}_+^\mathcal{A}(X, Y)$. 
\begin{align*}
    [\slashed S, \nabla^\mathcal{A}, \nabla^{\slashed S}, h_0, h_1, \eta]
    = [\slashed S, \nabla^\mathcal{A} + C, \nabla^{\slashed S} + C, h_0, h_1, \eta].
\end{align*}
In particular, if we choose and fix a connection $\nabla^\mathcal{A}$ on $\mathcal{A}$, then any element in $\widehat{KO}^\mathcal{A}_+$ is represented by a $\widehat{KO}_+$-cocycle of the form $(\slashed S, \nabla^\mathcal{A}, \nabla^{\slashed S}, h_0, h_1, \eta)$. 
\end{lem}

\begin{proof}
By \eqref{eq_twisted_diff_KO_additive_inv}, it is enough to prove that the following $\widehat{KO}_+$-cocycle is in $\widehat{Z}^\mathcal{A}_+(X, Y)$. 
\begin{align*}
    &[\slashed S, \nabla^\mathcal{A} + C, \nabla^{\slashed S} + C, h_0, h_1, \eta] \\
    &\quad + \left[\slashed S,\nabla^\mathcal{A}, \nabla^{\slashed S}, h_1 , h_0, -\eta + \mathrm{CS}_\lself\left(h_I; \nabla^\mathcal{A} \oplus  \nabla^\mathcal{A}, \nabla^{\slashed S}\oplus \nabla^{\slashed S}\right)\right] \\
    &= \left[\slashed S \oplus \slashed S,(\nabla^\mathcal{A} + C) \oplus \nabla^\mathcal{A}, (\nabla^{\slashed S} + C) \oplus \nabla^{\slashed S}, h_0 \oplus h_1 , h_1 \oplus h_0,  \right. \\ 
    & \hspace{6cm} 
    \left. \mathrm{CS}_\lself\left(h_I; \nabla^\mathcal{A} \oplus  \nabla^\mathcal{A}, \nabla^{\slashed S}\oplus \nabla^{\slashed S}\right)\right]. 
\end{align*}
Here $h_I$ is any homotopy between $h_0 \oplus h_1$ and $h_1 \oplus h_0$ on $\slashed S \oplus \slashed S$. 
Indeed, this claim follows from
\begin{multline*}
\mathrm{CS}_\lself\left(h_I; \nabla^\mathcal{A} \oplus  \nabla^\mathcal{A}, \nabla^{\slashed S}\oplus \nabla^{\slashed S}\right) \\
= \mathrm{CS}_\lself\left(h_I; (\nabla^\mathcal{A} + C) \oplus \nabla^\mathcal{A}, (\nabla^{\slashed S} + C) \oplus \nabla^{\slashed S}\right)
\end{multline*}
as a result of Lemma \ref{lem_change_connection_Ph}. 
\end{proof}

\begin{thm}\label{thm_axiom_diffcoh_KO}
Let $(X, Y, \mathcal{A}) \in \hatTw$. 
\begin{enumerate}
\item We have $R \circ a = d$. 
\item The following diagram commutes. 
\begin{align*}
    \xymatrix{
    \widehat{KO}_+^{\mathcal{A}}(X, Y)\ar[r]^-{R} \ar[d]^I & \Omega^{4\Z + \type(\mathcal{A})}_{\mathrm{clo}}(X, Y; \Ori(\mathcal{A})) \ar[d]^{\mathrm{Rham}} \\
    KO_+^{\mathcal{A}}(X, Y)\ar[r]^-{\mathrm{Ph}_{\mathrm{top}}} & H^{4\Z + \type(\mathcal{A})}(X, Y; \Ori(\mathcal{A}))
    }. 
\end{align*}

    \item The following sequence is exact. 
\begin{align}\label{eq_prop_axiom}
\xymatrix@R=10pt{
   KO_+^{\Sigma^{0,1}\mathcal{A}}(X, Y)
   \ar[r]^-{\mathrm{Ph}_{\mathrm{top}}} &
    \Omega^{4\Z + \type(\mathcal{A})-1}(X, Y; \Ori(\mathcal{A})) / \mathrm{Im}(d) \ar[ld]_{a} \\
    \widehat{KO}_+^{\mathcal{A}}(X, Y)
    \ar[r]^{I} &
    KO_+^{\mathcal{A}}(X, Y) \ar[r] & 0. 
}
\end{align}
\end{enumerate}

\end{thm}

\begin{proof}
(1) is obvious. 
(2) follows from Theorem \ref{thm_Ph=Ph_top_twisted}. 
The proof of (3) is analogous to the untwisted case in Theorem \ref{thm_axiom_diffcoh_untwisted_KO} (3), so we leave the details to the reader. 
\end{proof}

We use the axiom of differential extensions of twisted generalized cohomology theories given in \cite{BunkeSchickDiffKsurvey}. 
By Theorem \ref{thm_axiom_diffcoh_KO}, we get the following. 
\begin{thm}\label{thm_hat_KO_twisted}
The quadruple $\left(\widehat{KO}_+, R, I, a\right)$ is a differential extension of the twisted $KO$-theory ${KO}_+$ on $\hatTw$. 
\end{thm}

Now we investigate into the dependence of $\widehat{KO}^{\mathcal{A}}_+$ on the bundle $\mathcal{A}$. 
Recall that the twist of topological $KO$-theory is classified by $(\type(\mathcal{A}), w(\mathcal{A})) \in \mathrm{GBrO}(X) = \Z_8 \times HO(X)$ (Corollary \ref{cor_KO_twist_classification}). 
Actually the same classification holds for our differential refinement. 
We start with the following Proposition corresponding to Fact \ref{fact_isom_negligible}. 

\begin{prop}\label{prop_isom_negligible}
Let $(X, Y)$ be an object of $\mathrm{MfdPair}_f$. 
Let $E = E^0 \oplus E^1$ be a $\Z_2$-graded vector bundle over $X$ with fiberwise positive definite inner product on each $E^i$.
For any smooth bundle of nondegenerate simple central graded $*$-algebras $\mathcal{A}$, we have a canonical isomorphism
\begin{align}\label{eq_isom_spin_hat}
    \widehat{KO}^{\mathcal{A}}_+(X, Y) \simeq \widehat{KO}^{\End(E) \widehat{\otimes} \mathcal{A}}_+(X, Y)
\end{align} 
which is compatible with the corresponding isomorphisms of other functors appearing in Definition \ref{def_str_hom} via the structure homomorphisms 
(under the isomorphism $\mathrm{Ori}(\mathcal{A}) \simeq \mathrm{Ori}(\End(E) \widehat{\otimes}\mathcal{A})$ given by $u \mapsto \gamma_E \widehat{\otimes}u$). 
\end{prop}

\begin{proof}
Fix an orthogonal connection $\nabla^E$ on $E$, and denote by $\nabla^{\mathrm{End}(E)}$ the induced connection on $\mathrm{End}(E)$. 
The isomorphism \eqref{eq_isom_spin_hat} is given by the following refinement of the map \eqref{eq_isom_negligible_map}, 
\begin{multline}\label{eq_isom_negligible_map_hat}
    [\slashed S, \nabla^\mathcal{A}, \nabla^{\slashed S}, h_0, h_1, \eta] \\
    \mapsto [E \otimes \slashed S, \nabla^{\mathrm{End}(E)} \otimes 1 + 1 \otimes \nabla^{\mathcal{A}}, \nabla^{E} \otimes 1 + 1 \otimes \nabla^{\slashed S}, \psi^E(h_0), \psi^E(h_1), \eta]. 
\end{multline}
The well-definedness and the compatibility with the structure maps follow from Lemma \ref{lem_Ph_spin}. 
By the exactness of \eqref{eq_prop_axiom} and the five lemma, we see that \eqref{eq_isom_negligible_map_hat} defines an isomorphism. 
Actually, the right hand side of \eqref{eq_isom_negligible_map_hat} is independent of the choice of $\nabla_E$ by Lemma \ref{lem_twisted_KO_connection_V}. 
Thus we see that \eqref{lem_twisted_KO_connection_V} defines a canonical isomorphism \eqref{eq_isom_spin_hat} independent of any additional choice, as desired. 
\end{proof}

Thus we get the following differential version of Corollary \ref{cor_KO_twist_classification}. 

\begin{cor}\label{cor_KO_twist_classification_hat}
Given an object $(X, Y, \mathcal{A}) \in \hatTw$, the isomorphism class of the twisted differential $KO$-theory group $\widehat{KO}^{\mathcal{A}}_+(X, Y)$ only depends on the class of $\mathcal{A}$ in $\mathrm{GBrO}(X)$. 
\end{cor}

\section{The model in terms of skew-adjoint operators}\label{sec_skew}
In this section, we introduce another model of (twisted) differential $KO$-theory, $\widehat{KO}_-^A$ ($\widehat{KO}_-^{\mathcal{A}}$), in terms of skew-adjoint operators. 
In this section we work over $\R$, unless otherwise stated (e.g., in Remarks \ref{rem_complex_self_skew} and \ref{rem_Karoubi_K_self_skew}). 

\subsection{Algebraic preparations}\label{subsec_self_skew}
This subsection is devoted to algebraic preparations. 

\begin{defn}[{$\widetilde{\Sigma}^{0, 1}A$}]\label{def_ungraded_suspension}
Let $A$ be a simple central graded $*$-algebra. 
We define $\widetilde{\Sigma}^{0, 1}A$ to be $\Sigma^{0, 1}A = A \widehat{\otimes} Cl_{0, 1}$ as a $*$-algebra, and introduce a $\Z_2$ grading on $\widetilde{\Sigma}^{0, 1}A$ by $1\otimes \gamma$, where $\gamma$ is the $\Z_2$-grading operator on $Cl_{0, 1}$. 
\end{defn}
Thus we have
\begin{align*}
    (\widetilde{\Sigma}^{0, 1}A)^0 &= \{a \widehat{\otimes} 1 \ | \ a \in A \}, \\
    (\widetilde{\Sigma}^{0, 1}A)^1 &= \{a \widehat{\otimes} \beta \ | \ a \in A \}. 
\end{align*}

\begin{ex}
In the case where $A = Cl_{p, q}$ for nonnegative integers $p$ and $q$, we have $\Sigma^{0, 1}A = Cl_{p, q+1}$ and $\widetilde{\Sigma}^{0, 1}A \simeq Cl_{q, p+1}$. 
An isomorphism of $*$-algebras $\Sigma^{0, 1}A \simeq \widetilde{\Sigma}^{0, 1}A$ does not preserve the $\Z_2$-grading:
\begin{align*}
    Cl_{p, q+1} &\simeq Cl_{q, p+1} \\
    \beta_0 &\mapsto \beta'_0 \\
    \alpha_i &\mapsto \beta'_0\beta'_i  \ (1 \le i \le p) \\
    \beta_j &\mapsto \beta'_0 \alpha'_j\ (1 \le j \le q)
\end{align*}
where we denoted by $\alpha_*$ ($\alpha'_*$) and $\beta_*$ ($\beta'_*$) the negative and positive Clifford generators of $Cl_{p, q+1}$ ($Cl_{q, p+1}$), respectively.  
\end{ex}

\begin{lem}\label{lem_ungraded_susp_degree}
For any simple central graded $*$-algebra $A$, the $\Z_2$-graded $*$-algebra $\widetilde{\Sigma}^{0, 1}A$ is a simple central graded $*$-algebra with
\begin{align*}
    \type(\widetilde{\Sigma}^{0, 1}A) = - \type(A)-1. 
\end{align*} 
If $u \in A$ is a volume element of $A$, the element $u \otimes \beta^{\type(A)+1} \in \widetilde{\Sigma}^{0, 1}A$ is a volume element of $\widetilde{\Sigma}^{0, 1}A$. 
\end{lem}
\begin{proof}
The proof can be given by a case-by-case check of Table \ref{table:gscR}, and we leave the details to the reader. 
One possible simplification is to use the $\Z_2$-grading-preserving isomorphism $\widetilde{\Sigma}^{0, 1}\Sigma^{0, 1}A \simeq \Sigma^{1, 0}\widetilde{\Sigma}^{0, 1}A$ to reduce the problem to the case $\type(A) = 0$. 
\end{proof}

The following lemma can also be proved easily. 

\begin{lem}\label{lem_self_skew}
Let $A$ be a simple central graded $*$-algebra
and $S$ be a $\Sigma^{0,1}A$-module with an inner product. 
We also regard $S$ as a $\widetilde{\Sigma}^{0, 1}A$-module. 
Then we have a bijection
\begin{align}\label{eq_bij_End_self_skew}
   \psi_\beta \colon \mathrm{End}_{\Sigma^{0,1}A}(S) &\simeq
\mathrm{End}_{\widetilde{\Sigma}^{0, 1}A}(S) \\
f & \mapsto \beta^{|f|_{\Sigma^{0, 1}A}} f \notag
\end{align}
which preserves the $\Z_2$ grading. 
Here $|f|_{\Sigma^{0,1}A} \in \Z_2$ in the right hand side of \eqref{eq_bij_End_self_skew} is the $\Z_2$-grading of $f$ as an element in $\End_{\Sigma^{0, 1}A}(S)$. 
Moreover, the bijection restricts to the following bijections relating $\Self$ with $\Skew$, 
\begin{align*}
    \mathrm{Self}_{\Sigma^{0,1}A}(S) \simeq
\mathrm{Skew}_{\widetilde{\Sigma}^{0, 1}A}(S),  \quad &\mathrm{Self}^*_{\Sigma^{0,1}A}(S) \simeq
\mathrm{Skew}^*_{\widetilde{\Sigma}^{0, 1}A}(S),\\
 \mathrm{Skew}_{\Sigma^{0,1}A}(S) \simeq
\mathrm{Self}_{\widetilde{\Sigma}^{0, 1}A}(S),  \quad &\mathrm{Skew}^*_{\Sigma^{0,1}A}(S) \simeq
\mathrm{Self}^*_{\widetilde{\Sigma}^{0, 1}A}(S). 
\end{align*}
\end{lem}
Applying Lemma \ref{lem_self_skew} fiberwise, we also get the bijection in the bundle case, 
\begin{align}\label{eq_bij_End_self_skew_bundle}
   \psi_\beta \colon \mathrm{End}_{\Sigma^{0,1}\mathcal{A}}(\slashed S) &\simeq
\mathrm{End}_{\widetilde{\Sigma}^{0, 1}\mathcal{A}}(\slashed S).  
\end{align}

\begin{rem}\label{rem_complex_self_skew}
In the $\C$-linear setting, we simply have the bijection
\begin{align}
\End^1_A(S) \simeq \End^1_A(S), \ f \mapsto \sqrt{-1}f, 
\end{align}
which transforms $\Skew$ to $\Self$. 
We use this bijection when we relate $K_+$ with $K_-$ in Section \ref{sec_hat_K}. 
\end{rem}

\subsection{The topological model \texorpdfstring{$KO_-$}{KO-}}
In this subsection we define another model of the (twisted) $KO$-theory in terms of skew-adjoint operators. 
The functor $KO_-^A$ is shown to be a model of $KO^{-\type(A) - 2}$.

We start with the untwisted groups. 
The definition of $KO_-$ is given by simply replacing $\Self$ with $\Skew$ in Definitions \ref{def_KO_karoubi_untwisted_triple} and \ref{def_karoubi_untwisted_KO} of untwisted $KO_+$, as follows.  
\begin{defn}
Let $(X, Y)$ be a finite CW-pair and $A$ be a simple central graded $*$-algebra. 
\begin{itemize}
\item
A {\it triple} $( S, m_0, m_1)$ on $(X, Y)$ consists of an $A$-module $ S$ with an inner product and two continuous maps $m_0, m_1 \in \mathrm{Map} (X,  \Skew_{A}^*(S))$ with $m_0|_Y = m_1|_Y$.  
Such $m_i$'s are called {\it mass terms} on $S$ (c.f. Subsubsection \ref{subsubsec_intro_phys_interpretation}). 
\item 
Triples $( S, m_0, m_1)$ and $( S', m'_0, m'_1)$ on $(X, Y)$ are {\it isomorphic} if there exists an isometric isomorphism of $A$-modules $f \colon  S \simeq  S'$ such that $f \circ m_i = m'_i\circ f$ for $i = 0, 1$. 
\end{itemize}
\end{defn}

\begin{defn}[{$KO_-^{A}(X, Y)$}]\label{def_mass_untwisted_KO}
Let $(X, Y)$ and $A$ be as above. 
\begin{itemize}
\item
On the set $M^{p,q}_-(X, Y)$ of isomorphism classes of triples $( S, m_0, m_1)$ on $(X, Y)$, we introduce an abelian monoid structure by the direct sum. 

\item
We define $Z^{A}_-(X, Y)$ to be the submonoid of $M^{A}_-(X, Y)$ consisting of isomorphism classes of triples $( S, m_0, m_1)$ on $(X, Y)$ such that there exists a homotopy between $m_0$ and $m_1$ which is constant on $Y$, i.e., a map $m_I \in \mathrm{Map}(I \times X, \Skew^*_{A}( S))$ with $m_I|_{\{i\} \times X} = m_i$ for $i = 0, 1$ and $m_I|_{\{t\} \times Y} = m_0|_Y$ for all $t \in I$. 

\item
We define $KO^{A}_-(X, Y) = M^{A}_-(X, Y)/Z^{A}_-(X, Y)$ to be the quotient monoid.

\end{itemize}
\end{defn}

Next we define twisted groups. 
The category of twists $\Tw^2_{{KO}_-}$ in ${KO}_-$ is just the same as that for ${KO}_+$ in Definition \ref{def_twist_cat_KO_+}, 
\begin{align}
    \Tw^2_{{KO}_-} := \Tw^2_{{KO}_+}. 
\end{align}
The definition of ${KO}_-$ is given by modifiying Definition \ref{def_karoubi_twisted_KO} for twisted $KO_+$ in the same way, so we go briefly. 
\begin{defn}[{$KO^\mathcal{A}_-(X, Y)$}]\label{def_mass_twisted_KO}
Let $(X, Y, \mathcal{A}) \in \Tw^2_{{KO}_-}$.  
By replacing $\Self$ to $\Skew$ in Definition \ref{def_karoubi_twisted_KO}, we define the abelian group $KO^\mathcal{A}_-(X, Y)$. 
\end{defn}
Thus elements in $KO^\mathcal{A}_-(X, Y)$ is of the form $[\slashed S, m_0, m_1]$ with $m_0, m_1 \in \Gamma( X ,  \Skew_\mathcal{A}^*(S))$ such that $m_0|_Y = m_1|_Y$. 
We have $[\slashed S, m_0, m_1] = 0$ if $m_0$ and $m_1$ are homotopic relative to $Y$.

The result corresponding to Fact \ref{fact_isom_negligible} also holds for $KO_-$. 
For $E = E^0 \oplus E^1$ and any $\mathcal{A}$ as in the setting there, we have a canonical isomorphism using Lemma \ref{lem_periodicity_negligible_bundle}, 
\begin{align}\label{eq_isom_negligible_-}
    KO^{\mathcal{A}}_-(X, Y) \simeq KO^{\mathrm{End}(E) \widehat{\otimes} \mathcal{A}}_-(X, Y). 
\end{align}

We can relate $KO_-$ with $KO_+$ using the results in Subsection \ref{subsec_self_skew}. 

\begin{prop}\label{prop_KO_skew=self}
Let $(X, Y, \mathcal{A}) \in \Tw^2_{{KO}_-}$. 
We have an isomorphism
\begin{align}\label{eq_prop_KO_skew=self}
    KO_-^{\Sigma^{0, 1}\mathcal{A}}(X, Y) \simeq KO_+^{\widetilde{\Sigma}^{0, 1}\mathcal{A}}(X, Y),  
\end{align}
which is natural in $(X, Y)$ and $\mathcal{A}$. 
\end{prop}
\begin{proof}
The map \eqref{eq_prop_KO_skew=self} is given by
\begin{align*}
    [\slashed S, m_0, m_1] \mapsto [\slashed S, \psi_\beta(m_0), \psi_\beta(m_1)], 
\end{align*}
where $\psi_\beta \colon \mathrm{Skew}^*_{\Sigma^{0,1}\mathcal{A
}}(\slashed S) \simeq
\mathrm{Self}^*_{\widetilde{\Sigma}^{0, 1}\mathcal{A}}(\slashed S)$ is the isomorphism in \eqref{eq_bij_End_self_skew_bundle}. 
The well-definedness and the fact that this map gives a natural isomorphism are obvious. 
\end{proof}

Since we have $\Sigma^{0, 1}\Sigma^{1, 0} \mathcal{A} = \mathcal{A} \hat{\otimes} Cl_{1, 1}$ and $Cl_{1, 1}$ is negligible, by Proposition \ref{prop_KO_skew=self} and \eqref{eq_isom_negligible_-}, we have a natural isomorphism
\begin{align}
    KO_-^{\mathcal{A}}(X, Y) \simeq KO_+^{\widetilde{\Sigma}^{0, 1}\Sigma^{1, 0}\mathcal{A}}(X, Y). 
\end{align}
Thus we see that the two models $KO_+$ and $KO_-$ are equivalent under explicit algebraic operations. 
In particular, in the case of untwisted groups, also using Lemma \ref{lem_ungraded_susp_degree} we have the following result corresponding to Fact \ref{fact_karoubi_KO}. 

\begin{prop}\label{prop_mass_KO}
Let $A$ be a simple central graded $*$-algebra. 
We have a natural isomorphism on the category of finite CW-pairs, 
\begin{align}
    KO_-^A \simeq KO^{-\type(A) - 2}. 
\end{align}
\end{prop}

\begin{rem}\label{rem_Karoubi_K_self_skew}
In the $\C$-linear setting, just replacing the coefficients $\R$ by $\C$ in Definitions \ref{def_mass_untwisted_KO} and \ref{def_mass_twisted_KO}
as we did in Remark \ref{rem_Karoubi_K}, 
we define groups $K_-^A(X, Y)$ and $K_-^{\mathcal{A}}(X, Y)$. 
In this case, we have an isomorphism
\begin{align}\label{eq_isom_self_skew_K}
    K_-^{\mathcal{A}} \simeq K_+^{\mathcal{A}}
\end{align}
by $[\slashed S, m_0, m_1] \mapsto [\slashed S, \sqrt{-1}m_0, \sqrt{-1}m_1]$ (see Remark \ref{rem_complex_self_skew}). 
By \eqref{eq_rem_karoubi_K} and \eqref{eq_isom_self_skew_K}, we have a natural isomorphism on the category of finite CW-pairs, 
\begin{align}
    K_-^A \simeq K^{\type(A)}. 
\end{align}
\end{rem}

\subsection{The Pontryagin character forms for mass terms}\label{subsec_Ph_form_-}
Recall that the definition of the differential model $\widehat{KO}_+^A$ used the Pontryagin character forms $\Ph_\lself(h)$ for invertible elements $h \in C^\infty(X, \Self^*_A(S))$. 
In this subsection, we define their skew-adjoint variants. 

\subsubsection{The case of trivial bundles}\label{subsubsec_Ph_triv_-}
Here we explain the skew-adjoint version of Subsubsection \ref{subsubsec_Ph_triv}. 
Let $X$ be a manifold and $A$ be nondegenerate. 
We define the {\it Pontryagin character form} $\Ph_\lskew(m)$ for an invertible element (a mass term) $m \in C^\infty(X, \Skew^*_A(S))$. 

For such $m$, consider the section $tm \in C^\infty((0, \infty) \times X, \Skew^*_A(S))$, where $t$ is the coordinate in $(0, \infty)$. 
We have 
\begin{multline}
    \Ph_\lskew(d_{(0, \infty) \times X} + tm; d)  = dt \wedge \mathrm{Tr}_A\left(me^{td_X m + t^2 m^2} \right)  \\
    \in  \Omega_{\mathrm{clo}}^{4\Z -\mathrm{type}(\mathcal{A})-1}(X; \mathrm{Ori}(\mathcal{A})). 
\end{multline}
Here the first term is defined in \eqref{eq_Ph_skew}, in which $d$ is the trivial connection on $\underline{A}$. 
Since $m^2$ is strictly negative, the convergence result analogous to Lemma \ref{lem_Ph_convergence_triv} holds and allows us to define the following. 

\begin{defn}[{$\Ph_\lskew(m)$}]\label{def_Ph_skew_triv}
For $m \in C^\infty(X, \Skew^*_{A}(S))$, its {\it Pontryagin character form} 
\begin{align*}
    \mathrm{Ph}_\lskew(m) \in \Omega^{4\Z - \type(A) - 2}_{\mathrm{clo}}(X; \Ori(A))
\end{align*}
is defined by
\begin{align}
    \mathrm{Ph}_\lskew(m) &:=   -\pi^{-1/2}\mathcal{R} \circ\int_{(0, \infty)} \mathrm{Ph}_\lskew(d_{(0, \infty) \times X} + tm) \\
    &=-\pi^{-1/2}\mathcal{R} \circ\int_{(0, \infty)} dt \wedge \mathrm{Tr}_A\left(m e^{td_X m + t^2m^2} \right). \notag
\end{align}
\end{defn}
We also define
\begin{defn}[{$\mathrm{CS}_\lskew(m_I)$}]\label{def_CS_skew_triv}
For $m_I \in C^\infty(I \times X , \Skew^*_{A}(S))$, we define its {\it Chern-Simons form} by
\begin{align} 
    \mathrm{CS}_\lskew(m_I) := \int_{I}\mathrm{Ph}_\lskew\left(m_I \right) 
    \in \Omega^{4\Z - \type(A)-3}(X;\Ori(A)). 
\end{align}
\end{defn}

\subsubsection{The general case}
Now we explain the skew-adjoint version of Subsection \ref{subsec_Ph_form}. 
Let $X$, $\mathcal{A}$, $\nabla^{\mathcal{A}}$ and $\slashed S$ with an inner product be as in that subsection, where we always assume that the fibers of $\mathcal{A}$ are nondegenerate. 
Let $\nabla^{\slashed S}$ be a skew-adjoint $\mathcal{A}$-{\it connection} on $\slashed S$ compatible with $\nabla^{\mathcal{A}}$. 
In this setting, given a smooth invertible section $m \in C^\infty( X ; \mathrm{Skew}_{\mathcal{A}}^*(\slashed S))$, we are going to define its {\it Pontryagin character form} $\mathrm{Ph}_\lskew(m; \nabla^{\mathcal{A}}, \nabla^{\slashed S})$.  

On the manifold $(0, \infty) \times X$,
We consider the skew-adjoint $\mathrm{pr}_X^*\mathcal{A}$-superconnection $\mathrm{pr}_X^*\nabla^{\slashed S} + tm$ on $\mathrm{pr}_X^*\slashed S$ compatible with $\mathrm{pr}_X^*\nabla^{\mathcal{A}}$. 
We have
\begin{align*}
   \mathrm{Ph}_\lskew(\mathrm{pr}_X^*\nabla^{\slashed S} + tm; \mathrm{pr}_X^*\nabla^{\mathcal{A}}) \in  \Omega_{\mathrm{clo}}^{4\Z-\mathrm{type}(\mathcal{A})-1}((0, \infty) \times X ; \mathrm{pr}_X^*\mathrm{Ori}(\mathcal{A})). 
\end{align*}
Using the convergence result analogous to Lemma \ref{lem_Ph_convergence}, we define the following. 

\begin{defn}[{$\mathrm{Ph}(m; \nabla^{\mathcal{A}}, \nabla^{\slashed S})$, $\mathrm{CS}(m_I; \nabla^{\mathcal{A}}, \nabla^{\slashed S})$}]\label{def_Ph_skew}
Let $X$, $\mathcal{A}$, $\nabla^{\mathcal{A}}$ and $\slashed S$ be as above.
\begin{enumerate}
    \item For any smooth invertible section $m \in C^\infty( X ; \mathrm{Skew}_{\mathcal{A}}^*(\slashed S))$, we define
\begin{multline}
    \mathrm{Ph}_\lskew(m; \nabla^{\mathcal{A}}, \nabla^{\slashed S}) :=  - \pi^{-1/2}\mathcal{R} \circ\int_{(0, \infty)} \mathrm{Ph}_\lskew(\mathrm{pr}_X^*\nabla^{\slashed S} + tm; \mathrm{pr}_X^*\nabla^{\mathcal{A}}) \\ 
    \in \Omega^{4\Z - \type(\mathcal{A}) - 2}(X; \Ori(\mathcal{A}))
\end{multline}
\item  
For any smooth invertible section $m_I \in C^\infty(I \times X ; \mathrm{pr}_X^*\mathrm{Skew}_{\mathcal{A}}^*(\slashed S))$, we define
\begin{multline} 
    \mathrm{CS}_\lskew(m_I; \nabla^{\mathcal{A}}, \nabla^{\slashed S}) := \int_{I}\mathrm{Ph}_\lskew\left(m_I; \mathrm{pr}_X^*\nabla^{\mathcal{A}}, \mathrm{pr}_X^*\nabla^{\slashed S} \right) \\
    \in \Omega^{4\Z - \type(\mathcal{A})-3}(X; \Ori(\mathcal{A})). 
\end{multline}
\end{enumerate}
\end{defn}

\subsubsection{The relation between $\Ph_\lskew$ and $\Ph_\lself$}\label{subsubsec_Ph_self_skew}
In Subsection \ref{subsec_self_skew}, we saw that we have an isomorphism $\psi_\beta \colon \mathrm{Skew}^*_{\Sigma^{0,1}\mathcal{A}}(\slashed S) \simeq
\mathrm{Self}^*_{\widetilde{\Sigma}^{0, 1}\mathcal{A}}(\slashed S)$. 
In this subsection, we show that the Pontryagin character forms $\Ph_\lskew$ and $\Ph_\lself$ are also related under this isomorphism. 

Let $(\mathcal{A}, \nabla^{\mathcal{A}})$ be a bundle of simple central graded $*$-algebras equipped with a connection. 
The induced connection on $\Sigma^{0,1}\mathcal{A}$ and $\widetilde{\Sigma}^{0,1}\mathcal{A}$ are denoted by $\Sigma^{0,1}\nabla^{\mathcal{A}}$ and $\widetilde{\Sigma}^{0,1}\nabla^{\mathcal{A}}$, respectively. 
Assume that a $\Sigma^{0, 1}\mathcal{A}$-module bundle $\slashed S$ is equipped with an inner product and let $\Grad^{\slashed S}$ be a $\Sigma^{0, 1}\mathcal{A}$-superconnection compatible with $\Sigma^{0,1}\nabla^{\mathcal{A}}$. 
Then decompose $\Grad^{\slashed S}$ as in \eqref{eq_superconn_sum}, 
\begin{align*}
    \Grad^{\slashed S} = \nabla + \sum_j \omega_j \widehat{\otimes}\xi_j , 
\end{align*}
where $\nabla$ is an $\mathcal{A}$-connection, $\omega_j \in \Omega^j(X)$ and $\xi_j \in C^\infty(X; \mathrm{End}_{\Sigma^{0,1}\mathcal{A}}(\slashed S))$ with $|\omega_j| + |\xi_j| = 1 \in \Z_2$. 
Then consider the following operator on $\Omega^*(X; \slashed S)$ 
\begin{align}\label{eq_def_psi_beta_superconn}
    \psi_\beta (\Grad^{\slashed S}): = \nabla + \sum_j (-1)^{\nu(j)} \omega_j \widehat{\otimes}\psi_\beta(\xi_j ), 
\end{align}
where 
\begin{align}\label{eq_def_nu}
    \nu(j) := \begin{cases}
    0 & j \equiv 0, 1 \pmod 4 \\
    1 & j \equiv 2,3 \pmod 4. 
    \end{cases}
\end{align}
The expression \eqref{eq_def_psi_beta_superconn} does not depend on the decomposition \eqref{eq_superconn_sum}. 

\begin{lem}\label{lem_psi_beta_superconn}
The operator $\psi_\beta (\Grad^{\slashed S})$ in \eqref{eq_def_psi_beta_superconn} is a $\widetilde{\Sigma}^{0,1}\mathcal{A}$-superconnection on $\slashed S$ compatible with $\widetilde{\Sigma}^{0,1}\nabla^{\mathcal{A}}$. 
Moreover, if $\Grad^{\slashed S} $ is a skew-adjoint superconnection, then $\psi_\beta (\Grad^{\slashed S})$ is a self-adjoint superconnection. 
\end{lem}
\begin{proof}
The first statement follows from the fact that $\psi_\beta$ preserves the $\Z_2$-grading. 
The second statement follows from a degree-wise check by using the definition of self/skew-adjointness of superconnections (Definition \ref{def_skew_superconn}). 
\end{proof}

\begin{prop}\label{prop_skew_self_characteristic_form}
In the above setting, assume that $\Grad^{\slashed S} $ is skew-adjoint. 
For any element $f \in \R[[z]]$, we have 
\begin{align}\label{eq_skew_self_characteristic_form}
    \mathrm{Tr}_{\Sigma^{0, 1}\mathcal{A}}(f(F(\Grad^{\slashed S};  \Sigma^{0, 1}\nabla^\mathcal{A})))
    = \mathrm{Tr}_{\widetilde{\Sigma}^{0, 1}\mathcal{A}}(f(-F(\psi_\beta(\Grad^{\slashed S});  \widetilde{\Sigma}^{0, 1}\nabla^\mathcal{A}))) \\
    \in \Omega_\mathrm{clo}^{4\Z - \type(\mathcal{A})}(X; \Ori(\mathcal{A})), \notag
\end{align}
under the following identifications of orientation bundles,
\begin{align*}
\Ori(\mathcal{A}) &\simeq \Ori(\Sigma^{0, 1}\mathcal{A}), &
u &\mapsto u \widehat{\otimes} \beta, \\
\Ori(\mathcal{A}) &\simeq \Ori(\widetilde{\Sigma}^{0, 1}\mathcal{A}), & 
u &\mapsto (-1)^{\nu(\type(\mathcal{A}))} u \widehat{\otimes} \beta^{\type(\mathcal{A})+1},
\end{align*}
where $\nu$ is given by \eqref{eq_def_nu}. 
\end{prop}
\begin{proof}
The both sides of \eqref{eq_skew_self_characteristic_form} are closed forms of degree $-\type(\mathcal{A})\pmod 4$ by Proposition \ref{prop_chracteristic_form_mod4} and Lemma \ref{lem_ungraded_susp_degree}. 

In this proof, we denote 
\begin{align*}
F &:= F(\Grad^{\slashed S};  \Sigma^{0, 1}\nabla^\mathcal{A}),
&
\widetilde{F} &:= F(\psi_\beta(\Grad^{\slashed S});  \widetilde{\Sigma}^{0, 1}\nabla^\mathcal{A}).
\end{align*}
Also, for elements $\Xi \in \Omega^*(X; \End(\slashed S))$, we denote by $\Xi = \sum_{j \in \Z_4}\Xi_j$ the decomposition with respect to the form-degree modulo $4$, in which $\Xi_j \in \Omega^{4\Z + j}(U; \End(\slashed S))$. 

It is enough to prove in the case that $f(z) = z^n$ for each $n$. 
We may work locally, and use the local expression as in \eqref{eq_superconn_local}. 
As explained there, we may write $\nabla^{\mathcal{A}} = d + C$ for some $C \in \Omega^1(U; A^0_{\mathrm{skew}})$, and $\Grad^{\slashed S} = d + C + B$ for some $B \in \Omega^*(U; \mathrm{End}_{\Sigma^{0, 1}\mathcal{A}}(\slashed S))^1$. 
Then we have
\begin{align*}
    \psi_\beta(\Grad^{\slashed S}) = d + C + \beta \cdot B_0 + B_1 - \beta \cdot B_2 - B_3. 
\end{align*}
Here we denote $\beta \cdot (\omega \widehat{\otimes}\xi) := \omega \widehat{\otimes}( \beta \xi)$ for $\omega \widehat{\otimes}\xi \in \Omega^*(U; \mathrm{End}_{\Sigma^{0, 1}\mathcal{A}}(\slashed S))$. 
Let us denote $\widetilde{B} := \beta \cdot B_0 + B_1 - \beta \cdot B_2 - B_3$. 
Then we have \eqref{eq_curv_local}
\begin{align*}
    F &= dB + B^2, &
    \widetilde{F}
    &= d\widetilde{B} + \widetilde{B}^2. 
\end{align*}
Using $\beta \cdot B_j \cdot = (-1)^{|j|+1}B_j \cdot \beta \cdot$ as operators on $\Omega^*(U; \mathrm{End}_{\Sigma^{0, 1}\mathcal{A}}(\slashed S))$, we can easily check that
\begin{align*}
    \widetilde{F}_0 &= -F_0, & 
    \widetilde{F}_1 &= \beta \cdot F_1, &
    \widetilde{F}_2 &= F_2, & 
    \widetilde{F}_3 &= -\beta \cdot  F_3. 
\end{align*}
It holds that $\beta \cdot F_j \cdot = (-1)^{|j|}F_j \cdot \beta \cdot$ since $F \in \Omega^*(U; \mathrm{End}_{\Sigma^{0, 1}\mathcal{A}}(\slashed S))^0$.
Using this formula, we can check that
\begin{align*}
    ((-\widetilde{F})^n)_0 &= (F^n)_0, & 
    ((-\widetilde{F})^n)_1 &= -\beta \cdot (F^n)_1, \\ ((-\widetilde{F})^n)_2 &= - (F^n)_2, & ((-\widetilde{F})^n)_3 &= \beta\cdot (F^n)_3,
\end{align*}
by an induction on $n$. 
As we have checked at the beginning of the proof, the both sides of \eqref{eq_skew_self_characteristic_form} are forms of degree $-\type(\mathcal{A})\pmod 4$. Using this fact, we see that
\begin{align*}
    \mathrm{Tr}_{u\widehat{\otimes} \beta}(F^n) &= \mathrm{Tr}_{u \widehat{\otimes} \beta}((F^n)_{-\type(\mathcal{A})}) \\
    &= \mathrm{Tr}_{(-1)^{\nu(\type(\mathcal{A}))} u \widehat{\otimes} \beta^{\type(\mathcal{A})+1}}(((-\widetilde{F})^n)_{-\type(\mathcal{A})}) \\
    &= \mathrm{Tr}_{(-1)^{\nu(\type(\mathcal{A}))} u \widehat{\otimes} \beta^{\type(\mathcal{A})+1}}((-\widetilde{F})^n), 
\end{align*}
as desired. 
\end{proof}

By Proposition \ref{prop_skew_self_characteristic_form}, we get the following. 
\begin{cor}\label{cor_Ph_self_skew}
Let $X$ and $(\mathcal{A}, \nabla^{\mathcal{A}})$ be as above and $\slashed S$ be a $\Sigma^{0, 1}\mathcal{A}$-module bundle with inner product. 
Then for any skew-adjoint $\Sigma^{0, 1}\mathcal{A}$-superconnection $\Grad^{\slashed S}$ compatible with $\Sigma^{0, 1}\nabla^\mathcal{A}$, we have
\begin{align}
    \Ph_\lskew(\Grad^{\slashed S}; \Sigma^{0, 1}\nabla^\mathcal{A}) = \Ph_\lself (\psi_\beta(\Grad^{\slashed S}); \widetilde{\Sigma}^{0, 1}\nabla^\mathcal{A}). 
\end{align}
On the manifold $I \times X$, for any skew-adjoint $\mathrm{pr}_X^*(\Sigma^{0, 1}\mathcal{A})$-superconnection $\Grad^{\slashed S, I}$ compatible with $\mathrm{pr}_X^*(\Sigma^{0, 1}\nabla^\mathcal{A})$, we have
\begin{align}
    \CS_\lskew(\Grad^{\slashed S, I}; \Sigma^{0, 1}\nabla^\mathcal{A}) = \CS_\lself (\psi_\beta(\Grad^{\slashed S, I}); \widetilde{\Sigma}^{0, 1}\nabla^\mathcal{A}). 
\end{align}
Here the orientation bundles are identified as in Proposition \ref{prop_skew_self_characteristic_form}. 
\end{cor}
Using this corollary, we also get the correspondence of the Pontryagin character forms for mass terms and gradations. 
\begin{cor}\label{cor_Ph_mass_gradation}
Let $X$ and $(\mathcal{A}, \nabla^{\mathcal{A}})$ be as above and $\slashed S$ be a $\Sigma^{0, 1}\mathcal{A}$-module bundle with inner product, equipped with a skew-adjoint $\Sigma^{0, 1}\mathcal{A}$-connection $\nabla^{\slashed S}$. 
Then for any element $m \in C^\infty(X; \Skew^*_{\Sigma^{0, 1}\mathcal{A}}(\slashed S))$, we have 
\begin{align}
    \mathrm{Ph}_\lskew(m; \Sigma^{0, 1}\nabla^{\mathcal{A}}, \nabla^{\slashed S})
    =\mathrm{Ph}_\lself(\psi_\beta(m); \widetilde{\Sigma}^{0, 1}\nabla^{\mathcal{A}}, \nabla^{\slashed S}),
\end{align}
in which the identity $\psi_\beta(\nabla^{\slashed S}) = \nabla^{\slashed S}$ is used. For any element $m_I \in C^\infty(I \times X ; \mathrm{pr}_X^*\mathrm{Skew}_{\Sigma^{0, 1}\mathcal{A}}^*(\slashed S))$, we also have
\begin{align}
    \mathrm{CS}_\lskew(m_I; \Sigma^{0, 1}\nabla^{\mathcal{A}}, \nabla^{\slashed S})
    =\mathrm{CS}_\lself(\psi_\beta(m_I); \widetilde{\Sigma}^{0, 1}\nabla^{\mathcal{A}}, \nabla^{\slashed S}). 
\end{align}
Here the orientation bundles are identified as in Proposition \ref{prop_skew_self_characteristic_form}. 
\end{cor}

\subsection{The differential model}

\subsubsection{The untwisted groups $\widehat{KO}_-^A(X, Y)$}
Now we define the untwisted differential $KO$-theory $\widehat{KO}_-$. 
As in the case for topological $KO_-$, the definition is given by just replacing $\Self$ to $\Skew$ and changing the form degree from $(4\Z + \type(A) -1)$ to $(4\Z - \type(A) - 3)$ in Definitions \ref{def_KO_hat_quadruple} and \ref{def_untwisted_KO_hat_m} of $\widehat{KO}_+$, as follows. 
\begin{defn}\label{def_untwisted_KO_-}
Let $A$ be a nondegenerate simple central graded $*$-algebra. 
Let $(X, Y)$ be an object in $\mathrm{MfdPair}_f$. 
\begin{itemize}
    \item A {\it $\widehat{KO}_-$-cocycle} $(S, m_0, m_1, \eta)$ on $(X, Y)$ consists of an $A$-module $S$ with an inner product, two smooth maps $m_0, m_1 \in C^\infty( X , \mathrm{Skew}_{A}^*(S))$ such that $m_0|_Y = m_1|_Y$, and an element
    \[
        \eta \in \Omega^{4\Z - \type(A) - 3}(X, Y; \Ori(A))/ \mathrm{Im}(d).
    \] 
    \item Two $\widehat{KO}_-$-cocycles $(S, m_0, m_1, \eta)$ and $(S', m'_0, m'_1, \eta')$ are {\it isomorphic}
    if there exists an isometric isomorphism of $A$-modules $f \colon S \simeq  S'$ such that $f\circ m_i = m'_i \circ f$ for $i = 0, 1$, and we have $\eta = \eta'$. 
\end{itemize}
\end{defn}

\begin{defn}[{$\widehat{KO}_-^{A}(X, Y)$}] 
Let $A$ and $(X, Y)$ be as in Definition \ref{def_untwisted_KO_-}
\begin{itemize}
    \item
We introduce an abelian monoid structure on the set $\widehat{M}^{A}_-(X, Y)$ of isomorphism classes of $\widehat{KO}_-$-cocycles $(S, m_0, m_1, \eta)$ on $(X, Y)$ by
\begin{align*}
    [S, m_0, m_1, \eta] + [S', m'_0, m'_1, \eta']
    = [S \oplus  S', m_0 \oplus m'_0,m_1 \oplus m'_1, \eta +\eta']. 
\end{align*}
\item We define $\widehat{Z}^{A}_-(X, Y)$ to be the submonoid of $\widehat{M}^{A}_-(X, Y)$ consisting of elements of the form
\begin{align*}
    [S, m_0, m_1, \mathrm{CS}_\lskew(m_I)], 
\end{align*}
where $m_I$ is a smooth homotopy between $m_0$ and $m_1$ which is constant on $Y$, i.e., a smooth map $m_I \in C^\infty(I \times X, \Skew^*_{A}(S))$ with $m_I|_{\{i\} \times X} = m_i$ for $i = 0, 1$ and $m_I|_{\{t\} \times Y} = m_0|_Y$ for all $t \in I$. 
\item We define $\widehat{KO}^{A}_-(X, Y):=\widehat{M}^{A}_-(X, Y) / \widehat{Z}^{A}_-(X, Y)$. 
\end{itemize}
\end{defn}

In a way similar to the proof of Lemma \ref{lem_untwisted_additive_inverse}, we can show that $\widehat{KO}^{A}_-(X, Y)$ is an abelian group, and thus $\widehat{KO}^{\mathcal{A}}$ defines a functor $\mathrm{MfdPair}^{\mathrm{op}}_f \to \mathrm{Ab}$. 

\begin{defn}[{Structure homomorphisms for $\widehat{KO}^{A}_-$}]\label{def_str_hom_untwisted_-}
We define the following structure homomorphisms, which are natural in $(X, Y)$. 
\begin{align*}
        R &\colon \widehat{KO}_-^{A}(X, Y) \to \Omega^{4\Z - \type(A) - 2}_{\mathrm{clo}}(X, Y; \Ori(A))\\
       &[S, m_0, m_1, \eta] \mapsto \mathrm{Ph}_\lskew(m_1) - \mathrm{Ph}_\lskew(m_0) + d\eta.\\
        I &\colon \widehat{KO}_-^{A}(X, Y) \to KO_-^{A}(X, Y) \\
       & [S, m_0, m_1, \eta] \mapsto [S, m_0, m_1]. \\
        a &\colon \Omega^{4\Z - \type(A) - 3}(X, Y; \Ori(A)) / \mathrm{Im}(d) \to \widehat{KO}_-^{A}(X, Y)  \\
       & \eta \mapsto [0, 0, 0, \eta]. 
\end{align*}
\end{defn}

We have the natural isomorphism corresponding to Proposition \ref{prop_isom_negligible_untwisted}, 
\begin{align}\label{eq_isom_negligible_-_hat}
    \widehat{KO}^{A}_-\simeq \widehat{KO}^{\mathrm{End}(V) \widehat{\otimes} A}_-,
\end{align}
by sending $[S, m_0, m_1, \eta]$ to $[E \otimes  S, \psi^E(m_0), \psi^E(m_1), \eta] $. 

Moreover, we have the following refinement of the untwisted version of Proposition \ref{prop_KO_skew=self}. 
\begin{prop}\label{prop_hat_KO_skew=self}
We have a natural isomorphism
\begin{align}\label{eq_prop_hat_KO_skew=self}
    \widehat{KO}_-^{\Sigma^{0, 1}A}\simeq  \widehat{KO}_+^{\widetilde{\Sigma}^{0, 1}A}
\end{align}
which is compatible with the corresponding isomorphisms of other functors appearing in Definitions \ref{def_str_hom_untwisted_-} and \ref{def_str_hom_untwisted} via the structure homomorphisms. 
Here we use Proposition \ref{prop_KO_skew=self} and the identification $\Ori(\Sigma^{0, 1}A) \simeq \Ori(\widetilde{\Sigma}^{0, 1}A)$  in Proposition \ref{prop_skew_self_characteristic_form}. 
\end{prop}
\begin{proof}
The transformation \eqref{eq_prop_hat_KO_skew=self} on an object $(X, Y)$ in $\mathrm{MfdPair}_f$ is given by
\begin{align*}
    [S, m_0, m_1, \eta] \mapsto [ S, \psi_\beta(m_0), \psi_\beta(m_1), \eta], 
\end{align*}
where $\psi_\beta \colon \mathrm{Skew}^*_{\Sigma^{0,1}A}(S) \simeq
\mathrm{Self}^*_{\widetilde{\Sigma}^{0, 1}{A}}( S)$ is the isomorphism in \eqref{eq_bij_End_self_skew_bundle}. 
The well-definedness follows from Corollary \ref{cor_Ph_mass_gradation}. 
Since $\psi_\beta$ is an isomorphism, we easily see that this map gives a natural isomorphism. 
\end{proof}

Since we have $\Sigma^{0, 1}\Sigma^{1, 0} A = A \hat{\otimes} Cl_{1, 1}$ and $Cl_{1, 1}$ is negligible, by Proposition \ref{prop_KO_skew=self} and \eqref{eq_isom_negligible_-}, we have a natural isomorphism
\begin{align}\label{eq_isom_hat_KO_self_skew}
    \widehat{KO}_-^{A}\simeq \widehat{KO}_+^{\widetilde{\Sigma}^{0, 1}\Sigma^{1, 0}A}, 
\end{align}
which is compatible with the structure homomorphisms. 
Recall that by Theorem \ref{thm_KO_hat_untwisted} we know that $(\widehat{KO}^{\widetilde{\Sigma}^{0, 1}\Sigma^{1, 0}A}_+, R, I, a)$ is a differential extension of $KO^{\widetilde{\Sigma}^{0, 1}\Sigma^{1, 0}A}_+ \simeq KO^{-\type(A) -2}$. 
Thus we get the following. 
\begin{thm}\label{thm_KO_hat_untwisted_-}
The quadruple $\left(\widehat{KO}^A_-, R, I, a\right)$ is a differential extension of the $KO$-theory $KO^A_- \simeq KO^{-\type(A)-2}$ on $\mathrm{MfdPair}_f$. 
\end{thm}

The natural isomorphism \eqref{eq_isom_hat_KO_self_skew}, being an explicit algebraic manipulation, allows us to translate results on $\widehat{KO}_+$ into corresponding ones on $\widehat{KO}_-$. 

\subsubsection{The twisted groups $\widehat{KO}_-^{\mathcal{A}}(X, Y)$}
Now we define the twisted groups $\widehat{KO}_-^{\mathcal{A}}(X, Y)$. 
Again the category $\Tw^2_{\widehat{KO}_-}$ of twists for $\widehat{KO}_-$ is just the same as that for $\widehat{KO}_+$, 
\begin{align}
    \Tw^2_{\widehat{KO}_-} := \Tw^2_{\widehat{KO}_+}. 
\end{align}
The definition of the twisted $\widehat{KO}_-$ is also a simple modification of Definition \ref{def_twisted_KO_hat_m}, so we go briefly. 

\begin{defn}[{$\widehat{KO}_-^{\mathcal{A}}(X, Y)$}]\label{def_twisted_KO_hat_-}
Let $(X, Y, \mathcal{A}) \in \Tw^2_{\widehat{KO}_-}$. 
Replacing $\Self$ by $\Skew$ and changing the form degree from $(4\Z + \type(A) -1)$ to $(4\Z - \type(A) - 3)$ in Definition \ref{def_twisted_KO_hat_m}, we define the abelian group $\widehat{KO}_-^{\mathcal{A}}(X, Y)$. 
\end{defn}

An element of $\widehat{KO}_-^{\mathcal{A}}(X, Y)$ is of the form $[\slashed S, \nabla^\mathcal{A}, \nabla^{\slashed S}, m_0, m_1, \eta]$, where $m_i \in C^\infty( X ; \mathrm{Skew}_{\mathcal{A}}^*(\slashed S))$ are such that $m_0|_Y = m_1|_Y$, and 
\[
\eta \in \Omega^{4\Z - \type(\mathcal{A}) - 3}(X, Y; \Ori(\mathcal{A})) / \mathrm{Im}(d).
\]
We have $[\slashed S, \nabla^{\mathcal{A}}, \nabla^{\slashed S}, m_0, m_1, \mathrm{CS}_\lskew(m_I; \nabla^{\mathcal{A}}, \nabla^{\slashed S})] = 0$ for a homotopy $m_I$ from $m_0$ to $m_1$ relative to $Y$.  
We clearly have the functor
\begin{align*}
    \widehat{KO}_- \colon \Tw^2_{\widehat{KO}_-} \to \mathrm{Ab}.
\end{align*}

\begin{defn}[{Structure homomorphisms for $\widehat{KO}^\mathcal{A}_-$}]\label{def_str_hom_twisted_-}
We define the structure homomorphisms 
\begin{align*}
    R &\colon \widehat{KO}_-^\mathcal{A}(X, Y) \to \Omega^{4\Z -\type(\mathcal{A})-2}_{\mathrm{clo}}(X, Y; \Ori(\mathcal{A})), \\
    I &\colon \widehat{KO}_-^\mathcal{A}(X, Y) \to KO_-^\mathcal{A}(X, Y), \\
     a &\colon \Omega^{4\Z -\type(\mathcal{A})-3}(X, Y; \Ori(\mathcal{A})) / \mathrm{Im}(d) \to \widehat{KO}_-^\mathcal{A}(X, Y)
\end{align*}
applying the same modification as in Definition  \ref{def_twisted_KO_hat_-} to Definition \ref{def_str_hom}. 
\end{defn}

As a counterprt of Proposition \ref{prop_isom_negligible}, we have an isomorphism
\begin{align}
     \widehat{KO}^{\mathcal{A}}_-(X, Y) \simeq \widehat{KO}^{\End(E) \widehat{\otimes} \mathcal{A}}_-(X, Y), 
\end{align}
which are compatible with the structure homomorphisms. 
Moreover, as in the untwisted case (Proposition \ref{prop_hat_KO_skew=self}), we get the refinement of Proposition \ref{prop_KO_skew=self} as follows. 

\begin{prop}\label{prop_hat_KO_skew=self_twisted}
We have an isomorphism
\begin{align}\label{eq_prop_hat_KO_skew=self_twisted}
    \widehat{KO}_-^{\Sigma^{0, 1}\mathcal{A}}(X, Y)\simeq  \widehat{KO}_+^{\widetilde{\Sigma}^{0, 1}\mathcal{A}}(X, Y)
\end{align}
which is compatible with the corresponding isomorphisms of other functors appearing in Definitions \ref{def_str_hom_twisted_-} and \ref{def_str_hom} via the structure homomorphisms. 
Here we use Proposition \ref{prop_KO_skew=self} and the identification $\Ori(\Sigma^{0, 1}\mathcal{A}) \simeq \Ori(\widetilde{\Sigma}^{0, 1}\mathcal{A})$ in Proposition \ref{prop_skew_self_characteristic_form}. 
\end{prop}
\begin{proof}
Take a connection $\nabla^{\mathcal{A}}$ on $\mathcal{A}$. 
Recall that, by Lemma \ref{lem_twisted_KO_connection_V}, any element in $\widehat{KO}_+^{\widetilde{\Sigma}^{0, 1}\mathcal{A}}(X, Y)$ is represented by a $\widehat{KO}_+$-cocycle of the form $(\slashed S, \widetilde{\Sigma}^{0, 1}\nabla^\mathcal{A}, \nabla^{\slashed S}, h_0, h_1, \eta)$. 
The $\widehat{KO}_-$-version of the lemma is shown by exactly the same proof, so we can represent any element in $\widehat{KO}_-^{\Sigma^{0, 1}\mathcal{A}}(X, Y)$ by a $\widehat{KO}_-$-cocycle of the form $(\slashed S, {\Sigma}^{0, 1}\nabla^\mathcal{A}, \nabla^{\slashed S}, m_0, m_1, \eta)$. 

Then, the homomorphism \eqref{eq_prop_hat_KO_skew=self_twisted} is given by
\begin{align*}
    [\slashed S, {\Sigma}^{0, 1}\nabla^\mathcal{A}, \nabla^{\slashed S}, m_0, m_1, \eta] \mapsto [\slashed S,\widetilde{\Sigma}^{0, 1}\nabla^\mathcal{A},  \nabla^{\slashed S}, \psi_\beta(m_0), \psi_\beta(m_1), \eta], 
\end{align*}
where $\psi_\beta \colon \mathrm{Skew}^*_{\Sigma^{0,1}\mathcal{A}}(\slashed S) \simeq
\mathrm{Self}^*_{\widetilde{\Sigma}^{0, 1}\mathcal{A}}(\slashed S)$ is the isomorphism in \eqref{eq_bij_End_self_skew_bundle}. 
The well-definedness follows from Corollary \ref{cor_Ph_mass_gradation}, and the resulting map is independent on the choice of $\nabla^{\mathcal{A}}$ by Lemma \ref{lem_twisted_KO_connection_V} and its $\widehat{KO}_-$-version. 
Since $\psi_\beta$ is an isomorphism, we easily see that this map gives an isomorphism. 
\end{proof}

The above results give a twisted generalization of \eqref{eq_isom_hat_KO_self_skew}
\begin{align}\label{eq_isom_hat_KO_self_skew_twisted}
    \widehat{KO}_-^\mathcal{A}(X, Y)\simeq \widehat{KO}_+^{\widetilde{\Sigma}^{0, 1}\Sigma^{1, 0}\mathcal{A}}(X, Y), 
\end{align}
which is compatible with the structure homomorphisms. 
Recall that by Theorem \ref{thm_hat_KO_twisted} we know that $(\widehat{KO}_+, R, I, a)$ is a differential extension of the twisted $KO$-theory $KO_+$. 
Thus we get the following. 
\begin{thm}\label{thm_KO_hat_twisted_-}
The quadruple $\left(\widehat{KO}_-, R, I, a\right)$ is a differential extension of the twisted $KO$-theory $KO_-$ on $ \Tw^2_{\widehat{KO}_-}$. 
\end{thm}

Again, the natural isomorphism \eqref{eq_isom_hat_KO_self_skew_twisted} allows us to translate results on twisted $\widehat{KO}_+$ into corresponding ones on twisted $\widehat{KO}_-$. 

\section{The complex case : \texorpdfstring{$\widehat{K}_+$}{K+} and \texorpdfstring{$\widehat{K}_-$}{K-}}\label{sec_hat_K}
So far we have worked in the $\R$-linear setting. 
In this section, we explain that the $\C$-linear version of the above story works by straightforward modifications, producing differential extensions of the Karoubi's $K$-theory $K_+$ (Remark \ref{rem_Karoubi_K}). 
\subsection{The complex superconnections}\label{subsec_complex_superconn}
In this subsection we explain the $\C$-linear version of Section \ref{sec_superconn}. 
In the $\C$-linear setting, we work with simple central graded $*$-algebras over $\C$. 
The traces $\mathrm{Tr}_u$ and $\mathrm{Tr}_A$ are defined exactly as in Subsection \ref{subsec_Clifford_end}, by replacing $\dim_\R$ with $\dim_\C$. 

We explain the {\it complex superconnection formalism} by modifying Subsection \ref{subsec_superconn_general}. 
We start with a bundle of nondegenerate simple central graded $*$-algebras $\mathcal{A}$ over $\C$ and an $\mathcal{A}$-module bundle $\slashed S$ with an inner product on a manifold. 
Then we define $\mathcal{A}$-superconnections $\Grad^{\slashed S}$, the curvature $F(\Grad^{\slashed S}; \nabla^{\mathcal{A}})$, the self/skew-adjointness of superconnections exactly in the same way as in that subsection. 
We easily see that the obvious $\C$-linear version of the results until Remark \ref{rem_self_skew_Dirac} hold. 
In the $\C$-linear case, we have (recall $\type(\mathcal{A}) \in H^0(X; \Z_2)$ in the complex settings)
\begin{align}\label{eq_complex_mod_2}
    \mathrm{Tr}_{\mathcal{A}}(f(F(\Grad^{\slashed S};  \nabla^\mathcal{A}))) \in  \Omega_{\mathrm{clo}}^{2\Z +\mathrm{type}(\mathcal{A})+1}(X; \mathrm{Ori}(\mathcal{A})\otimes_\R \C). 
\end{align}
for a $\mathcal{A}$-superconnection $\Grad^{\slashed S}$, and any $f \in \C[[z]]$. 
For $\Grad^{\slashed S}$ and $\Grad^{\slashed S, I}$ self/skew-adjoint, we define
\begin{align*}
    \Ch_{\lself / \lskew}(\Grad^{\slashed S};  \nabla^\mathcal{A}) &:=   \mathrm{Tr}_{\mathcal{A}}(e^{\mp F(\Grad^{\slashed S};  \nabla^\mathcal{A})})
    \in   \Omega_{\mathrm{clo}}^{2\Z +\mathrm{type}(\mathcal{A})+1}(X; \mathrm{Ori}(\mathcal{A})\otimes_\R \C),  \\
     \mathrm{CS}_{\lself / \lskew}( \Grad^{\slashed S, I};  \nabla^\mathcal{A}) &:= \int_{I}\Ch_{\lself / \lskew}\left(\Grad^{\slashed S, I}; \mathrm{pr}_X^*  \nabla^\mathcal{A}\right)  \in   \Omega^{2\Z +\mathrm{type}(\mathcal{A})}(X; \mathrm{Ori}(\mathcal{A})\otimes_\R \C). 
\end{align*}

Now we move on to the constructions corresponding to Subsections \ref{subsec_Ph_form} and \ref{subsec_Ph_form_-}. 
Recall the endomorphism $\mathcal{R}_\C$ on $\C$-valued differential forms defined in \eqref{eq_def_R_complex}. 
Given $\nabla^{\mathcal{A}}$ and a self-adjoint (=skew-adjoint) $\mathcal{A}$-connection $ \nabla^{\slashed S}$, for an invertible section $h \in C^\infty(X; \Self_{\mathcal{A}}^*(\slashed S))$, 
we consider the manifold $(0, \infty)_t \times X$ with the self-adjoint superconnection $\mathrm{pr}_X^*\nabla^{\slashed S} + th$, and
define its {\it Chern character form} by
\begin{multline}\label{eq_def_Ch_self}
    \mathrm{Ch}_\lself(h; \nabla^{\mathcal{A}}, \nabla^{\slashed S}) :=   -\pi^{-1/2}\mathcal{R}_\C  \circ\int_{(0, \infty)} \mathrm{Ch}_\lself(\mathrm{pr}_X^*\nabla^{\slashed S} + th; \mathrm{pr}_X^*\nabla^{\mathcal{A}}) \\
     \in \Omega^{2\Z + \type(\mathcal{A})}(X; \mathrm{Ori}(\mathcal{A})). 
\end{multline}
Note that $\Ori(\mathcal{A})$ is a {\it real} line bundle. 
Thus it can be nontrivial, but can be checked directly, that the formula \eqref{eq_def_Ch_self} defines an element in $\Omega^{2\Z + \type(\mathcal{A})}(X; \mathrm{Ori}(\mathcal{A}))$. 

Similarly for $m \in C^\infty(X; \Skew_{\mathcal{A}}^*(\slashed S))$, we define
\begin{multline}\label{eq_def_Ch_skew}
    \mathrm{Ch}_\lskew(m; \nabla^{\mathcal{A}}, \nabla^{\slashed S}) :=   -\pi^{-1/2}(-\sqrt{-1})^{\deg}  \mathcal{R}_\C  \circ\int_{(0, \infty)} \mathrm{Ch}_\lskew(\mathrm{pr}_X^*\nabla^{\slashed S} + tm; \mathrm{pr}_X^*\nabla^{\mathcal{A}}) \\
    \in \Omega^{2\Z + \type(\mathcal{A})}(X; \mathrm{Ori}(\mathcal{A})). 
\end{multline}
Here $(-\sqrt{-1})^{\deg}$ is an operator on $\C$-valued differential forms which multiplies $d$-forms by $(-\sqrt{-1})^{d}$. 
For a homotopy $h_I \in C^\infty(I \times X; \mathrm{pr}_X^*\Self_{\mathcal{A}}^*(\slashed S))$, we define the {\it Chern-Simons form} by
\begin{multline}\label{eq_def_CS_self_complex}
    \mathrm{CS}_\lself(h_I; \nabla^{\mathcal{A}}, \nabla^{\slashed S}) := \int_{I}\mathrm{Ch}_\lself\left(h_I; \mathrm{pr}_X^*\nabla^{\mathcal{A}}, \mathrm{pr}_X^*\nabla^{\slashed S} \right) \\
    \in \Omega^{2\Z + \type(\mathcal{A})-1}(X; \mathrm{Ori}(\mathcal{A}) ).  
\end{multline}
Similarly for $m_I \in C^\infty(I \times X; \mathrm{pr}_X^*\Skew_{\mathcal{A}}^*(\slashed S))$, we define 
\begin{multline}\label{eq_def_CS_skew_complex}
    \mathrm{CS}_\lskew(m_I; \nabla^{\mathcal{A}}, \nabla^{\slashed S}) := \int_{I}\mathrm{Ch}_\lskew\left(m_I; \mathrm{pr}_X^*\nabla^{\mathcal{A}}, \mathrm{pr}_X^*\nabla^{\slashed S} \right) \\
    \in \Omega^{2\Z + \type(\mathcal{A})-1}(X; \mathrm{Ori}(\mathcal{A})).  
\end{multline}
As a special case of the above constructions applied to trivial bundles $\underline{A}$ and $\underline{S}$, we get the $\C$-linear version of Subsection \ref{subsec_superconn_triv} and Subsubsection \ref{subsubsec_Ph_triv_-}. 
We use the corresponding notations such as $\Ch_{\lself}(h)$ for $h \in \Skew^*_A(S)$. 

We can show that the analogous properties of Chern character forms $\Ch_{\lself}(h)$ corresponding to those listed in Subsection \ref{subsubsec_properties_Ph}. 
The invariance under tensoring negligible modules (Lemma \ref{lem_Ph_negligible_triv}) can be shown in the same way. 
The compatibility with the suspension (Proposition \ref{prop_int_Ph_m}) is also basically the same, but we need to take care of the multiplications by $(\sqrt{-1})^\bullet$ appearing in the definition \eqref{eq_def_R_complex} of $\mathcal{R}_\C$.  
The statement becomes the following. 
\begin{prop}
For any element $h \in C^\infty(X, \Self_{\Sigma A}^\dagger (S))$, we define $\widetilde{h} \in C^\infty(I \times X, \Self_A^\dagger (S))$ by the same formula as \eqref{eq_susp_gradation}. 
We have
  \begin{align*}
         \mathrm{Ch}_\lself(h) = \int_I \mathrm{Ch}_\lself(\widetilde{h})
     \end{align*}
under the identification of orientation bundles (cf.\ \eqref{eq_complex_volume})
\begin{align*}
\Ori(A) &\simeq \Ori(\Sigma A), &
u &\mapsto 
\left\{
\begin{array}{ll}
u\widehat{\otimes}\beta, & (\type(A) = 0) \\
\sqrt{-1} u\widehat{\otimes}\beta. & (\type(A) = 1)
\end{array}
\right.
\end{align*}
\end{prop}

To show that the Chern character forms represent the correct cohomology classes, we need the result corresponding to Theorem \ref{thm_Ph=Ph_top} about the Chern character forms for the universal (tautological) gradation $h_{\mathrm{univ}} \in C^\infty(\Self^*_A(S), \Self^*_A(S))$. 
The topological Chern character homomorphism
\eqref{eq_def_Ch_top} can be regarded as an element
\begin{align*}
    \mathrm{Ch}_{\mathrm{top}} \in  {H}^{2\Z + {\type}(A)}\left( K_{\type(A)}, \{*\}; \R\right). 
\end{align*}
We claim that the equality
\begin{align}
   \mathrm{Rham}\left( \mathrm{Ch}_\lself(h_{\mathrm{univ}})\right) = \iota_{S, u}^*  \mathrm{Ch}_{\mathrm{top}}
\end{align}
holds in
$H^{2\Z + {\type}(A)}\left( \Self_{A}^*(S), \{\beta\}; \R\right)$.
Indeed, the proof of Theorem \ref{thm_Ph=Ph_top} can be modified to the $K$-theory. 
In that proof, we used the fact that the degree-zero part of the topological Pontryagin character is realized as the Chern-Weil construction on the real Grassmannians, and used the diffeomorphism $\Self^\dagger_A(S_n) \simeq \mathrm{Gr}(\R^{2n})$ in \eqref{eq_Gr_vs_Skew} to compare $\Ph_{\lself}(h_{\mathrm{univ}})$ with $\mathcal{R}\circ \mathrm{Tr}(e^{-\nabla_{\mathrm{Gr}}^2})$. 
The complex analogue of this argument works since the topological Chern character is realized as the Chern-Weil construction on the complex Grassmannians $\mathrm{Gr}(\C^{2n})$, and we have a diffeomorphism among models of the $K$-theory spectrum corresponding to that detailed in the study of the $KO$-spectrum in Subsubsection \ref{subsubsec_KO_spectrum}. 
These results imply the following one, corresponding to Corollary \ref{cor_Ph=Ph_top}. 

\begin{cor}\label{cor_Ch=Ch_top}
Let $A$ be a simple central graded $*$-algebra over $\C$ and $(X, Y)$ be an object of $\mathrm{MfdPair}_f$. 
Represent classes of $K_+^{A}(X, Y)$ by triples $(S, h_0, h_1)$ with $h_i \in C^\infty(X; \Self_A^*(S))$ as in Remark \ref{rem_Karoubi_K}. Then the topological Chern character homomorphism
\begin{align*}
    \mathrm{Ch}_{\mathrm{top}} \colon K_+^{A}(X, Y) \to H^{2\Z + {\type}(A) }(X, Y; \Ori(A) )
\end{align*}
is given by 
\begin{align*}
    [S, h_0, h_1] \mapsto \mathrm{Rham}\left(\mathrm{Ch}_\lself(h_1) - \mathrm{Ch}_\lself(h_0)  \right). 
\end{align*}
\end{cor}

Using this result on untwisted Chern character forms, we also get the corresponding statement about the twisted case. 
Recall that, for a CW-pair $(X, Y)$, we have the twisted Chern character homomorphism,
\begin{align*}
    \mathrm{Ch}_{\mathrm{top}} \colon K^{\mathcal{A}}_+(X, Y) \to H^{2\Z +{\type}(\mathcal{A})}(X, Y; \mathrm{Ori}(\mathcal{A})). 
\end{align*}
We then have the following result corresponding to Corollary \ref{cor_Ph=Ph_top_twisted}.

\begin{cor}\label{cor_Ch=Ch_top_twisted}
let $(X, Y)$ be an object of $\mathrm{MfdPair}_f$, and $\mathcal{A}$ and $\slashed S$ over $\C$ as above.  
Suppose we have two smooth sections $h_0, h_1 \in C^\infty( X ; \mathrm{Self}_{\mathcal{A}}^*(\slashed S))$ such that $h_0|_Y = h_1|_Y$. 
Then we have
\begin{align} 
    \mathrm{Ch}_{\mathrm{top}}([\slashed S, h_0, h_1]) = \mathrm{Rham}\left( \mathrm{Ch}_\lself(h_1; \nabla^{\mathcal{A}}, \nabla^{\slashed S}) - \mathrm{Ch}_\lself(h_0; \nabla^{\mathcal{A}}, \nabla^{\slashed S})\right). 
\end{align}
\end{cor}

Now we turn to the discussion corresponding to Subsubsection \ref{subsubsec_Ph_self_skew}, relating $\Ch_\lself$ and $\Ch_\lskew$. 
In the $\C$-linear setting here, the argument is simpler. 
Recall that we have a bijection $\Skew^*_A(S) \simeq \Self_A^*(S)$ by $m \mapsto \sqrt{-1}m$ (Remarks \ref{rem_complex_self_skew} and \ref{rem_Karoubi_K_self_skew}). Let us decompose a skew-adjoint $\mathcal{A}$-superconnection $\Grad^{\slashed S}$ compatible with $\nabla^{\mathcal{A}}$ as \eqref{eq_superconn_local}, 
\begin{align*}
    \Grad^{\slashed S} = \nabla + \sum_j \omega_j \widehat{\otimes} \xi_j,
\end{align*}
where $\omega_j \in \Omega^j(X)$ and $\xi_j \in C^\infty(X; \mathrm{End}_{\mathcal{A}}(\slashed S))$ with $|\omega_j| + |\xi_j| = 1 \in \Z_2$. 
Then consider the following operator on $\Omega^*(X; \slashed S)$
\begin{align}\label{eq_def_psi_beta_superconn_complex}
    \psi_{\sqrt{-1}} (\Grad^{\slashed S}): = \nabla + \sum_j (-\sqrt{-1})^{j-1} \omega_j \widehat{\otimes}\xi_j .  
\end{align}
We easily see that the expression \eqref{eq_def_psi_beta_superconn_complex} is independent of the decomposition and defines a self-adjoint $\mathcal{A}$-superconnection compatible with $\nabla^{\mathcal{A}}$. 
We have the following result corresponding to Proposition \ref{prop_skew_self_characteristic_form}, whose proof is an easy computation checking the case $f(z) = z^n$ and so omitted.

\begin{prop}\label{prop_skew_self_characteristic_form_complex}
For a skew-adjoint $\mathcal{A}$-superconnection $\Grad^{\slashed S}$ compatible with $\nabla^{\mathcal{A}}$ and for any $f \in \C[[z]]$, we have the equality 
\begin{align}
   f(F(\Grad^{\slashed S};  \nabla^\mathcal{A}))
    =(-\sqrt{-1})^{\deg} \circ f(-F(\psi_{\sqrt{-1}}(\Grad^{\slashed S});  \nabla^\mathcal{A}))
\end{align}
in $\Omega^{2\Z + \type(\mathcal{A})+1}(X; \Ori(A)\otimes_\R \C)$. 
Here $(-\sqrt{-1})^{\deg}$ is the operator on differential forms which multiplies $d$-forms by $(-\sqrt{-1})^{d}$. 
\end{prop}

Thus we conclude:

\begin{cor}\label{cor_Ch_mass_gradation}
For any element $m \in C^\infty(X; \Skew^*_{\mathcal{A}}(\slashed S))$, we have 
\begin{align}
    \mathrm{Ch}_\lskew(m; \nabla^{\mathcal{A}}, \nabla^{\slashed S})
    =\mathrm{Ch}_\lself(\sqrt{-1}m; \nabla^{\mathcal{A}}, \nabla^{\slashed S}). 
\end{align}
For any element $m_I \in C^\infty(I \times X ; \mathrm{pr}_X^*\mathrm{Skew}_{\mathcal{A}}^*(\slashed S))$, we have
\begin{align}
    \mathrm{CS}_\lskew(m_I; \nabla^{\mathcal{A}}, \nabla^{\slashed S})
    =\mathrm{CS}_\lself(\sqrt{-1}(m_I); \nabla^{\mathcal{A}}, \nabla^{\slashed S}). 
\end{align}
\end{cor}

\subsection{The definitions of differential \texorpdfstring{$K$}{K}-theories \texorpdfstring{$\widehat{K}_+$}{K+} and \texorpdfstring{$\widehat{K}_-$}{K-}}
Now it should be obvious to the reader that the definitions of the differential extensions of $K_+$ and $K_-$ are given by just replacing the coefficient from $\R$ to $\C$ in the definitions of $\widehat{KO}_+$ and $\widehat{KO}_-$, and by using the $\C$-linear version of the characteristic forms defined in the last subsection.  

\begin{defn}[{$\widehat{K}_+^A(X, Y)$ and $\widehat{K}_-^A(X, Y)$}]\label{def_hat_K}
Let $A$ be a nondegenerate simple central graded $*$-algebra over $\C$ and let $(X, Y)$ be an object in $\mathrm{MfdPair}_f$. 
Replacing the coefficient $\R$ by $\C$ and the form degree $4\Z$ by $2\Z$ and using characteristic forms in Subsection \ref{subsec_complex_superconn}, we define $\widehat{K}_+^A(X, Y)$ and $\widehat{K}_-^A(X, Y)$, respectively. 
\end{defn}
Thus, elements of the group $\widehat{K}_+^A(X, Y)$ are of the form $[S, h_0, h_1, \eta]$, where $h_i \in C^\infty(X, \Self_A^*(S))$ are such that $h_0|_Y = h_1|_Y$ and 
\[
\eta \in \Omega^{2\Z + \type(A)-1}(X, Y; \Ori(A))/\mathrm{Im}(d).
\]
We have $[S, h_0, h_1, \CS_\lself(h_I)] = 0$ when $h_I$ is a homotopy relative to $Y$ between $h_0$ and $h_1$. 
Similarly, elements of the group $\widehat{K}_-^A(X, Y)$ are of the form $[S, m_0, m_1, \eta]$, where $m_i \in C^\infty(X, \Skew_A^*(S))$ are such that $m_0|_Y = m_1|_Y$ and 
\[
\eta \in \Omega^{2\Z + \type(A)-1}(X, Y; \Ori(A))/\mathrm{Im}(d).
\]
We have $[S, m_0, m_1, \CS_\lskew(m_I)] = 0$ when $m_I$ is a homotopy relative to $Y$ between $m_0$ and $m_1$. 

For twisted groups, we define the (same) categories $\Tw^2_{\widehat{K}_+} = \Tw^2_{\widehat{K}_-}$ by the $\C$-linear analogue of Definition \ref{def_twist_cat_hat_KO_+}. 
We define the following. 

\begin{defn}[{$\widehat{K}_+^\mathcal{A}(X, Y)$ and $\widehat{K}_-^{\mathcal{A}}(X, Y)$}]\label{def_hat_K_twisted}
Let $(X, Y, \mathcal{A}) \in \Tw^2_{\widehat{K}_+} = \Tw^2_{\widehat{K}_-}$. 
By the same replacement as in Definition \ref{def_hat_K} of Definitions \ref{def_twisted_KO_hat_m} and \ref{def_twisted_KO_hat_-}, we define $\widehat{K}_+^\mathcal{A}(X, Y)$ and $\widehat{K}_-^{\mathcal{A}}(X, Y)$, respectively. 
\end{defn}

\begin{defn}[{Structure homomorphisms for $\widehat{K}_+$ and $\widehat{K}_-$}]\label{def_str_hom_K}
We define the structure homomorphisms $R$, $I$, $a$ for $\widehat{K}^A_+$, $\widehat{K}^{\mathcal{A}}_+$, $\widehat{K}^A_-$ and $\widehat{K}^{\mathcal{A}}_-$ by the same replacement as in Definition \ref{def_hat_K} of Definitions \ref{def_str_hom_untwisted}, \ref{def_str_hom}, \ref{def_str_hom_untwisted_-} and \ref{def_str_hom_twisted_-}, respectively. 
\end{defn}

For example in the case $\widehat{K}^A_+$ we define
\begin{align*}
        R &\colon \widehat{K}_+^{A}(X, Y) \to \Omega_{\mathrm{clo}}^{2\Z + \type(A)}(X, Y; \mathrm{Ori}(A)) \\
       &[S, h_0, h_1, \eta] \mapsto \mathrm{Ch}_\lself(h_1) - \mathrm{Ch}_\lself(h_0) + d\eta.\\
        I &\colon \widehat{K}_+^{A}(X, Y) \to K_+^{A}(X, Y) \\
       & [S, h_0, h_1, \eta] \mapsto [S, h_0, h_1]. \\
        a &\colon \Omega^{2\Z + \type(A) - 1}(X, Y; \Ori(A) ) / \mathrm{Im}(d) \to \widehat{K}_+^{A}(X, Y)  \\
       & \eta \mapsto [0, 0, 0, \eta]. 
\end{align*}

We have natural isomorphisms for both untwisted and twisted cases corresponding to Propositions \ref{prop_hat_KO_skew=self} and \ref{prop_hat_KO_skew=self_twisted} 
\begin{align}\label{eq_isom_hat_K_self_skew}
    \widehat{K}_- \simeq \widehat{K}_+, 
\end{align}
which are compatible with the structure homomorphisms. 
The isomorphisms \eqref{eq_isom_hat_K_self_skew} are given by $[S, m_0, m_1, \eta] \mapsto [S, \sqrt{-1}m_0, \sqrt{-1}m_1, \eta]$ in the untwisted case and $[\slashed S, \nabla^\mathcal{A}, \nabla^{\slashed S}, m_0, m_1, \eta] \mapsto [\slashed S, \nabla^\mathcal{A}, \nabla^{\slashed S}, \sqrt{-1}m_0, \sqrt{-1}m_1, \eta]$ in the twisted case. 
The well-definedness follows from Corollary \ref{cor_Ch_mass_gradation}. 

Moreover, by the cohomological properties of the Chern character forms $\Ch_\lself$ in Corollaries \ref{cor_Ch=Ch_top} and \ref{cor_Ch=Ch_top_twisted}, we can check the axioms of differential cohomology theories as in Theorems \ref{thm_axiom_diffcoh_untwisted_KO} and  \ref{thm_axiom_diffcoh_KO}. 
Thus we conclude the following. 
\begin{thm}\label{thm_hat_K}
The quadruples $\left(\widehat{K}^A_+, R, I, a\right)$ and $\left(\widehat{K}^A_-, R, I, a\right)$ are differential extensions of the $K$-theories $K^A_+ \simeq K^A_- \simeq KO^{\type(A)}$ on $\mathrm{MfdPair}_f$. 

The quadruples $\left(\widehat{K}_+, R, I, a\right)$ and $\left(\widehat{K}_-, R, I, a\right)$ are differential extensions of the twisted $K$-theories ${K}_+ \simeq K_-$ on $\Tw^2_{\widehat{K}_+} = \Tw^2_{\widehat{K}_-}$. 
\end{thm}

\section{The proofs}

\subsection{The proof of Proposition \ref{prop_int_Ph_m}}\label{subsec_proof_int_Ph_m}
Here we prove Proposition \ref{prop_int_Ph_m}. 
Let $S$ be a $\Sigma^{0, 1}A$-module. For $h \in C^\infty(X, \Self_{\Sigma^{0, 1} A}^\dagger (S))$, we have the corresponding $\widetilde{h} \in C^\infty(I \times X, \Self^\dagger_{A}(S))$. 
Choose a volume element $u$ of $A$ to trivialize $\Ori(A)$. 

By Definition \ref{def_Ph_m_triv}, using $h^2  =1$ and $\widetilde{h}^2 = 1$ we have
\begin{align*}
    \Ph_\lself(h) &= \pi^{-1/2}\mathcal{R}\circ \int_{(0, \infty)}dt e^{-t^2}\mathrm{Tr}_{u \widehat{\otimes} \beta}(h\exp(-tdh)), \\
    \Ph_\lself(\widetilde{h}) &=\pi^{-1/2}\mathcal{R}\circ \int_{(0, \infty)}dt e^{-t^2}\mathrm{Tr}_{u}(\widetilde{h}\exp(-td\widetilde{h})).
\end{align*}
We compute (note that we are working in the algebra $\Omega^*(I \times X, \End_A(S))$ with the multiplication \eqref{eq_multi_form}),
\begin{align*}
    &\exp(-td\widetilde{h}) \\
    &= \exp(-t(d\theta \cdot (\del_\theta \widetilde{h}) +   (\sin\pi \theta) dh)) \\
    &= \sum_{n = 0}^\infty \frac{(-t)^n}{n!}(d\theta \cdot (\del_\theta \widetilde{h}) +  (\sin\pi \theta) dh)^n \\
    &=d\theta \cdot (\del_\theta \widetilde{h}) \left(\sum_{n = 0}^\infty \frac{(-t)^{n+1}}{n!} \cdot (\sin^{n}\pi\theta) (dh)^{n}\right) 
    + \exp(-t (\sin\pi \theta) dh).
\end{align*}
Here the last equality used the fact that $dh$ commutes with $d\theta \cdot (\del_\theta \widetilde{h})$, which can be checked easily by $h \in C^\infty(X, \Self_{\Sigma^{0, 1} A}^\dagger (S))$ and the formula for $\widetilde{h}$. 
We also note
\begin{align*}
    \widetilde{h}(\del_\theta \widetilde{h}) = \pi(\beta \cos \pi\theta + h \sin \pi\theta)(-\beta \sin\pi \theta + h \cos\pi \theta) =\pi \beta h. 
\end{align*}
Thus we have (note that $\widetilde{h} d\theta = - d \theta \widetilde{h}$)
\begin{multline}\label{eq_proof_susp_3}
\int_I\Ph_\lself(\widetilde{h}) \\
=\pi^{1/2}  \sum_{n = 0}^\infty \left(\frac{1}{n!} \int_0^\infty (-t)^{n+1}e^{-t^2} dt  \int_I \sin^n (\pi \theta) \cdot \mathcal{R} \left(d\theta\mathrm{Tr}_{u}(-\beta {h}(d{h})^n)\right)\right). 
\end{multline}

Now assume $\type(A)$ is even. 
In this case $\Ph_\lself(\widetilde{h})$ is an even form and $\Ph_\lself(h)$ is an odd form, so that \eqref{eq_proof_susp_3} becomes
\begin{align}\label{eq_proof_susp_4}
     \pi^{1/2} 
     \sum_{l = 0}^\infty \left(\frac{1}{(2l+1)!} \int_0^\infty t^{2l+2}e^{-t^2} dt  \int_I  \sin^{2l+1} (\pi \theta)\cdot \mathcal{R}\left(d\theta \mathrm{Tr}_{u}(-\beta h(dh)^{2l+1})\right)\right), 
\end{align}
and $\Ph_\lself(h)$ is expanded as
\begin{align}\label{eq_proof_susp_5}
    \Ph_\lself(h) = - \pi^{-1/2} \sum_{l = 0}^\infty \frac{1}{(2l+1)!}\int_0^\infty t^{2l+1}e^{-t^2}dt
    \cdot \mathcal{R}\left(\mathrm{Tr}_{u \widehat{\otimes} \beta}(h (dh)^{2l+1})\right). 
\end{align}
We now check that \eqref{eq_proof_susp_4} and \eqref{eq_proof_susp_5} coincide. 
We have:
\begin{align}\label{eq_proof_susp_6}
    \mathrm{Tr}_{u}(-\beta h(dh)^{2l+1}) = -\mathrm{Tr}_{u \widehat{\otimes} \beta}(h (dh)^{2l+1}),
\end{align}

\begin{align}\label{eq_proof_susp_7}
    \int_0^\infty t^{2l+2}e^{-t^2} dt   \int_I d\theta \sin^{2l+1} (\pi \theta)
    = \frac{(2l+1)!! \sqrt{\pi}}{2^{l+2}} \cdot \frac{(2l)!! \cdot 2}{(2l+1)!!\pi} = \frac{l!}{2\sqrt{\pi}},
\end{align}
\begin{align}\label{eq_proof_susp_8}
    \int_0^\infty t^{2l+1}e^{-t^2}dt = \frac{l!}{2}.
\end{align}
Using the above formulas and \eqref{eq_def_R}, we see that \eqref{eq_proof_susp_4} and \eqref{eq_proof_susp_5} coincide. 
This finishes the proof in the case where $\type(A)$ is even. 
 
Now assume $\type(A)$ is odd. 
In this case $\Ph_\lself(\widetilde{h})$ is an odd form and $\Ph_\lself(h)$ is an even form, so that \eqref{eq_proof_susp_3} becomes
\begin{align}\label{eq_proof_susp_9}
     -\pi^{1/2}  \sum_{l = 0}^\infty \left(\frac{1}{(2l)!} \int_0^\infty t^{2l+1}e^{-t^2} dt  \int_I \sin^{2l}(\pi \theta) \mathcal{R}\left(d\theta  \mathrm{Tr}_{u}(-\beta h(dh)^{2l})\right)\right), 
\end{align}
and $\Ph_\lself(h)$ is expanded as
\begin{align}\label{eq_proof_susp_10}
    \Ph_\lself(h) = \pi^{-1/2} \sum_{l = 0}^\infty \frac{1}{(2l)!}\int_0^\infty t^{2l}e^{-t^2}dt
    \cdot \mathcal{R}\left(\mathrm{Tr}_{u \widehat{\otimes} \beta}(h (dh)^{2l})\right). 
\end{align}
We now check that \eqref{eq_proof_susp_9} and \eqref{eq_proof_susp_10} coincide. 
In this case we have
\begin{align}\label{eq_proof_susp_11}
   -\mathrm{Tr}_{u}(-\beta h(dh)^{2l+1}) = 2\mathrm{Tr}_{u \widehat{\otimes} \beta}(h (dh)^{2l+1}),
\end{align}
\begin{align}\label{eq_proof_susp_12}
    \int_0^\infty t^{2l+1}e^{-t^2} dt  \int_I \sin^{2l}(\pi \theta) d\theta 
    = \frac{l!}{2} \cdot \frac{(2l-1)!!}{(2l)!!}
    = \frac{(2l-1)!! }{2^{l+1}},
\end{align}
\begin{align}\label{eq_proof_susp_13}
    \int_0^\infty t^{2l}e^{-t^2} dt  
    = \frac{(2l-1)!! \sqrt{\pi}}{2^{l+1}}.
\end{align}
Using above formulas and \eqref{eq_def_R}, we see that \eqref{eq_proof_susp_9} and \eqref{eq_proof_susp_10} coincide. 
This finishes the proof in the case where $\type(A)$ is odd and completes the proof of Proposition \ref{prop_int_Ph_m}. 

\subsection{Proof of Theorem \ref{thm_Ph=Ph_top}}\label{subsec_proof_Ph=Ph_top}
Here we prove Theorem \ref{thm_Ph=Ph_top}. 
The first step is to show the following. 
\begin{lem}\label{lem_Ph=Ph_top_20}
Theorem \ref{thm_Ph=Ph_top} holds if $\type(A) = 0$. 
\end{lem}

\begin{proof}
Choose and fix a volume element $u \in A$ to trivialize $\Ori(A)$.
Recall that, in Subsubsection \ref{subsubsec_KO_spectrum}, we explained the model of the $KO$-spectrum realized as a direct limit of $\mathrm{Skew}^\dagger_{A}(-)$, and in the case that $\type(A) = 0$ we gave an explicit homeomorphism \eqref{eq_homeo_Gr_Skew} with the model of $KO_0$ realized by using Grassmannians. 
Also, as explained in Subsection \ref{subsec_Ph_top}, the degree-zero part of the topological Pontryagin character is realized as the Chern-Weil construction on vector bundles. 
For each $n$, set $S_n := (S_+ \otimes \R^n)\oplus (S_- \otimes \R^n)$ in the notation there. Recall the diffeomorphism \eqref{eq_Gr_vs_Skew} 
\begin{align*}
\Self^\dagger_{A}(S_n) = 
     \Self^\dagger_{A}((S_+ \otimes \R^n)\oplus (S_- \otimes \R^n)) &\simeq \mathrm{Gr}(\R^{2n})  \\
     u \otimes a &\mapsto \mathrm{Im}((1-a)/ 2) = \mathrm{Im}(P), 
\end{align*}
where $P := (1-a)/2$. The canonical connection on the tautological real vector bundle $\theta_{2n}$ over $\mathrm{Gr}(\R^{2n})$ is given by
\begin{align*}
    \nabla_{\mathrm{Gr}} = PdP, 
\end{align*}
where the inclusions $\Omega^*(\mathrm{Gr}(\R^{2n}); \theta_{2n}) \subset \Omega^*(\mathrm{Gr}(\R^{2n}); \underline{\R}^{2n})$ are understood. 
We know that the topological Pontryagin character is the direct limit with respect to $n$ of the Pontryagin character form $\mathcal{R}\circ \mathrm{Tr}(e^{\nabla_{\mathrm{Gr}}^2}) - n$ for the connection $\nabla_{\mathrm{Gr}}$. 
Thus it is enough to show that (see \eqref{eq_Gr_vec_bdle})
\begin{align}\label{proof_gr_taut}
   \mathcal{R}\circ \mathrm{Tr}(e^{\nabla_{\mathrm{Gr}}^2}) - n -\mathrm{Ph}_\lself(h_\mathrm{univ}) \in \mathrm{Im}(d)
\end{align}
on $\Self^\dagger_{A}(S_n)$ for each $n$. 

First we compute $ \mathrm{Tr}(e^{\nabla_{\mathrm{Gr}}^2})$. 
We have
\begin{align*}
    \nabla_{\mathrm{Gr}}^2 = P dP \wedge dP = \frac{1}{4}P da \wedge da. 
\end{align*}
Using the relation $ada + da\cdot a = 0$ and $P ^2 = P$, 
we have
\begin{align}\label{eq_Pf_gr}
    \mathrm{Tr}(e^{\nabla_{\mathrm{Gr}}^2}) 
    &= \sum_{k=0}^\infty\frac{1}{2^{2k+1} k! } \left( \mathrm{Tr}(-a(da)^{2k}) + \mathrm{Tr}((da)^{2k})\right).
\end{align}

Next we compute $\mathrm{Ph}_\lself(h_\mathrm{univ})$. 
Applying Definition \ref{def_Ph_m_triv} to $h_\mathrm{univ} = u \otimes a$ and using $h_{\mathrm{univ}}^2 = 1$,  
we have
\begin{align*}
    \mathrm{Ph}_\lself(h_\mathrm{univ}) = \mathcal{R} \circ \pi^{-1/2} \int_{(0, \infty)} e^{-t^2}dt \wedge \mathrm{Tr}_u \left( (u \otimes a) \cdot e^{-t u \widehat{\otimes} da } \right). 
\end{align*}
We have
\begin{align*}
    e^{-tu \widehat{\otimes} da}
    =\sum_{k=0}^\infty \frac{t^{2k} }{(2k)!} \cdot 1 \widehat{\otimes} (da)^{2k}- \sum_{k=0}^\infty \frac{t^{2k+1} }{(2k+1)!}\cdot u\widehat{\otimes} (da)^{2k+1}, 
\end{align*}
so that, using $\mathrm{Tr}_u(u \otimes B) = -\mathrm{Tr}(B)$ and $\mathrm{Tr}_u(1 \otimes B) =0$ for any $B \in \mathrm{End}(\R^{2n})$, 
\begin{align*}
    \mathrm{Tr}_u\left((u \otimes a) e^{-tu \widehat{\otimes} da}\right)
     = \sum_{k=0}^\infty \frac{t^{2k} }{(2k)!} \cdot \mathrm{Tr} \left( -a(da)^{2k}\right) .
\end{align*}
Using the formula 
\begin{align*}
    \int_0^\infty e^{-t^2} t^{2k}dt = \sqrt{\pi} \cdot \frac{(2k-1)!!}{2^{k+1}}, 
\end{align*}
we get
\begin{align}\label{eq_univ_form_31}
    \mathrm{Ph}_\lself(h_\mathrm{univ}) = \mathcal{R} \circ
     \sum_{k=0}^\infty \frac{1}{2^{2k+1}k!}\mathrm{Tr} \left( -a(da)^{2k}\right). 
\end{align}

Now we can compute the element \eqref{proof_gr_taut}. 
By \eqref{eq_Pf_gr} and \eqref{eq_univ_form_31}, we have
\begin{align*}
    \mathcal{R}\circ \mathrm{Tr}(e^{\nabla_{\mathrm{Gr}}^2}) - n -\mathrm{Ph}_\lself(h_\mathrm{univ}) 
    &= \mathcal{R} \circ \left(-n + \sum_{k=0}^\infty \frac{1}{2^{2k+1}k!}  \mathrm{Tr}((da)^{2k})
    \right)\\
    &= \mathcal{R} \circ \frac{1}{2} \mathrm{Tr}(e^{\frac{1}{4}da \wedge da}-1). 
\end{align*}
This is exact, so we have proved \eqref{proof_gr_taut}. 
This completes the proof of Lemma \ref{lem_Ph=Ph_top_20}. 
\end{proof}

\begin{proof}[Proof of Theorem \ref{thm_Ph=Ph_top}]
By Proposition \ref{prop_int_Ph_m}, we know that $\mathrm{Ph}(m)$ is compatible with the suspension. 
Thus we can reduce to the case $\type(A) = 0$, and the result follows from Lemma \ref{lem_Ph=Ph_top_20}. 
\end{proof}

\subsection{The proof of Proposition \ref{prop_chracteristic_form_mod4}}\label{subsec_proof_mod4}
Here we prove Proposition \ref{prop_chracteristic_form_mod4}. 
We only give the proof for (1),  whose straightforward modification leads to (2). 

It is enough to consider the case that $f(z) = z^n$ for $n \in \Z_{\ge 0}$. 
Moreover, working locally, we may assume that $V = \underline{A}$ for some $A$ and $\slashed S = \underline{S}$ for an $A$-module $S$ with an inner product. 
Furthermore, by Lemma \ref{lem_change_connection} it is enough to consider the case that $ \nabla^\mathcal{A} = d$, since the action \eqref{eq_connection_action} preserves the self-adjointness. 
Fix a volume element $u$ for $A$ to trivialize $\mathrm{Ori}(A)$. 
Then we can decompose the self-adjoint $\mathcal{A}$-superconnection $\Grad^{\slashed S}$ as
\begin{align*}
    \Grad^{\slashed S} = d + \sum_{j \in J} \omega_j \widehat{\otimes}\xi_j , 
\end{align*}
for $\omega_j \in \Omega^*(X)$ and $\xi_j \in \mathrm{End}_{A}(S) $ with $|\omega_j| + |\xi_j| \equiv 1 \pmod 2$. 
We have
\begin{align*}
    F({\Grad^{\slashed S}}; d) = (\Grad^{\slashed S})^2 = \sum_j d\omega_j \widehat{\otimes}\xi_j
    + \sum_{j, k} (\omega_j\widehat{\otimes}\xi_j) \cdot (\omega_k\widehat{\otimes}\xi_k). 
\end{align*}
Thus $(F({\Grad^{\slashed S}}; d))^n = (\Grad^{\slashed S})^{2n}$ is the sum of the all possible terms of the form
\begin{align}\label{eq_proof_pfmod4_1}
    (\alpha_{j_1} \widehat{\otimes}\xi_{j_1}) \cdot (\alpha_{j_2} \widehat{\otimes}\xi_{j_2})\cdot \cdots \cdot (\alpha_{j_l} \widehat{\otimes}\xi_{j_l}), 
\end{align}
where $j_t \in J$ and $\alpha_{j_t}$ is either of $d\omega_{j_t}$ or $\omega_{j_t}$ for each $t$, and if the number of $t$ such that $\alpha_{j_t} = d\omega_{j_t}$ is $m$, we have $l + m =2 n$. 

We show that, for each term \eqref{eq_proof_pfmod4_1} as above, the sum
\begin{multline}\label{eq_proof_pfmod4_sum}
    \mathrm{Tr}_{u}\left( (\alpha_{j_1} \widehat{\otimes}\xi_{j_1}) \cdot (\alpha_{j_2} \widehat{\otimes}\xi_{j_2})\cdot \cdots \cdot (\alpha_{j_l} \widehat{\otimes}\xi_{j_l})\right) \\
    + \mathrm{Tr}_{u}\left( (\alpha_{j_l} \widehat{\otimes}\xi_{j_l}) \cdot (\alpha_{j_{l-1}} \widehat{\otimes}\xi_{j_{l-1}})\cdot \cdots \cdot (\alpha_{j_1} \widehat{\otimes}\xi_{j_1})\right)
\end{multline}
can be nonzero only if we have $\sum_t |\alpha_{j_t}| \equiv \mathrm{type}(A)+1 \pmod 4$. 
This implies the desired result. 

Fix a term \eqref{eq_proof_pfmod4_1}. 
We define $I \subset \{1, \cdots, l\}$ so that $\alpha_{j_t} = d\omega_{j_t}$ for $t \in I$ and $\alpha_{j_t} = \omega_{j_t}$ for $t \notin I$. 
We have $|I| = m$. 
We also define $\sigma \colon \Z \to  \Z_2$ by
\begin{align*}
    \sigma(x) := \begin{cases}
    1 & \mbox{if } x \equiv 1, 2 \pmod 4, \\
    0 & \mbox{if } x \equiv 0, 3 \pmod 4. 
    \end{cases}
\end{align*}
Then the self-adjointness of $\Grad$ (Definition \ref{def_skew_superconn}) implies that
\begin{align}\label{eq_proof_pfmod4_2}
    \xi_j^* = (-1)^{\sigma(|\omega_j|)}\xi_j
\end{align}
for all $j \in J$. 
Also, a straightforward computation shows that
\begin{align}\label{eq_proof_pfmod4_3}
    u^* = (-1)^{\sigma(\mathrm{type}(A))} u. 
\end{align}

Now we compute \eqref{eq_proof_pfmod4_1}. 
We define $A, B, C, D \in \Z_2$ by the following. 
\begin{align}
    (\alpha_{j_1} \widehat{\otimes}\xi_{j_1}) \cdot (\alpha_{j_2} \widehat{\otimes}\xi_{j_2})\cdot \cdots \cdot (\alpha_{j_l} \widehat{\otimes}\xi_{j_l})
    &= (-1)^A \alpha_{j_1} \wedge \alpha_{j_2} \wedge \cdots \wedge \alpha_{j_l} \widehat{\otimes} \xi_{j_1} \xi_{j_2} \cdots \xi_{j_l}, \\
    (\alpha_{j_l} \widehat{\otimes}\xi_{j_l}) \cdot (\alpha_{j_{l-1}} \widehat{\otimes}\xi_{j_{l-1}})\cdot \cdots \cdot (\alpha_{j_1} \widehat{\otimes}\xi_{j_1})
    &= (-1)^B \alpha_{j_l} \wedge \alpha_{j_{l-1}} \wedge \cdots \wedge \alpha_{j_1} \widehat{\otimes} \xi_{j_l} \xi_{j_{l-1}} \cdots \xi_{j_1}, \notag\\
    \alpha_{j_1} \wedge \alpha_{j_2} \wedge \cdots \wedge \alpha_{j_l} 
    &= (-1)^C\alpha_{j_l} \wedge \alpha_{j_{l-1}} \wedge \cdots \wedge \alpha_{j_1}, \notag \\
    (u \xi_{j_1} \xi_{j_2} \cdots \xi_{j_l})^* 
    &= (-1)^D \xi_{j_l} \xi_{j_{l-1}} \cdots \xi_{j_1}u  \notag. 
\end{align}
Then we have
\begin{align*}
    \mathrm{Tr}\left(u\xi_{j_1} \xi_{j_2} \cdots \xi_{j_l} \right) &= \mathrm{Tr}\left((u\xi_{j_1} \xi_{j_2} \cdots \xi_{j_l})^* \right) \\
    &= (-1)^D \mathrm{Tr}\left(  \xi_{j_l} \xi_{j_{l-1}} \cdots \xi_{j_1}u \right) \\
    &= (-1)^D\mathrm{Tr}\left( u \xi_{j_l} \xi_{j_{l-1}} \cdots \xi_{j_1}\right), 
\end{align*}
so that
\begin{align*}
    \mathrm{Tr}_u\left(\xi_{j_1} \xi_{j_2} \cdots \xi_{j_l} \right) =(-1)^D\mathrm{Tr}_{u} \left(\xi_{j_l} \xi_{j_{l-1}} \cdots \xi_{j_1} \right).
\end{align*}
Using this, we get
\begin{align}\label{eq_proof_pfmod4_8}
    &\mathrm{Tr}_u\left( (\alpha_{j_1} \widehat{\otimes}\xi_{j_1}) \cdot (\alpha_{j_2} \widehat{\otimes}\xi_{j_2})\cdot \cdots \cdot (\alpha_{j_l} \widehat{\otimes}\xi_{j_l})\right) \\
    &= (-1)^A \mathrm{Tr}_{u}\left(\xi_{j_1} \xi_{j_2} \cdots \xi_{j_l} \right) \alpha_{j_1} \wedge \alpha_{j_2} \wedge \cdots \wedge \alpha_{j_l} \notag \\
    &=(-1)^{A +C + D}\mathrm{Tr}_{u} \left(\xi_{j_l} \xi_{j_{l-1}} \cdots \xi_{j_1} \right)\alpha_{j_l} \wedge \alpha_{j_{l-1}} \wedge \cdots \wedge \alpha_{j_1} \notag \\
    &= (-1)^{A + B+C+D} \mathrm{Tr}_{u}\left( (\alpha_{j_l} \widehat{\otimes}\xi_{j_l}) \cdot (\alpha_{j_{l-1}} \widehat{\otimes}\xi_{j_{l-1}})\cdot \cdots \cdot (\alpha_{j_1} \widehat{\otimes}\xi_{j_1})\right). \notag
\end{align}
Now we compute $A+B+C+D$.  
We have
\begin{align}
    A &= \sum_{s < t} |\alpha_{j_s}| \cdot |\xi_{j_t}|, \\ 
    B &= \sum_{t < s}|\alpha_{j_s}| \cdot |\xi_{j_t}|, \notag \\
    C &= \sum_{s < t} |\alpha_{j_s}| \cdot|\alpha_{j_t}|, \notag \\
    D &= \sum_t \sigma(|\omega_{j_t}|) + \sigma(\mathrm{type}(A)), \notag
\end{align}
where the last equality used \eqref{eq_proof_pfmod4_2} and \eqref{eq_proof_pfmod4_3}. 
Recall that $|\xi_{j_t}| \equiv |\alpha_{j_t}| \pmod 2$ if $t \in I$ and $|\xi_{j_t}| \equiv |\alpha_{j_t}| -1 \pmod 2$ if $t \notin I$. 
We get
\begin{align}\label{eq_proof_pfmod4_4}
    A + B &= \sum_{t \neq s}|\alpha_{j_s}| \cdot |\xi_{j_t}| \\
    &=  \sum_{t \neq s}|\alpha_{j_s}| \cdot|\alpha_{j_t}| - (l-m) \sum_{s \in I}|\alpha_{j_s}| - (l - m - 1)\sum_{s \in J \setminus I}|\alpha_{j_s}| \notag \\
    &= \sum_{s \in J \setminus I}|\alpha_{j_s}|, \notag
\end{align}
where the last equality follows from the fact that the first term in the second line is even, and that $l-m$ is even since $l + m = 2n$. 
Using $\sigma(x - 1 ) = \sigma(x) + x$ and the fact that $|\omega_{j_t}| = |\alpha_{j_t} | - 1$ if $t \in I$ and $|\omega_{j_t}| = |\alpha_{j_t} |$ if $t \notin I$, we have
\begin{align}
    D &= \sum_{t \in J} \sigma(|\alpha_{j_t}|) + \sum_{t \in I}|\alpha_{j_t}|  +  \sigma(\mathrm{type}(A)) . 
\end{align}
Moreover, using $\sigma(x )+\sigma( y)+xy= \sigma(x+ y)  $, we get
\begin{align}\label{eq_proof_pfmod4_5}
    C + D &= \sigma \left(\sum_{t \in J} |\alpha_{j_t}|\right) +  \sum_{t \in I}|\alpha_{j_t}|  + \sigma(\mathrm{type}(A))  .
\end{align}
By \eqref{eq_proof_pfmod4_4} and \eqref{eq_proof_pfmod4_5}, we conclude
\begin{align}\label{eq_proof_pfmod4_6}
    A + B + C + D = \sigma \left(\sum_{t \in J} |\alpha_{j_t}|\right) + \sum_{t \in J}|\alpha_{j_t}|  + \sigma(\mathrm{type}(A)) . 
\end{align}

Recall that we already know that the map \eqref{eq_def_Gamma_tr} preserves the $\Z_2$-grading if $\mathrm{type}(A)$ is odd, and reverses the $\Z_2$-grading if $\mathrm{type}(A)$ is even. 
In particular, by Lemma \ref{lem_curv_even}, we know that \eqref{eq_proof_pfmod4_sum} can be nonzero only when 
\begin{align}\label{eq_proof_pfmod4_7}
    \sum_{t \in J} |\alpha_{j_t}| \equiv \mathrm{type}(A) + 1 \pmod 2. 
\end{align}
In this case \eqref{eq_proof_pfmod4_6} becomes
\begin{align*}
    A + B + C + D &= \sigma \left(\sum_{t \in J} |\alpha_{j_t}|\right)  +\mathrm{type}(A) +1 + \sigma(\mathrm{type}(A)) \\
    &= \sigma \left(\sum_{t \in J} |\alpha_{j_t}|\right) + \sigma(\mathrm{type}(A) + 1)
\end{align*}
where the last equality used $\sigma(x+1) = \sigma(x) +x + 1$. 
Using \eqref{eq_proof_pfmod4_8}, the sum \eqref{eq_proof_pfmod4_sum} can be nonzero only when both \eqref{eq_proof_pfmod4_7} and 
\begin{align*}
    \sigma \left(\sum_t |\alpha_{j_t}|\right) = \sigma(\mathrm{type}(A) + 1)
\end{align*}
are satisfied. 
But $\sigma(x) = \sigma(y)$ and $x \equiv y \pmod 2$ imply $x \equiv y \pmod 4$, so these conditions imply
\begin{align*}
    \sum_t |\alpha_{j_t}| \equiv \mathrm{type}(A)+1 \pmod 4. 
\end{align*}
Thus we get the desired result.

\section*{Acknowledgment}
The authors are grateful to Mikio Furuta, Yuji Tachikawa and Kazuya Yonekura for helpful discussion and comments. 
KG is supported by 
JSPS KAKENHI Grant Numbers 20K03606 and JP17H06461.
MY is supported by JSPS KAKENHI Grant Number 20K14307 and JST CREST program JPMJCR18T6.

\bibliographystyle{ytamsalpha}
\bibliography{genusQFT}

\end{document}